\documentclass[11pt,reqno]{amsart}

\usepackage{color}
\usepackage{amsmath}
\usepackage{amssymb}
\usepackage{mathrsfs}
\usepackage[utf8]{inputenc}
\usepackage[english]{babel}
\usepackage{amsthm}
\usepackage{tikz}

\usepackage{graphicx}
\graphicspath{{./images/}}

\usepackage[toc,title,page]{appendix}

\usepackage{epsfig}

\usepackage[scaled=1]{helvet}
\usepackage{multido}
\usepackage{multirow}
\usepackage{blindtext}
\usepackage{color}
\usepackage{lipsum}
\usepackage{mathtools}
\usepackage[vcentermath]{youngtab}

\usepackage[all,ps,dvips,graph]{xy}
\usepackage[misc,geometry]{ifsym}

\usepackage[left=3cm,top=2cm,right=3cm,bottom=2cm]{geometry}
\usepackage{graphicx}
\graphicspath{{./images/}}
\numberwithin{equation}{section}

\let\OLDthebibliography\thebibliography
\renewcommand\thebibliography[1]{
  \OLDthebibliography{#1}
  \setlength{\parskip}{0pt}
  \setlength{\itemsep}{0pt plus 0.3ex}
}

\newtheorem{theorem}{Theorem}
\newtheorem{proposition}[theorem]{Proposition}
\newtheorem{corollary}[theorem]{Corollary}
\newtheorem{lemma}[theorem]{Lemma}
\newtheorem{remark}[theorem]{Remark}
\newtheorem{example}[theorem]{Example}
\newtheorem{definition}[theorem]{Definition}

\numberwithin{theorem}{section}

\newcommand{\alphan}{\boldsymbol{\alpha}}

\newcommand{\BBU}{\mathfrak{U}}

\newcommand{\bbeta}{\boldsymbol{\beta}}
\newcommand{\bgamma}{\boldsymbol{\gamma}}

\newcommand\bi{\boldsymbol{i}}

\newcommand\calA{\mathcal{A}}

\newcommand\calL{\mathcal{L}}

\newcommand\calQ{\mathcal{Q}}

\newcommand\calR{\mathcal{R}}

\newcommand{\kappan}{\boldsymbol{\kappa}}

\newcommand{\ulU}{\underline{U}}

\newcommand{\ulz}{\underline{z}}

\newcommand{\zetan}{\boldsymbol{\zeta}}

\usepackage{graphicx} 
\usepackage{cancel}

\newtheorem{notation}{Notation}

\newcommand{\ZZ}{\mathbb{Z}}

\begin{document}

\title{Framed blob monoids}

\author[J. Juyumaya]{Jes\'us Juyumaya}
\address{IMUV, Universidad de Valpara\'{\i}so\\Gran Breta\~na 1111, 2340000 Valpara\'{\i}so, Chile.}
\email{juyumaya@uv.cl}

\author[D. Lobos]{Diego Lobos}
\address{IMUV, Universidad de Valpara\'{\i}so\\Gran Breta\~na 1111, 2340000 Valpara\'{\i}so, Chile.}
\email{diego.lobosm@uv.cl}

\date{}
\thanks{The second author was partially supported by FONDECYT de Postdoctorado 2024, N°3240046, ANID-Chile
and Concurso Subvención a la Instalación en la Academia 2024, N°85240053, ANID, Chile.}
\setcounter{tocdepth}{1}

\begin{abstract}
We introduce and study blob  and framed blob monoids. In particular, several realizations of these monoids are given. We compute the cardinality of the framed blob monoid and derive some combinatorial formulas involving this cardinality.
\end{abstract}
\maketitle

\tableofcontents

\section{Introduction}
The prototype of a framed structure is the so-called framed braid group, which was introduced independently  by Melvin-Tufillaro\cite{MeTuPhReA1991}  and Ko-Smolinsky\cite{KoSm1992}. In \cite{MeTuPhReA1991} the reason for introducing this group  is motivated in the representation of a template for a dynamical system, and in \cite{KoSm1992} is motivated to prove a type Markov theorem  for 3-manifolds.
Moreover, in \cite{
Juyum-Lamb1,
Juyum-Lamb2},  it was  introduced the $p$-adic framed braid group and a deformation of it which yields an invariant for framed knots. For other literature on  framed structures in the context of knot theory, see \cite{FlJuLaJPAA2018, DiFlJPAA2020, GoJuKoLaMRL2017,JuLaWSP2015} for instance. These works provide the main motivation for constructing framizations of the so-called knot monoids, i.e. monoids involved in the study of knots. One class of such objects are the Brauer-type monoids, which include the symmetric group, the Brauer, Jones, Motzkin and inverse symmetric monoids among other, see \cite{
KuMaCECJM2006}. All these monoids can be realized as submonoids of the so-called
partition monoid\cite[Section 2]{KuMaCECJM2006}, cf. \cite{
HaRaEJC2005, Martin-PartAlg}. The main monoid studied in this work is the {\it blob monoid}, which is here realized as a submonoid of the partition monoid.

The blob algebra is a two-parameters algebra that can be considered an extension of the Temperley–Lieb algebra and was defined independently by Martin-Saleur and tom Dieck, see \cite{MaSaLMP1994} and \cite{DiJRM1994}, respectively. The definition given by Martin-Saleur is motivated by reasons specific to statistical mechanics and uses certain diagrams that we call Martin-Saleur diagrams. Tom Dieck's approach was completely different, his motivation was to define an extension of the Jones polynomial as Jones did but using the blob algebra instead of the Temperley-Lieb algebra, see \cite[Section 11]{DiJRM1994}.
 Furthermore, it is worth noting that recently in \cite{LoPlRyJA2021} a degenerate version of the blob algebra appears as a certain endomorphism algebra coming from the Elias-Williamson  diagrammatic category, see \cite{EW}.

Now, a specialization of the blob algebra yields the monoid algebra of a certain monoid $Bl_{n}$ which we call blob monoid, and is an extension of the Jones monoid  defined in \cite{LaFiCA2006}.
This article begins a systematic study of the blob monoid in terms of its framization and deframization, in the sense of  \cite{AiJuPa2024}, cf. \cite{Aicardi-Arcis-Juyum-2023,  AiArJu2024, LaPoarXiv2023}. Therefore, the main objective here is to carry out a first study of $Bl_{n}$ as well as a framization $Bl_{d,n}$ of it. The definition of $Bl_{d,n}$ is derived by considering the realization of $Bl_n$ by means of Martin-Saleur diagrams and considering that the arcs are provided with beads that move freely on them. This procedure follows the original construction of the framed braid group, see \cite{KoSm1992}, cf. \cite[Section 3]{AiJuPa2024}. From this diagrammatic realization of $Bl_{d,n}$ together the so-called {\it indexing matrices}  we can calculate its cardinality,  as well as the normal forms of its elements. This cardinal is a sum involving certain numbers $\Omega_k^{(n)}$'s, where $k$ is a non-negative integer less than $n+1$. The numbers $\Omega_k^{(n)}$'s are also studied and we find interesting combinatorial formulas, some of which are included in the article.

The organization and main results of the work are presented in the following.
In Section 2, Definition \ref{DefBln}, we introduce the blob monoid $Bl_n$, which is the monoidal face of the so-called blob algebra introduced independently in \cite{MaSaLMP1994, DiJRM1994}. We also show a normal form for the element of $Bl_n$, which is encoded by the so-called indexing matrices; see Subsection 2.3. In Subsection 2.4, $Bl_n$ is realized as the diagrammatic monoid $\mathfrak{Bl}_n$, i.e. the one formed by the Martin-Saleur diagrams, see Theorem \ref{theo-iso-diagBlobMon-algBlobMon}.
The Subsection 2.5 is devoted to the realization of  $Bl_n$ as submonoid of the partition monoid, namely, as the Motzkin blob (Subsection 2.5.2) and the hook blob monoids (Subsection 2.5.3).

In Section 3 we introduced the abacus monoid $\mathfrak{Bl}_{d,n}$ (Definition \ref{def-abacus-blob-mon}) which consists, roughly speaking, in adding beads to the arcs of the elements of $\mathfrak{Bl}_n$. Using a series of technical lemmas and diagrammatic arguments, we obtain Theorem \ref{theo-card-Fblob-mon-first}. This theorem gives the cardinality of $\mathfrak{Bl}_{d,n}$ as a sum involving the numbers $\Omega_k^{(n)}$ mentioned above and is the key in Section 4, as well as in the combinatorial study that gave rise to Section 5.
The goal of Section 4 is the proof of Theorem \ref{theo-iso-AlgFrblob-DiagFrblob}, which says that the monoids $\mathfrak{Bl}_{d,n}$ and $Bl_{d,n}$ are isomorphic. The results of Section 5 relate to several combinatorial formulas derived from the numbers $\Omega_k^{(n)}$'s; in particular, we obtain alternative formulas to compute the numbers $\Omega_k^{(n)}$'s and the  cardinality of $\mathfrak{Bl}_{d,n}$, see Corollary \ref{eq-rel2-omega-vs-theta-numbers} and Theorem \ref{theo-alternative-formula-diag-frBlob}, respectively.

In Section 6 we give another framization $Bl_{d,n}^c$ of $Bl_{n}$, called the connected framed blob monoid, which is not isomorphic to $Bl_{d,n}$, see Definition \ref{def-connect-frm-blob-mon}. This framed monoid arises by adding beads to the hook blob model of $Bl_{n}$,  and the key difference between them is that we have added a new framing generator $z_0$ that commutes with the blob generator, see (\ref{fracon}). Finally, Theorem \ref{theo-card-Fblob-mon-first-conn} and Theorem \ref{theo-alternative-formula-diag-frBlob-conn} give two ways to calculate the cardinality of $Bl_{d,n}^c$.

\section{Preliminaries}
Here we  introduce the blob monoid $Bl_{n}$ and we show how a normal form of its element is encoded using indexing matrices; also,  diagrammatic realizations of $Bl_{n}$ are constructed, see Subsection 2.5.

\subsection{Notation}Throughout the article, if generator subscripts are not specified, it means that they take any value for which the respective generator is defined. We denote by $[a,b]$ the set of integers between  $a$ and $b$. From now on,  $d$ and $n$ denote two positive integers.

\subsection{Blob monoid}
\begin{definition}\label{DefBln}
The blob  monoid $Bl_n$ is the monoid presented by generators  $u_0,u_1, \ldots,  u_{n-1}$ satisfying  the following relations:
\begin{align}
u_i^2=u_i,\quad  & u_iu_j=u_ju_i, \quad \text{for $|i-j|>1$,}\label{Bl1} \\
u_1u_0u_1=u_1, &\quad u_iu_ju_i=u_i,\quad \text{for $|i-j|>1$ and $i,j\not=0$.}\label{Bl2}
\end{align}
 \end{definition}
 \begin{remark}\label{JextBl}\rm
The blob monoid is an extension of the Jones monoid, see  \cite{LaFiCA2006}. To be precise, the Jones monoid is the submonoid generated by the elements $u_i$, with $i\not=0$, which satisfies all the defining relations of $Bl_n$ that do not involve the element $u_0$.
 \end{remark}
\begin{proposition}\label{NormalBlob}
For $j\leq i$, define $U_{i,j} :=u_i u_{i-1}\cdots u_j$. Every element of $Bl_n$ can be written in normal form, that is, in the form
$$
U_{i_1,j_1}\cdots U_{i_p,j_p},
$$
where
\begin{align}\label{CondTij}
\left\{\begin{array}{c}
0\leq i_1<\cdots <i_p\leq n-1,\\
0=j_1=\cdots =j_{r-1}<j_r<\cdots < j_p \leq n-1,\\
j_s \leq i_s.
\end{array}\right.
\end{align}
\end{proposition}
\begin{proof}Analogous to the proof of \cite[Lemma 4.1.2, Aside 4.1.4]{JoIM1983}. See \cite{DiJRM1994, AlHar-Gon-Pl}  for details.
\end{proof}

In order to provide a useful way to codify the normal form, we introduce the following matrix encoding that will be used frequently during  the paper.

\subsection{Indexing matrices}\label{def-indexing-matrix}
  Let $A=[a_{ij}]$ be a $2\times m$ matrix, with $m\in [1,n]$. We say that $A$ is an \emph{indexing matrix}, if $A=\Big[\!\!\begin{array}{c}\infty\\\infty\end{array}\!\!\Big]$, or if its entries $a_{ij},$ satisfy the following rules:
  \begin{enumerate}
    \item[(i)] $a_{ij}\in [0,n-1]$,
    \item[(ii)] $a_{1j}<a_{1k},$ whenever $j<k$,
    \item[(iii)] $a_{2j}\leq a_{2k},$ whenever $j<k$,
    \item[(iv)] $a_{2j}=a_{2k}$ and $j<k$ implies that $a_{2j}=a_{2k}=0$,
    \item[(v)] $a_{1j}\geq a_{2j},$ for each $j\in [1,m]$.
  \end{enumerate}
 The matrix $\Big[\!\!\begin{array}{c}\infty\\\infty\end{array}\!\!\Big]$ will be called \emph{degenerate} indexing matrix, and the other indexing matrices will be called \emph{non degenerate} indexing matrices.

  The number of columns of an indexing matrix $A$ will be called the \emph{length of} $A$, and it will be denoted by $l(A).$ The number of columns of an indexing matrix $A$ containing a $0$ will be called the \emph{blob-rank} of $A.$
 An indexing matrix  will be called \emph{positive indexing matrix,} if all its entries are positive. Finally, an indexing matrix is called an \emph{initial indexing matrix} if its second row consists only of zeros.

Observe that if $A$ is a non degenerated indexing matrix, then there are an initial indexing matrix $B$ and a positive indexing matrix $C$ such that $A=[B|C]$, that is, $A$ is obtained from $B$ and $C$ by \emph{horizontal concatenation of matrices},
\begin{equation}\label{A=BC}
    A=\left[\begin{matrix}
        b_1&\cdots &b_r&c_{11}&\cdots&c_{1m}\\
        0&\cdots & 0&c_{21}&\cdots&c_{2m}
    \end{matrix}\right].
\end{equation}
The following definition says how to parameterize the normal form of the elements of $Bl_n$ in terms of indexing matrices.
\begin{definition}\label{def-basis-matrix-blob}
  \begin{enumerate}
    \item[(i)] For $A=\Big[\!\!\begin{array}{c}\infty\\\infty\end{array}\!\!\Big]$, we define $U(A)=1$.
    \item[(ii)] For $0\leq j\leq i\leq n-1,$ we define  $U\Big[\!\!\begin{array}{c}i\\j\end{array}\!\!\Big]$ as follows:
     \begin{equation*}
      U\Big[\!\!\begin{array}{c}i\\j\end{array}\!\!\Big]=\left\{\begin{array}{ll}
                              u_i,&\textrm{if}\quad i=j \\
                              \quad \\
                             u_iu_{i-1}\cdots u_{j},& \textrm{if}\quad i>j.
                            \end{array}\right.
     \end{equation*}
     \item[(iii)] If  $A=[a_{ij}]$ is an indexing matrix with $l(A)>1$, the element $U(A)$ is defined by
         \begin{equation*}
           U\left[\begin{matrix}
                       a_{11} & a_{12} & \cdots & a_{1m} \\
                       a_{21} & a_{22} & \cdots & a_{2m}
                     \end{matrix}\right]:=
           U\left[\begin{matrix}
            a_{11} \\
            a_{21}
          \end{matrix}\right]U\left[\begin{matrix}
            a_{12} \\
            a_{22}
          \end{matrix}\right]\cdots U\left[\begin{matrix}
            a_{1m} \\
            a_{2m}
          \end{matrix}\right].
         \end{equation*}
     \end{enumerate}
\end{definition}
\begin{remark}\label{nthCatalan}\rm
As usual we denote $\mathrm{C}_n=\frac{1}{n+1}\binom{2n}{n}$  the $n-$th Catalan number.  Recall that there are $\binom{2n}{n}$ indexing matrices and that there are $\mathrm{C}_n$  positive indexing matrices, including the degenerated one.
Also note that if $A$ is a positive indexing matrix, then $U(A)$ belongs to the so called Jones monoid \cite{JoIM1983}. Moreover, the  $U(A)$ correspond to the normal form of the elements of the Jones monoid.
\end{remark}

\subsection{Martin-Saleur blob monoid}

 Recall that a Temperley-Lieb diagram or TL-diagram, on $n$ points is a finite graph $D$ immersed at the rectangular frame $R:=\{(x,y)\in\mathbb{R}^2; 0\leq y\leq 1\}$. The set of vertices consists of $2n$ points located on the boundary of $R$, which are called \emph{boundary points}
of $D$. There are $n$ of those points contained in $\mathbb{R}\times \{0\}$ called \emph{bottom boundary points} and the other ones contained in $\mathbb{R}\times \{1\}$ are called
\emph{top boundary points}.
The edges of $D$ are contained in the interior of $R$ and are called \emph{arcs}; these arcs connect pairs of boundary points with the restriction that they never intersect each other.
The connected components of the complement of a TL-diagram in  $R$ are called regions.
\begin{equation}\label{eq-tikz-TL-diagram-1}
   \begin{tikzpicture}[xscale=0.3,yscale=0.3]
\node[] at (-1,1.5) {$D=$};
\draw[thick] (1,0) to [out=90,in=90](4,0);
\draw[thick] (2,0) to [out=90,in=90](3,0);
\draw[thick] (5,0) to [out=90,in=90](6,0);
\draw[thick] (7,0) to [out=90,in=270](3,3);
\draw[thick] (8,0) to [out=90,in=270](8,3);
\draw[thick] (2,3) to [out=270,in=270](1,3);
\draw[thick] (4,3) to [out=270,in=270](7,3);
\draw[thick] (5,3) to [out=270,in=270](6,3);
\draw[thick] (9,0) to [out=90,in=270](9,3);
\end{tikzpicture}
\end{equation}
Note that in a TL-diagram $D,$ there are at least one and at most two not bounded regions.
When we are in the former case, we differentiate between the left and right regions, which are called \emph{left side} and \emph{right side} of $D$, respectively. If there is only one not bounded region we called it the \emph{left side} anyway.

If an arc of $D$ is contained in the boundary of the left region of $D,$ we say that it is \emph{exposed to the left side of the diagram}. In the following figure we have painted in red the arcs that are exposed to the left side of $D:$
\begin{equation*}
    \begin{tikzpicture}[xscale=0.3,yscale=0.3]
\node[] at (-1,1.5) {$D=$};
\draw[thick,red] (1,0) to [out=90,in=90](4,0);
\draw[thick,black] (2,0) to [out=90,in=90](3,0);
\draw[thick,red] (5,0) to [out=90,in=90](6,0);
\draw[thick,red] (7,0) to [out=90,in=270](3,3);
\draw[thick,black] (8,0) to [out=90,in=270](8,3);
\draw[thick,red] (2,3) to [out=270,in=270](1,3);
\draw[thick,black] (4,3) to [out=270,in=270](7,3);
\draw[thick,black] (5,3) to [out=270,in=270](6,3);
\draw[thick,black] (9,0) to [out=90,in=270](9,3);
\end{tikzpicture}
\end{equation*}

\begin{definition}\label{def-blob-TL-diagram}
   A blobbed TL-diagram is a TL-diagram that may have some decorations (blobs) on the arcs, according to the following rules:
    \begin{enumerate}
        \item[(i)] Only arcs exposed to the left side of the diagram may be decorated by a blob.
        \item[(ii)] Each arc has  at  most one blob.
    \end{enumerate}
\end{definition}
Note that the definition of blobbed TL-diagrams includes the classical TL-diagrams.
 The following are examples of blobbed TL-diagram.

\begin{equation}\label{eq-tikz-blobbed-diagram}
\begin{tikzpicture}[xscale=0.3,yscale=0.3]
\node[] at (-1,1.5) {$D=$};
\draw[thick] (1,0) to [out=90,in=90](4,0);
\draw[thick] (2,0) to [out=90,in=90](3,0);
\draw[thick] (5,0) to [out=90,in=90](6,0);
\draw[thick] (7,0) to [out=90,in=270](3,3);
\draw[thick] (8,0) to [out=90,in=270](8,3);
\draw[thick] (2,3) to [out=270,in=270](1,3);
\draw[thick] (4,3) to [out=270,in=270](7,3);
\draw[thick] (5,3) to [out=270,in=270](6,3);
\draw[thick] (9,0) to [out=90,in=270](9,3);

\draw[fill] (5,1.5) circle [radius=0.15];
\draw[fill] (2.5,0.9) circle [radius=0.15];
\node[] at (10,1.5) {$;$};
\end{tikzpicture} \quad \begin{tikzpicture}[xscale=0.3,yscale=0.3]
\node[] at (-1,1.5) {$D'=$};
\draw[thick] (1,0) to [out=90,in=90](2,0);
\draw[thick] (3,0) to [out=90,in=270](7,3);
\draw[thick] (4,0) to [out=90,in=90](5,0);
\draw[thick] (6,0) to [out=90,in=90](7,0);

\draw[thick] (8,0) to [out=90,in=270](8,3);
\draw[thick] (2,3) to [out=270,in=270](1,3);
\draw[thick] (4,3) to [out=270,in=270](3,3);
\draw[thick] (5,3) to [out=270,in=270](6,3);
\draw[thick] (9,0) to [out=90,in=270](9,3);

\draw[fill] (5,1.5) circle [radius=0.15];
\draw[fill] (1.5,2.7) circle [radius=0.15];
\draw[fill] (1.5,0.3) circle [radius=0.15];
\end{tikzpicture}
\end{equation}

\begin{definition}\label{def-diag-blob-monoid}
The \emph{diagrammatic blob monoid} $\mathfrak{Bl}_n$ consists of the blobbed TL-diagrams at $n$ points with the usual concatenation product of diagrams, but neglecting all loops, whether they contain blobs or not, and leaving only one blob on all arcs containing some blobs.
\end{definition}

   Below we show the product $DD'$ of the diagrams of (\ref{eq-tikz-blobbed-diagram}).
   \begin{equation}\label{eq-tikz-blobbed-product}
    \begin{tikzpicture}[xscale=0.3,yscale=0.3]
\node[] at (-2,3) {$DD'=$};
\draw[thick] (1,0+3) to [out=90,in=90](4,0+3);
\draw[thick] (2,0+3) to [out=90,in=90](3,0+3);
\draw[thick] (5,0+3) to [out=90,in=90](6,0+3);
\draw[thick] (7,0+3) to [out=90,in=270](3,3+3);
\draw[thick] (8,0+3) to [out=90,in=270](8,3+3);
\draw[thick] (2,3+3) to [out=270,in=270](1,3+3);
\draw[thick] (4,3+3) to [out=270,in=270](7,3+3);
\draw[thick] (5,3+3) to [out=270,in=270](6,3+3);
\draw[thick] (9,0+3) to [out=90,in=270](9,3+3);

\draw[fill] (5,1.5+3) circle [radius=0.15];
\draw[fill] (2.5,0.9+3) circle [radius=0.15];
\node[] at (11,3) {$=$};
\draw[thin,dashed] (0.5,3)--(9.5,3);

\draw[thick] (1,0) to [out=90,in=90](2,0);
\draw[thick] (3,0) to [out=90,in=270](7,3);
\draw[thick] (4,0) to [out=90,in=90](5,0);
\draw[thick] (6,0) to [out=90,in=90](7,0);
\draw[thick] (8,0) to [out=90,in=270](8,3);
\draw[thick] (2,3) to [out=270,in=270](1,3);
\draw[thick] (4,3) to [out=270,in=270](3,3);
\draw[thick] (5,3) to [out=270,in=270](6,3);
\draw[thick] (9,0) to [out=90,in=270](9,3);

\draw[fill] (5,1.5) circle [radius=0.15];
\draw[fill] (1.5,2.7) circle [radius=0.15];
\draw[fill] (1.5,0.3) circle [radius=0.15];
\draw[thick] (2+12,3+1.5) to [out=270,in=270](1+12,3+1.5);
\draw[thick] (4+12,3+1.5) to [out=270,in=270](7+12,3+1.5);
\draw[thick] (5+12,3+1.5) to [out=270,in=270](6+12,3+1.5);
\draw[thick] (1+12,1.5) to [out=90,in=90](2+12,1.5);
\draw[thick] (3+12,1.5) to [out=90,in=270](3+12,4.5);
\draw[thick] (4+12,1.5) to [out=90,in=90](5+12,1.5);
\draw[thick] (6+12,1.5) to [out=90,in=90](7+12,1.5);
\draw[thick] (8+12,1.5) to [out=90,in=270](8+12,4.5);
\draw[thick] (9+12,1.5) to [out=90,in=270](9+12,4.5);
\draw[fill] (1.5+12,0.3+1.5) circle [radius=0.15];
\draw[fill] (3+12,3) circle [radius=0.15];
\end{tikzpicture}
\end{equation}

Recall that  $\mathfrak{Bl}_n$ is generated by $\mathfrak{u}_0,\mathfrak{u}_1,\ldots, \mathfrak{u}_{n-1}$, where:
\begin{equation*}
    \begin{tikzpicture}[xscale=0.3,yscale=0.3]
\node[] at (-1,1.5) {$\mathfrak{u}_0:=$};
\draw[thick] (1,0) to [out=90,in=270](1,3);
\draw[fill] (1,1.5) circle [radius=0.15];
\draw[thick] (2,0) to [out=90,in=270](2,3);
\node[] at (3,1.5) {$\cdots$};
\draw[thick] (4,0) to [out=90,in=270](4,3);
\draw[thick] (5,0) to [out=90,in=270](5,3);
\draw[thick] (6,0) to [out=90,in=270](6,3);
\draw[thick] (7,0) to [out=90,in=270](7,3);
\node[] at (8,1.5) {$\cdots$};
\draw[thick] (9,0) to [out=90,in=270](9,3);
\node[below] at (1,0) {\small$1$};
\node[below] at (9,0) {\small$n$};
\node[] at (9.5,1.4) {};

\end{tikzpicture}\qquad
\begin{tikzpicture}[xscale=0.3,yscale=0.3]
\node[] at (-1,1.5) {$\mathfrak{u}_i:=$};
\draw[thick] (1,0) to [out=90,in=270](1,3);
\draw[thick] (2,0) to [out=90,in=270](2,3);
\node[] at (3,1.5) {$\cdots$};
\draw[thick] (4,0) to [out=90,in=270](4,3);
\draw[thick] (5,0) to [out=90,in=90](6,0);
\draw[thick] (5,3) to [out=270,in=270](6,3);
\draw[thick] (7,0) to [out=90,in=270](7,3);
\node[] at (8,1.5) {$\cdots$};
\draw[thick] (9,0) to [out=90,in=270](9,3);
\node[below] at (5,0) {\small$i$};
\node[below] at (1,0) {\small$1$};
\node[below] at (9,0) {\small$n$};
\node[] at (9.5,1.4) {};
\node[] at (14,1.4) {with $i\not= 0.$};

\end{tikzpicture}
\end{equation*}
These elements $\mathfrak{u}_i$'s satisfy the defining relations
 of $Bl_n$, i.e. we have:
\begin{align}\label{eq-all-def-blob-mon-diag}
\mathfrak{u}_i^2=\mathfrak{u}_i,\quad  & \mathfrak{u}_i\mathfrak{u}_j=\mathfrak{u}_j\mathfrak{u}_i,\quad \text{for $|i-j|>1$,} \\
\label{eq-all-def-blob-mon-diag1}
\mathfrak{u}_1\mathfrak{u}_0\mathfrak{u}_1=\mathfrak{u}_1, &\quad \mathfrak{u}_i\mathfrak{u}_j\mathfrak{u}_i=\mathfrak{u}_i,\quad \text{for $|i-j|>1$ and $i,j\not=0$.}
\end{align}
Moreover, we will conclude the subsection by proving in Theorem \ref{theo-iso-diagBlobMon-algBlobMon} that these generators and relations define a presentation for $\mathfrak{Bl}_n$. For this purpose, we start by associating to each indexing matrix $A$ an element $\BBU( A)\in \mathfrak{Bl}_n$. To be precise, the association is the one given in Definition \ref{def-basis-matrix-blob}, but
changing $u_i$ by $\mathfrak{u}_i$. In this way the elementary indexing matrices are as in the following figure.
\begin{figure}[ht]
\begin{tikzpicture}[xscale=0.4,yscale=0.4]
\draw[thick] (1,0) to [out=90,in=90](2,0);
\draw[fill] (1.5,0.3) circle [radius=0.15];
\draw[thick] (3,0) to [out=90,in=270](1,3);
\draw[thick] (4,0) to [out=90,in=270](2,3);
\node[] at (4,1.5) {$\cdots$};
\draw[thick] (5,3) to [out=270,in=90](7,0);
\draw[thick] (6,3) to [out=270,in=90](8,0);
\draw[thick] (8,3) to [out=270,in=90](10,0);
\draw[thick] (11.3,3) to [out=270,in=90](11.3,0);
\node[] at (8,1.5) {$\cdots$};
\draw[thick] (10,3) to [out=270,in=270](9,3);
\node[] at (12,1.5) {$\cdots$};
\draw[thick] (13,3) to [out=270,in=90](13,0);
\node[below] at (10,0) {\small$ i+1$};
\node[below] at (1,0) {\small$1$};
\node[above] at (9,3) {\small$i$};
\node[below] at (13,0) {\small$n$};
\node[] at (14,1.5) {$;$};
\end{tikzpicture} \quad
\begin{tikzpicture}[xscale=0.4,yscale=0.4]
\draw[thick] (1,0) to [out=90,in=270](1,3);
\draw[thick] (2,0) to [out=90,in=270](2,3);
\node[] at (3,1.5) {$\cdots$};
\draw[thick] (4,0) to [out=90,in=270](4,3);
\draw[thick] (5,0) to [out=90,in=90](6,0);
\draw[thick] (5,3) to [out=270,in=90](7,0);
\draw[thick] (6,3) to [out=270,in=90](8,0);
\draw[thick] (8,3) to [out=270,in=90](10,0);
\draw[thick] (11.3,3) to [out=270,in=90](11.3,0);
\node[] at (8,1.5) {$\cdots$};
\draw[thick] (10,3) to [out=270,in=270](9,3);
\node[] at (12,1.5) {$\cdots$};
\draw[thick] (13,3) to [out=270,in=90](13,0);
\node[below] at (5,0) {\small$j$};
\node[above] at (5,3) {\small$j$};
\node[below] at (10,0) {\small$i+1$};
\node[below] at (1,0) {\small$ 1$};
\node[above] at (9,3) {\small$i$};
\node[below] at (13,0) {\small$n$};
\end{tikzpicture}
\caption{On the left is $\BBU\left[\begin{matrix}
           i \\
           0
          \end{matrix}\right]$ and on the right is $\BBU\left[\begin{matrix}
           i \\
           j
          \end{matrix}\right]$ with $0\not= j<i.$}
\end{figure}

\begin{example}  For $n=6$ and $A=\left[\begin{matrix}
              3 & 4 & 5 \\
              0 & 0 & 3
            \end{matrix}\right]$, we have $\BBU(A) =
           \BBU\left[\begin{matrix}
            3 \\
            0
          \end{matrix}\right]\BBU\left[\begin{matrix}
           4 \\
           0
          \end{matrix}\right]\BBU\left[\begin{matrix}
           5\\
            3
          \end{matrix}\right]$. Then:
\begin{equation*}
\begin{tikzpicture}[xscale=0.5,yscale=0.5]
\node[] at (-2,-1) {$\BBU(A)=$};
\draw[thick] (1,0) to [out=90,in=90](2,0);
\draw[fill] (1.5,0.3) circle [radius=0.15];
\draw[thick] (3,0) to [out=90,in=270](1,2);
\draw[thick] (4,0) to [out=90,in=270](2,2);
\draw[thick] (3,2) to [out=270,in=270](4,2);
\draw[thick] (5,2) to [out=270,in=90](5,0);
\draw[thick] (6,2) to [out=270,in=90](6,0);
\draw[thin] (0.5,0)--(6.5,0);
\draw[thick] (1,0) to [out=270,in=90](3,-2);
\draw[thick] (1,-2) to [out=90,in=90](2,-2);
\draw[fill] (1.5,-1.7) circle [radius=0.15];
\draw[thick] (2,0) to [out=270,in=90](4,-2);
\draw[thick] (3,0) to [out=270,in=90](5,-2);
\draw[thick] (4,0) to [out=270,in=270](5,0);
\draw[thick] (6,0) to [out=270,in=90](6,-2);
\draw[thin] (0.5,-2)--(6.5,-2);
\draw[thick] (1,-2) to [out=270,in=90](1,-4);
\draw[thick] (2,-2) to [out=270,in=90](2,-4);
\draw[thick] (3,-2) to [out=270,in=90](5,-4);
\draw[thick] (3,-4) to [out=90,in=90](4,-4);
\draw[thick] (5,-2) to [out=270,in=270](6,-2);
\draw[thick] (4,-2) to [out=270,in=90](6,-4);
\node[] at (8,-1) {$=$};
\draw[thick] (11,-2.5) to [out=90,in=90](12,-2.5);
\draw[fill] (11.5,-2.2) circle [radius=0.15];
\draw[thick] (13,-2.5) to [out=90,in=90](14,-2.5);
\draw[fill] (15.5,-2.2) circle [radius=0.15];
\draw[thick] (15,-2.5) to [out=90,in=90](16,-2.5);
\draw[thick] (12,0.5) to [out=270,in=270](15,0.5);
\draw[thick] (13,0.5) to [out=270,in=270](14,0.5);
\draw[thick] (11,0.5) to [out=270,in=270](16,0.5);

\end{tikzpicture}
\end{equation*}
\end{example}
\begin{example}\label{ex-matrix-form-tikz-blobbed-diagram}
   For the diagrams of (\ref{eq-tikz-blobbed-diagram}), we have
    \begin{equation*}
        D=\BBU\left[\begin{matrix}
        1 & 2 & 5 & 6 \\
        0 & 0 & 2 & 5
        \end{matrix}\right],\quad D'=
       \BBU \left[\begin{matrix}
            0&1&3&5&6\\
            0&0&0&4&6
        \end{matrix}\right]
    \end{equation*}
  and
    \begin{equation*}
        DD'=\BBU\left[\begin{matrix}
            1&2&5&6\\
            0&0&4&6
        \end{matrix}\right].
    \end{equation*}
\end{example}

\begin{remark}\rm\label{remark-relation-blobrank-and-blobbed-arcs}
 Note that the number of blobbed arcs in the element $\BBU(A)$ corresponds to the blob-rank of the indexing matrix $A.$
\end{remark}

\begin{theorem}\label{theo-iso-diagBlobMon-algBlobMon}
The map $u_i\mapsto \mathfrak{u}_i$ defines a monoid isomorphism $\phi:Bl_n\longrightarrow \mathfrak{Bl}_n$.
\end{theorem}
 \begin{proof}  From (\ref{eq-all-def-blob-mon-diag}) and (\ref{eq-all-def-blob-mon-diag1}) follows that  $\phi$ is a monoid epimorphism. So to finish the proof it is enough to show that  $|Bl_n| \leq |\mathfrak{Bl}_n|$.
Now, using the defining relations of $Bl_n$, each element of $Bl_n$ can be  written in the form $U(A)$ for some indexing matrix $A$. Then,  $|Bl_n|\leq\binom{2n}{n}=|\mathfrak{Bl}_n|$.
\end{proof}

\begin{corollary}\label{NFBlob}
The generators $\mathfrak{u}_0, \mathfrak{u}_1, \ldots , \mathfrak{u}_{n-1}$  together with the relations   (\ref{eq-all-def-blob-mon-diag})-(\ref{eq-all-def-blob-mon-diag1}) define a presentation for $\mathfrak{Bl}_n$.  In particular, $\BBU(A)$ is in normal form, for  all $A$  indexing matrix.
\end{corollary}

\subsection{Blob monoid as partition monoid }
The goal of this subsection is to show two ways of realizing the blob monoid as a submonoid of the partition monoid: one realization is called the Motzkin blob monoid and the other the hook blob monoid.

\subsubsection{Partition and abacus monoids}
As usual we denote by  $\mathfrak{C}_n$, the partition monoid, that is, the set formed by the partitions  of $[1,n]\cup[1',n']$ with the product by concatenation, see \cite[Section 2]{KuMaCECJM2006} and \cite[Section 1]{HaRaEJC2005}  for details. For example, for $\alpha = \{\{1,2\}, \{3,3'\},\{4,1',4'\},\{2'\}\}$,
$\beta= \{\{1,3'\}, \{2,4\},\{3,1',2',4'\}\}\in \mathfrak{C}_4$, the product $\alpha\beta$ is depicted below.
\begin{center}
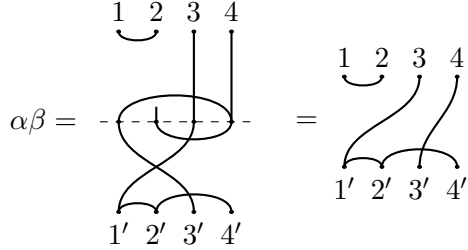
\begin{figure}[ht]
$$
\begin{tikzpicture}[xscale=0.5,yscale=0.4]
\node[] at (-1,0) {$\alpha\beta=$};
\draw[thick] (1,3) to [out=270,in=270](2,3);
\draw[thick](3,0)--(3,3);
\draw[thick] (1,0) to [out=90,in=90](4,0);
\draw[thick](4,0)--(4,3);
\draw[thick] (2,0) to [out=90,in=270](2,0.5);
\draw[dashed](0.5,0)--(4.5,0);
\node[above] at (1,3) {$1$};
\node[above] at (2,3) {$2$};
\node[above] at (3,3) {$3$};
\node[above] at (4,3) {$4$};
\draw[fill] (1,3) circle [radius=0.05];
\draw[fill] (2,3) circle [radius=0.05];
\draw[fill] (3,3) circle [radius=0.05];
\draw[fill] (4,3) circle [radius=0.05];
\draw[fill] (1,0) circle [radius=0.05];
\draw[fill] (2,0) circle [radius=0.05];
\draw[fill] (3,0) circle [radius=0.05];
\draw[fill] (4,0) circle [radius=0.05];
\draw[thick] (3,-3) to [out=90,in=270](1,0);
\draw[thick] (2,0) to [out=270,in=270](4,0);
\draw[thick] (1,-3) to [out=90,in=270](3,0);
\draw[thick] (1,-3) to [out=90,in=90](2,-3);
\draw[thick] (2,-3) to [out=90,in=90](4,-3);
\node[below] at (1,-3) {$1'$};
\node[below] at (2,-3) {$2'$};
\node[below] at (3,-3) {$3'$};
\node[below] at (4,-3) {$4'$};
\draw[fill] (1,-3) circle [radius=0.05];
\draw[fill] (2,-3) circle [radius=0.05];
\draw[fill] (3,-3) circle [radius=0.05];
\draw[fill] (4,-3) circle [radius=0.05];

\node[] at (6,0) {$=$};
\draw[thick] (1+6,1.5) to [out=270,in=270](2+6,1.5);
\draw[thick] (3+6,1.5) to [out=270,in=90](1+6,-1.5);
\draw[thick] (4+6,1.5) to [out=270,in=90](3+6,-1.5);

\node[below] at (1+6,-1.5) {$1'$};
\node[below] at (2+6,-1.5) {$2'$};
\node[below] at (3+6,-1.5) {$3'$};
\node[below] at (4+6,-1.5) {$4'$};
\node[above] at (1+6,1.5) {$1$};
\node[above] at (2+6,1.5) {$2$};
\node[above] at (3+6,1.5) {$3$};
\node[above] at (4+6,1.5) {$4$};
\draw[fill] (1+6,1.5) circle [radius=0.05];
\draw[fill] (2+6,1.5) circle [radius=0.05];
\draw[fill] (3+6,1.5) circle [radius=0.05];
\draw[fill] (4+6,1.5) circle [radius=0.05];
\draw[fill] (1+6,-1.5) circle [radius=0.05];
\draw[fill] (2+6,-1.5) circle [radius=0.05];
\draw[fill] (3+6,-1.5) circle [radius=0.05];
\draw[fill] (4+6,-1.5) circle [radius=0.05];

\draw[thick] (1+6,-1.5) to [out=90,in=90](2+6,-1.5);
\draw[thick] (2+6,-1.5) to [out=90,in=90](4+6,-1.5);
\end{tikzpicture}
$$
\caption{The product $\alpha\beta$.}\label{Fig01}
\end{figure}
\end{center}

Recall that the Jones monoid  can be realized as the submonoid of $\mathfrak{C}_n$ formed by the planar partitions having only blocks of cardinality 2. The {\it elementary tangle} $U_i$, where $i\in [1,n-1]$,  is the partition with blocks: $\{i,i+1,i', (i+1)'\}$ and
$\{k, k'\}$ for $k\not= i $. Cf.  \cite[Section 2]{KuMaCECJM2006}.

\begin{center}
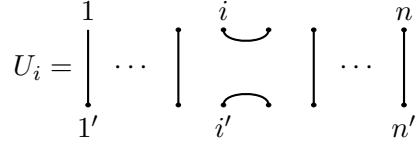
\begin{figure}[ht]
\begin{equation*}
\begin{tikzpicture}[xscale=0.6,yscale=0.5]
\node[] at (1,1) {$U_i=$};
\draw[thick] (2,0)--(2,2);
\node[] at (3,1) {$\cdots$};
\draw[thick] (4,0)--(4,2);
\draw[thick] (5,0) to [out=90,in=90](6,0);
\draw[thick] (7,0)--(7,2);
\node[] at (8,1) {$\cdots$};
\draw[thick] (9,0)--(9,2);
\draw[thick] (5,2) to [out=270,in=270](6,2);

\draw[fill] (4,2) circle [radius=0.05];
\draw[fill] (5,2) circle [radius=0.05];
\draw[fill] (6,2) circle [radius=0.05];
\draw[fill] (7,2) circle [radius=0.05];
\draw[fill] (9,2) circle [radius=0.05];
\draw[fill] (2,0) circle [radius=0.05];
\draw[fill] (4,0) circle [radius=0.05];
\draw[fill] (5,0) circle [radius=0.05];
\draw[fill] (6,0) circle [radius=0.05];
\draw[fill] (7,0) circle [radius=0.05];
\draw[fill] (9,0) circle [radius=0.05];
\node[below] at (2,0) {$1'$};
\node[below] at (9,0) {$n'$};
\node[below] at (5,0) {$i'$};
\node[above] at (2,2) {$1$};
\node[above] at (5,2) {$i$};
\node[above] at (9,2) {$n$};
\end{tikzpicture}
\end{equation*}
\caption{The elementary tangle $U_i$.}\label{TangleUi}
\end{figure}
\end{center}
\subsubsection{Motzkin blob monoid}
Define $U_0$ as the partition of $\mathfrak{C}_n$ formed by the blocks $\{1\},\{1'\}$, and $\{i,i'\}$ with $i\not= 1$.
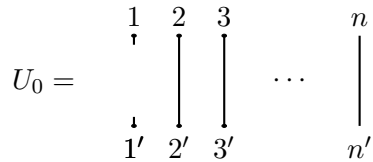
\begin{figure}[ht]
\begin{center}
\begin{equation*}
 \begin{tikzpicture}[xscale=0.6,yscale=0.4]
\node[] at (-1,1.5) {$U_0=$};
\draw[fill] (1,3) circle [radius=0.05];
\draw[thick] (1,3)--(1,2.7);
\draw[fill] (2,3) circle [radius=0.05];
\draw[fill] (3,3) circle [radius=0.05];
\draw[fill] (1,0) circle [radius=0.05];
\draw[thick] (1,0)--(1,0.3);
\draw[fill] (2,0) circle [radius=0.05];
\draw[fill] (3,0) circle [radius=0.05];
\draw[thick] (2,0) to [out=90,in=270](2,3);
\draw[thick] (3,0) to [out=90,in=270](3,3);
\draw[thick] (6,3)--(6,0);
\node[] at (4.5,1.5) {$\cdots$};
\node[below] at (1,0) {$1'$};

\node[below] at (1,0) {$1'$};
\node[below] at (2,0) {$2'$};
\node[below] at (3,0) {$3'$};
\node[below] at (6,0) {$n'$};
\node[above] at (1,3) {$1$};
\node[above] at (2,3) {$2$};
\node[above] at (3,3) {$3$};
\node[above] at (6,3) {$n$};
\end{tikzpicture}
\end{equation*}
\caption{The diagram for $U_0$.} \label{U0}
\end{center}
\end{figure}

The {\it Motzkin blob monoid}, denoted  $\mathfrak{M}_n^{(1)}$, is defined  as  the  submonoid of $\mathfrak{C}_n$ generated by $U_0, U_1,\ldots,U_{n-1}$.

\begin{theorem}\label{theo-iso-algblobmon-rookblobmon}
    The map $\mathfrak{u}_i\mapsto U_i$ defines a monoid isomorphism  $ \mathfrak{Bl}_n\rightarrow \mathfrak{M}_n^{(1)}.$
\end{theorem}
\begin{proof}
  It is routine to verify that the generators $U_0,\dots,U_{n-1}$ satisfies  the relations (\ref{eq-all-def-blob-mon-diag})-(\ref{eq-all-def-blob-mon-diag1}). For example:
    \begin{equation*}
       \begin{tikzpicture}[xscale=0.5,yscale=0.3]
\node[] at (-0.5,3) {$U_0^2=$};
\draw[thick] (1,0)--(1,1);
\draw[thick] (1,2)--(1,4);
\draw[thick] (1,5)--(1,6);
\draw[thick] (2,0) to [out=90,in=270](2,3);
\node[] at (3,3) {$\cdots$};

\draw[thick] (4,0+3) to [out=90,in=270](4,3+3);
\draw[thick] (2,0+3) to [out=90,in=270](2,3+3);
\draw[thick] (4,0) to [out=90,in=270](4,3+3);
\node[below] at (4,0) {$n'$};
\node[below] at (1,0) {$1'$};
\node[below] at (2,0) {$2'$};
\node[above] at (1,6) {$1$};
\node[above] at (2,6) {$2$};
\node[above] at (4,6) {$n$};
\draw[fill] (1,6) circle [radius=0.05];
\draw[fill] (2,6) circle [radius=0.05];
\draw[fill] (4,6) circle [radius=0.05];
\draw[fill] (1,0) circle [radius=0.05];
\draw[fill] (2,0) circle [radius=0.05];
\draw[fill] (4,0) circle [radius=0.05];
\node[] at (5.5,3) {$=U_0,$};

\end{tikzpicture} \quad \begin{tikzpicture}[xscale=0.5,yscale=0.3]
\node[] at (-1.5,2) {$U_1U_0U_1=$};
\draw[fill] (1,5) circle [radius=0.05];
\draw[thick] (1,3)--(1,2.7);
\draw[fill] (2,5) circle [radius=0.05];
\draw[fill] (3,5) circle [radius=0.05];
\draw[fill] (1,-1) circle [radius=0.05];
\draw[thick] (1,1)--(1,1.3);
\draw[fill] (2,-1) circle [radius=0.05];
\draw[fill] (3,-1) circle [radius=0.05];
\draw[fill] (6,-1) circle [radius=0.05];
\draw[fill] (6,5) circle [radius=0.05];
\draw[thick] (2,1) to [out=90,in=270](2,3);
\draw[thick] (3,-1) to [out=90,in=270](3,5);
\draw[thick] (6,5)--(6,-1);
\node[] at (4.5,2) {$\cdots$};

\node[below] at (1,-1) {$1'$};
\node[below] at (2,-1) {$2'$};
\node[below] at (3,-1) {$3'$};
\node[below] at (6,-1) {$n'$};
\node[above] at (1,5) {$1$};
\node[above] at (2,5) {$2$};
\node[above] at (3,5) {$3$};
\node[above] at (6,5) {$n$};
\node[] at (8,2) {$=U_1$};
\draw[thick] (1,3) to [out=90,in=90](2,3);
\draw[thick] (1,5) to [out=270,in=270](2,5);
\draw[thick] (1,-1) to [out=90,in=90](2,-1);
\draw[thick] (1,1) to [out=270,in=270](2,1);
\end{tikzpicture}
    \end{equation*}
Thus, by Corollary \ref{NFBlob} the map $\mathfrak{u}_i\mapsto U_i$ defines a monoid epimorphism $\phi:{\mathfrak{Bl}}_n\rightarrow\mathfrak{M}_n^{(1)}.$
 Note that if in a blobbed TL-diagram $D$ we remove the blobs, then the respective arcs are divided into pieces of arcs that, by homotopy, become points (singular blocks), thus obtaining a new diagram $\widetilde{D}$ corresponding to an element of $\mathfrak{M}_n^{(1)}$. It is clear that the map $D\mapsto \widetilde{D}$ is the diagrammatic version of the epimorphisms $\phi$ and that it is also injective, so the proof follows.
\end{proof}
\begin{remark}\rm
The above theorem says that diagrams of $\mathfrak{M}_n^{(1)}$ are obtained by replacing the arcs that connect $k$ with $l$  and that   have  some blobs,  by a piece of arc at $k$ and another piece of arc at $l$, where $k,l\in [1,n]\cup[1',n']$.
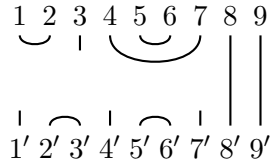
\begin{figure}[ht]
\begin{equation*}
\begin{tikzpicture}[xscale=0.4,yscale=0.4]
\draw[thick] (1,0)--(1,0.5);
\draw[thick] (4,0)--(4,0.5);

\draw[thick] (2,0) to [out=90,in=90](3,0);
\draw[thick] (5,0) to [out=90,in=90](6,0);
\draw[thick] (7,0)--(7,0.5);
\draw[thick] (3,3)--(3,2.5);
\draw[thick] (8,0) to [out=90,in=270](8,3);
\draw[thick] (2,3) to [out=270,in=270](1,3);
\draw[thick] (4,3) to [out=270,in=270](7,3);
\draw[thick] (5,3) to [out=270,in=270](6,3);
\draw[thick] (9,0) to [out=90,in=270](9,3);

\node[below] at (1,0) {$1'$};
\node[below] at (2,0) {$2'$};
\node[below] at (3,0) {$3'$};
\node[below] at (4,0) {$4'$};
\node[below] at (5,0) {$5'$};
\node[below] at (6,0) {$6'$};
\node[below] at (7,0) {$7'$};
\node[below] at (8,0) {$8'$};
\node[below] at (9,0) {$9'$};
\node[above] at (1,3) {$1$};
\node[above] at (2,3) {$2$};
\node[above] at (3,3) {$3$};
\node[above] at (4,3) {$4$};
\node[above] at (5,3) {$5$};
\node[above] at (6,3) {$6$};
\node[above] at (7,3) {$7$};
\node[above] at (8,3) {$8$};
\node[above] at (9,3) {$9$};

\end{tikzpicture}
\end{equation*}
\caption{Martin-Saleur blob diagram $D$ of (\ref{eq-tikz-blobbed-diagram}) as an element of $\mathfrak{M}_9^{(1)}$.}
\end{figure}
\end{remark}
\subsubsection{Hook blob monoid}\label{ssec-Hook-blob}
To define the  hook blob monoid,  we need the following notation. Denote by  $\underline{\mathfrak{C}}_n$ the partition monoid of the set $[1,n]\cup [1',n']\cup \{0,0'\}$. Note that each partition $P\in\mathfrak{C}_n$ defines  a partition $\underline{P}\in\underline{\mathfrak{C}}_n$, which is formed by the blocks of $P$ and the block $\{0, 0'\}$. In fact, this correspondence defines a monoid homomorphism from $\mathfrak{C}_n$
into $\underline{\mathfrak{C}}_n$. Now, for $i\in[1,n-1]$ we denote by $\underline{U}_i$ the elementary tangle $U_i$ as element of $\underline{\mathfrak{C}}_n$.
In addition, we define   $\underline{U}_0$ as  the partition  of $\underline{\mathfrak{C}}_n$ formed by the blocks:
$$\{0,1,0',1'\},\{2,2'\},\ldots,\{n,n'\}.$$
\begin{figure}[ht]
\begin{center}
\begin{equation*}
 \begin{tikzpicture}[xscale=0.6,yscale=0.4]
\node[] at (-1,1.5) {$\underline{U}_0=$};

\draw[fill] (0,3) circle [radius=0.05];
\draw[thick] (0,3)--(0,0);
\draw[fill] (0,0) circle [radius=0.05];
\draw[thick] (0,0) to [out=90,in=90](1,0);
\draw[thick] (0,3) to [out=270,in=270](1,3);

\node[below] at (0,0) {$0'$};
\node[above] at (0,3) {$0$};

\draw[fill] (1,3) circle [radius=0.05];
\draw[fill] (2,3) circle [radius=0.05];
\draw[fill] (3,3) circle [radius=0.05];
\draw[fill] (1,0) circle [radius=0.05];
\draw[fill] (2,0) circle [radius=0.05];
\draw[fill] (3,0) circle [radius=0.05];
\draw[thick] (2,0) to [out=90,in=270](2,3);
\draw[thick] (3,0) to [out=90,in=270](3,3);
\draw[thick] (6,3)--(6,0);
\node[] at (4.5,1.5) {$\cdots$};
\node[below] at (1,0) {$1'$};

\node[below] at (1,0) {$1'$};
\node[below] at (2,0) {$2'$};
\node[below] at (3,0) {$3'$};
\node[below] at (6,0) {$n'$};
\node[above] at (1,3) {$1$};
\node[above] at (2,3) {$2$};
\node[above] at (3,3) {$3$};
\node[above] at (6,3) {$n$};
\end{tikzpicture}
\end{equation*}
\caption{hook-diagram for $\underline{U}_0$} \label{underlineU0}
\end{center}
\end{figure}

 The {\it hook blob monoid} $\mathfrak{H}_n$ is defined as  the submonoid of $\underline{\mathfrak{C}}_{n}$ generated by $\underline{U}_0, \underline{U}_1,\dots,\underline{U}_n.$ The diagram associated to any element of $\mathfrak{H}_n$ will be called \emph{hook-diagram}.
\begin{theorem}\label{theo-iso-algblobmon-hookblobmon}
    The map $\mathfrak{u}_i\mapsto \underline{U}_i$ defines a monoids  isomorphism from  $\mathfrak{Bl}_n$ to $\mathfrak{H}_n.$
\end{theorem}
\begin{proof}
 Analogous to the proof of Theorem \ref{theo-iso-algblobmon-rookblobmon}. For example, the hook-diagrams below show how two relations are respected by the map:
    \begin{equation*}
    \begin{tikzpicture}[xscale=0.5,yscale=0.3]
\node[] at (-1.5,3) {$\underline{U}_0^2=$};
\draw[thick] (0,0) to [out=90,in=90](1,0);
\draw[thick] (0,3) to [out=270,in=270](1,3);
\draw[thick] (0,0) to [out=90,in=270](0,3);
\draw[thick] (2,0) to [out=90,in=270](2,3);
\node[] at (3,3) {$\cdots$};

\draw[thick] (4,0+3) to [out=90,in=270](4,3+3);
\draw[thick] (0,0+3) to [out=90,in=90](1,0+3);
\draw[thick] (0,3+3) to [out=270,in=270](1,3+3);
\draw[thick] (0,0+3) to [out=90,in=270](0,3+3);
\draw[thick] (2,0+3) to [out=90,in=270](2,3+3);
\draw[thick] (4,0) to [out=90,in=270](4,3+3);
\node[below] at (0,0) {$0'$};
\node[below] at (4,0) {$n'$};
\node[below] at (1,0) {$1'$};
\node[below] at (2,0) {$2'$};
\node[below] at (0,0) {$0'$};
\node[above] at (1,6) {$1$};
\node[above] at (2,6) {$2$};
\node[above] at (0,6) {$0$};
\node[above] at (4,6) {$n$};
\draw[fill] (0,6) circle [radius=0.05];
\draw[fill] (0,0) circle [radius=0.05];
\draw[fill] (1,6) circle [radius=0.05];
\draw[fill] (2,6) circle [radius=0.05];
\draw[fill] (4,6) circle [radius=0.05];
\draw[fill] (1,0) circle [radius=0.05];
\draw[fill] (2,0) circle [radius=0.05];
\draw[fill] (4,0) circle [radius=0.05];
\node[] at (5.5,3) {$=\underline{U}_0,$};

\end{tikzpicture}  \quad \begin{tikzpicture}[xscale=0.5,yscale=0.2]
\node[] at (-2.5,1.5) {$\underline{U}_1\underline{U}_0\underline{U}_1=$};
\draw[thick] (0,0) to [out=90,in=90](1,0);
\draw[thick] (1,3) to [out=90,in=90](2,3);
\draw[thick] (1,0) to [out=270,in=270](2,0);
\draw[thick] (1,6) to [out=270,in=270](2,6);
\draw[thick] (1,-3) to [out=90,in=90](2,-3);
\draw[thick] (0,3) to [out=270,in=270](1,3);
\draw[thick] (0,-3) to [out=90,in=270](0,6);
\draw[thick] (2,0) to [out=90,in=270](2,3);
\node[] at (4,1.5) {$\cdots$};
\draw[thick] (5,-3) to [out=90,in=270](5,6);
\draw[thick] (3,-3) to [out=90,in=270](3,6);
\node[below] at (0,0-3) {$0'$};
\node[below] at (5,0-3) {$n'$};
\node[below] at (1,0-3) {$1'$};
\node[below] at (2,0-3) {$2'$};
\node[below] at (3,0-3) {$3'$};
\node[below] at (0,0-3) {$0'$};
\node[above] at (1,3+3) {$1$};
\node[above] at (2,3+3) {$2$};
\node[above] at (3,3+3) {$3$};
\node[above] at (0,3+3) {$0$};
\node[above] at (5,3+3) {$n$};
\draw[fill] (0,3+3) circle [radius=0.05];
\draw[fill] (0,3) circle [radius=0.05];
\draw[fill] (0,0) circle [radius=0.05];
\draw[fill] (0,0-3) circle [radius=0.05];
\draw[fill] (1,3+3) circle [radius=0.05];
\draw[fill] (2,3+3) circle [radius=0.05];
\draw[fill] (5,3+3) circle [radius=0.05];
\draw[fill] (1,0-3) circle [radius=0.05];
\draw[fill] (2,0-3) circle [radius=0.05];
\draw[fill] (5,0-3) circle [radius=0.05];
\draw[fill] (5,3+3) circle [radius=0.05];
\draw[fill] (3,3+3) circle [radius=0.05];
\draw[fill] (3,-3) circle [radius=0.05];
\node[] at (6.5,1.4) {$=\underline{U}_1$};

\end{tikzpicture} 
    \end{equation*}
The above mapping and the Theorem \ref{theo-iso-algblobmon-rookblobmon} imply that we have a  monoid epimorphism $\psi :  \mathfrak{M}_n^{(1)}\rightarrow  \mathfrak{H}_n$ such that $ U_i\mapsto \underline{U}_i$. Now, given $P\in\mathfrak{M}_n^{(1)}$, we define $P'\in \mathfrak{H}_n$ as the partition obtained by connecting all singleton blocks of $P$ to the block $\{0,0'\}$. Then one can check that map $P\mapsto P'$ is the diagrammatic version of the monoid epimorphism $\psi$, which is in fact injective. Hence the proof follows.
\end{proof}
For any indexing matrix $A$ we denote by $\underline{U}(A)$ the image of the element $\BBU(A)\in \mathfrak{Bl}_{n}$ via the isomorphism $\mathfrak{Bl}_{n}\rightarrow  \mathfrak{H}_{n}$ of Theorem \ref{theo-iso-algblobmon-hookblobmon}

\section{Abacus blob monoid}\label{ssec-abacus-blob-mon}
In this section we introduce the abacus blob monoid, which, roughly speaking, is the Martin-Saleur  abacus  monoid in which the arcs are  provided  by beads. This monoid is studied in diagrammatic terms, obtaining in particular the normal form of its elements and their cardinality, see Corollary \ref{coro-normal-form-framed-blob-mon} and Theorem \ref{theo-card-Fblob-mon-first}, respectively.
\begin{definition}
A  $d$-abacus  blob diagram or abacus blob diagram, is a  blobbed TL-diagram , whose arcs may have some beads (not filled circles). Beads can appear in any arc subject to the following rules:
    \begin{enumerate}
        \item[(i)] If an arc has no blobs, then it can contain at most $d-1$ beads  sliding freely on the arc.
        \item[(ii)] If an arc is decorated by a blob, then the blob will be considered as an obstruction, so the arc will be considered divided in two components. In this case, each component of the arc can contain at most $d-1$ beads that slide freely on the component.
\end{enumerate}
\end{definition}
 Given a   $d$-abacus  blob diagram $D^{\circ}$, we denote by  $D$ the diagram obtained by erasing  the beads  on it, will be called, the {\it base diagram} of  $D^{\circ}$.
\begin{figure}[ht]
$$
\begin{tikzpicture}[xscale=0.5,yscale=0.5]
\draw[thick] (1,0) to [out=90,in=90](4,0);
\draw[thick] (2,0) to [out=90,in=90](3,0);
\draw[thick] (5,0) to [out=90,in=90](6,0);
\draw[thick] (7,0) to [out=90,in=270](3,3);
\draw[thick] (8,0) to [out=90,in=270](8,3);
\draw[thick] (2,3) to [out=270,in=270](1,3);
\draw[thick] (4,3) to [out=270,in=270](7,3);
\draw[thick] (5,3) to [out=270,in=270](6,3);
\draw[thick] (9,0) to [out=90,in=270](9,3);

\draw[fill] (5,1.5) circle [radius=0.15];
\draw[fill] (2.5,0.9) circle [radius=0.15];
\draw[] (1.7,0.73) circle [radius=0.2];
\draw[] (5.5,0.28) circle [radius=0.2];
\draw[] (5.5,2.73) circle [radius=0.2];
\draw[] (6,1.2) circle [radius=0.2];
\draw[] (4,1.8) circle [radius=0.2];
\draw[] (3.3,2.2) circle [radius=0.2];
\draw[] (9,1.8) circle [radius=0.2];
\draw[] (9,1.2) circle [radius=0.2];
\draw[] (8,1.5) circle [radius=0.2];

\draw[thick] (1,-1) to (1,0);
\draw[thick] (2,-1) to (2,0);
\draw[thick] (3,-1) to (3,0);
\draw[thick] (4,-1) to (4,0);
\draw[thick] (5,-1) to (5,0);
\draw[thick] (6,-1) to (6,0);
\draw[thick] (7,-1) to (7,0);
\draw[thick] (8,-1) to (8,0);
\draw[thick] (9,-1) to (9,0);

\draw[thick] (1,3) to (1,4);
\draw[thick] (2,3) to (2,4);
\draw[thick] (3,3) to (3,4);
\draw[thick] (4,3) to (4,4);
\draw[thick] (5,3) to (5,4);
\draw[thick] (6,3) to (6,4);
\draw[thick] (7,3) to (7,4);
\draw[thick] (8,3) to (8,4);
\draw[thick] (9,3) to (9,4);

\end{tikzpicture} 
$$
\caption{The 3-abacus blob diagram $D^\circ$.}\label{ex-dframed-blob-diagram}
\end{figure}
Note that the base diagram of $D^{\circ}$ is the diagram $D$ of (\ref{eq-tikz-blobbed-diagram}).
Finally, we  declare that two abacus blob diagrams  are equal if  their base diagrams  are equal and the number of beads (modulo $d$) at each corresponding component of arcs are the same. This fact will be used explicitly or implicitly throughout the article.
The motivation for providing the arcs with beads comes from the so-called abacus monoid \cite[Section 3]{
AiJuPa2024}, which we will explain below.
Set $\ZZ/d\ZZ =\{0,1,\ldots , d-1\}$ the group of integers modulo $d$. In multiplicative notation, this group is written as  $C_d =\{1, z, z^2,\ldots , z^{d-1}\}$. The direct product group $C_d^n =C_d\times\cdots \times C_d $ can be presented as
\begin{equation}\label{PreCdn}
 C_d^n = \langle z_1, \ldots ,z_n ; z_iz_j = z_jz_i, \quad z_i^d=1\rangle.
\end{equation}
This group can be regarded as the framization of the the trivial subgroup
$\{1_{\mathfrak{C}_n}\}$ of $\mathfrak{C}_n$. In this diagrammatic context, $z_i^k$ means that line $i$ is framed with $k\in\ZZ/d\ZZ$ and the others lines framed with $0$. The $z_i^k$ can also be interpreted by putting $k$ counts on line $i$ and the others without counts.
The  handling in this last interpretation resembles the abacus, thus   we call $\mathrm{C}_d^n$ the abacus monoid and we denote $\mathfrak{z}_i^k$ instead of $z_i^k$ to distinguish the abacus interpretation from usual framing interpretation.
\begin{figure}[ht]
\begin{center}
 \begin{equation*}
   \begin{tikzpicture}[xscale=0.4,yscale=0.4]
\node[] at (-1,1.5) {$\mathfrak{z}_i^k=$};
\draw[thick] (1,0) to [out=90,in=270](1,3);
\draw[thick] (2,0) to [out=90,in=270](2,3);
\node[] at (3,1.5) {$\cdots$};
\draw[thick] (4,0) to [out=90,in=270](4,3);
\draw[thick] (5,0) to [out=90,in=270](5,3);
\draw[thick] (6,0) to [out=90,in=270](6,3);
\draw[thick] (7,0) to [out=90,in=270](7,3);
\node[] at (8,1.5) {$\cdots$};
\draw[thick] (9,0) to [out=90,in=270](9,3);
\draw[] (5,1.5) circle [radius=0.25];
\node[] at (5.5,1.6) {\small$k$};
\node[] at (9.5,1.4) {};
\end{tikzpicture}
\hspace{1.5cm}
\begin{tikzpicture}[xscale=0.4,yscale=0.4]
\node[] at (-1,1.5) {$z_i^k=$};
\draw[thick] (1,0) to [out=90,in=270](1,3);
\draw[thick] (2,0) to [out=90,in=270](2,3);
\node[] at (3,1.5) {$\cdots$};
\draw[thick] (4,0) to [out=90,in=270](4,3);
\draw[thick] (5,0) to [out=90,in=270](5,3);
\draw[thick] (6,0) to [out=90,in=270](6,3);
\draw[thick] (7,0) to [out=90,in=270](7,3);
\node[] at (8,1.5) {$\cdots$};
\draw[thick] (9,0) to [out=90,in=270](9,3);
\node[below] at (1,4) {\small$0$};
\node[below] at (2,4) {\small$0$};
\node[below] at (4,4) {\small$0$};
\node[below] at (5,4) {\small$k$};
\node[below] at (6,4) {\small$0$};
\node[below] at (7,4) {\small$0$};
\node[below] at (9,4) {\small$0$};
\node[] at (9.5,1.4) {};
\end{tikzpicture} 
 \end{equation*}
\caption{The diagrams of $\mathfrak{z}_i^k$ and $z_i^k$, respectively.} \label{Fig03}
\end{center}
\end{figure}
The symbol $\bigcirc k$ means $k$ beads.

\begin{definition}\label{def-abacus-blob-mon}
The abacus blob monoid $\mathfrak{Bl}_{d,n},$ is the monoid  formed  by
the $d$-abacus blob diagrams, on $n$ points, with the usual product by concatenation,  and  the  following rules:
\begin{enumerate}
\item[(i)] Loops, containing or not a decoration (blob and/or beads) will be erased from the diagram.
\item[(ii)] The following local relations are imposed. \begin{equation*}
\begin{tikzpicture}[xscale=0.5,yscale=0.5]
\draw[thick] (1,0) to [out=90,in=270](1,3);
\draw[fill] (1,1) circle [radius=0.15];
\draw[fill] (1,2) circle [radius=0.15];
\node[] at (2,1.5) {$=$};
\draw[thick] (3,0) to [out=90,in=270](3,3);
\draw[fill] (3,1.5) circle [radius=0.15];
\end{tikzpicture}
{\hspace{3cm}}
\begin{tikzpicture}[xscale=0.4,yscale=0.5]
\draw[thick] (6,0) to [out=90,in=270](6,3);
\draw[] (6,1.5) circle [radius=0.25];
\node[] at (6.6,1.5) {$k$};
\draw[fill] (6,2.5) circle [radius=0.15];
\draw[fill] (6,0.5) circle [radius=0.15];

\node[] at (7.8,1.5) {$=$};

\draw[thick] (9,0) to [out=90,in=270](9,3);
\draw[fill] (9,1.5) circle [radius=0.15];

\end{tikzpicture} 
\end{equation*}
\end{enumerate}
\end{definition}

\begin{remark}\rm
Since any blob TL-diagram can be regarded as an abacus blob TL-diagram (with no beads), it follows that $\mathfrak{Bl}_{n}$ can be regarded as a submonoid of $\mathfrak{Bl}_{d,n}$ for all $d, n$. Also, $\mathfrak{Bl}_{d,n} $ has as submonoid the abacus monoid $C_d^n$. Moreover, the monoid $\mathfrak{Bl}_{d,n}$ is
generated by $\mathfrak{Bl}_{n}\cup C_d^n$.
\end{remark}
Define now the elements $\mathfrak{z}_i, \mathfrak{u}_i\in \mathfrak{Bl}_{d,n} $ as show the following figure.
\begin{figure}[ht]
\begin{equation*}
    \begin{tikzpicture}[xscale=0.3,yscale=0.3]
\node[] at (-1,1.5) {$\mathfrak{z}_i=$};
\draw[thick] (1,0) to [out=90,in=270](1,3);
\draw[thick] (2,0) to [out=90,in=270](2,3);
\node[] at (3,1.5) {$\cdots$};
\draw[thick] (4,0) to [out=90,in=270](4,3);
\draw[thick] (5,0) to [out=90,in=270](5,3);
\draw[thick] (6,0) to [out=90,in=270](6,3);
\draw[thick] (7,0) to [out=90,in=270](7,3);
\node[] at (8,1.5) {$\cdots$};
\draw[thick] (9,0) to [out=90,in=270](9,3);
\node[below] at (5,0) {$i$};
\node[below] at (1,0) {$1$};
\node[below] at (9,0) {$n$};
\draw[] (5,1.5) circle [radius=0.25];
\node[] at (9.5,1.4) {};
\end{tikzpicture}
\qquad
\begin{tikzpicture}[xscale=0.3,yscale=0.3]
\node[] at (-1,1.5) {$\mathfrak{u}_0=$};
\draw[thick] (1,0) to [out=90,in=270](1,3);
\draw[fill] (1,1.5) circle [radius=0.15];
\draw[thick] (2,0) to [out=90,in=270](2,3);
\node[] at (3,1.5) {$\cdots$};
\draw[thick] (4,0) to [out=90,in=270](4,3);
\draw[thick] (5,0) to [out=90,in=270](5,3);
\draw[thick] (6,0) to [out=90,in=270](6,3);
\draw[thick] (7,0) to [out=90,in=270](7,3);
\node[] at (8,1.5) {$\cdots$};
\draw[thick] (9,0) to [out=90,in=270](9,3);
\node[below] at (1,0) {$1$};
\node[below] at (9,0) {$n$};
\node[] at (9.5,1.4) {};

\end{tikzpicture}
\qquad
\begin{tikzpicture}[xscale=0.3,yscale=0.3]
\node[] at (-1,1.5) {$\mathfrak{u}_i=$};
\draw[thick] (1,0) to [out=90,in=270](1,3);
\draw[thick] (2,0) to [out=90,in=270](2,3);
\node[] at (3,1.5) {$\cdots$};
\draw[thick] (4,0) to [out=90,in=270](4,3);
\draw[thick] (5,0) to [out=90,in=90](6,0);
\draw[thick] (5,3) to [out=270,in=270](6,3);
\draw[thick] (7,0) to [out=90,in=270](7,3);
\node[] at (8,1.5) {$\cdots$};
\draw[thick] (9,0) to [out=90,in=270](9,3);
\node[below] at (5,0) {$i$};
\node[below] at (1,0) {$1$};
\node[below] at (9,0) {$n$};
\node[] at (9.5,1.4) {};

\end{tikzpicture} 
\end{equation*}
\caption{The diagram of the generators $\mathfrak{z}_i$, $\mathfrak{u}_0$ and $\mathfrak{u}_i$, respectively.}\label{eq-blobbed-generators-Frblob-mon}
\end{figure}

\begin{lemma}\label{lemma-pre-presentation-frblob-digrams}The elements $\mathfrak{z}_i$'s, and $\mathfrak{u}_i$'s $\in \mathfrak{Bl}_{d,n}$ satisfy, respectively,  the
defining relations of  $C_d^n$ and $\mathfrak{Bl}_{n,}$, together with  the following relations:
\begin{align}
  \mathfrak{z}_i\mathfrak{u}_i = \mathfrak{z}_{i+1}\mathfrak{u}_i, &\qquad \mathfrak{u}_i \mathfrak{z}_i = \mathfrak{u}_i\mathfrak{z}_{i+1},\\
\mathfrak{u}_0 \mathfrak{z}_i =\mathfrak{z}_i \mathfrak{u}_0 \quad \text{if $i\not=1$,}&\qquad \mathfrak{u}_i \mathfrak{z}_j =\mathfrak{z}_j \mathfrak{u}_i \quad \text{if $j\not=i,i+1$,}\\
\mathfrak{u}_i \mathfrak{z}_i^k \mathfrak{u}_i=\mathfrak{u}_i, & \qquad \mathfrak{u}_1 \mathfrak{z}_1^k \mathfrak{u}_0 \mathfrak{u}_1=\mathfrak{u}_1, \qquad \mathfrak{u}_1 \mathfrak{u}_0 \mathfrak{z}_1^k \mathfrak{u}_1=\mathfrak{u}_1.
\end{align}
\end{lemma}
\begin{proof}
The proof follows from a direct check, so we will only show, as an example, that  $\mathfrak{u}_1\mathfrak{u}_0\mathfrak{z}_1^k\mathfrak{u}_1=\mathfrak{u} _1$. We have:
 \begin{equation*}
    \begin{tikzpicture}[xscale=0.6,yscale=0.3]
\node[] at (2,0.5) {$\mathfrak{u}_1\mathfrak{u}_0\mathfrak{z}_1^{k}\mathfrak{u}_1=$};
\draw[thick] (7,0) to [out=90,in=270](7,3);
\draw[thick] (5,-1) to [out=90,in=270](5,2);
\draw[thick] (6,-1) to [out=90,in=270](6,2);
\draw[thick] (5,2) to [out=90,in=90](6,2);
\draw[thick] (5,3) to [out=270,in=270](6,3);
\node[] at (8,0.5) {$\cdots$};
\draw[thick] (9,-2) to [out=90,in=270](9,3);
\draw[] (5,-0.3) circle [radius=0.25];
\node[] at (5.5,-0.2) {\small $k$};
\draw[fill] (5,0.9) circle [radius=0.15];
\draw[thick] (7,-2) to [out=90,in=270](7,3);
\draw[thick] (5,-2) to [out=90,in=90](6,-2);
\draw[thick] (5,-1) to [out=270,in=270](6,-1);
\node[] at (10,0.5) {$=$};

\draw[thick] (11,2.5) to [out=90,in=270](11,3);
\draw[thick] (12,2.5) to [out=90,in=270](12,3);
\draw[thick] (11,2.5) to [out=270,in=270](12,2.5);
\node[] at (14,0.5) {$\cdots$};
\draw[thick] (15,-2) to [out=90,in=270](15,3);
\draw[thick] (13,-2) to [out=90,in=270](13,3);
\draw[thick] (11,-2) to [out=90,in=270](11,-1.5);
\draw[thick] (12,-2) to [out=90,in=270](12,-1.5);
\draw[thick] (11,-1.5) to [out=90,in=90](12,-1.5);
\node[] at (16.3,0.5) {$=\,\, \mathfrak{u}_1$};
\end{tikzpicture} 
\end{equation*}
\end{proof}

\begin{notation}\label{Not1}
 We call  exponent vector to  $\alphan:=(\alpha_1,\dots,\alpha_n)\in (\ZZ/d\ZZ)^n$ and we  set   $\zetan^{\alphan}:=\mathfrak{z}_1^{\alpha_1}\cdots \mathfrak{z}_{n}^{\alpha_n}\in C_d^n$.
\end{notation}
\begin{lemma}\label{lemma-diag-normal-form-FrBlob-1}
   Each element $D^{\circ}\in \mathfrak{Bl}_{d,n}$ can be written in  the  form $\zetan^{\alphan}$ or $\zetan^{\alphan}D\zetan^{\bbeta},$ where  $D$ is  the base diagram  of $D^{\circ}$. In particular, $\mathfrak{Bl}_{d,n}$ is generated by $\mathfrak{z}_1,\dots,\mathfrak{z}_n,  \mathfrak{u}_0,\mathfrak{u}_1,\dots,\mathfrak{u}_{n-1}$.
\end{lemma}
\begin{proof}
The proof follows directly from the facts that the beads slide freely on the arcs,  components  of an arc with  blob and the rules of  Definition \ref{def-abacus-blob-mon}.
\begin{figure}[ht]
\begin{equation*}
    \begin{tikzpicture}[xscale=0.5,yscale=0.5]
\draw[thick] (1,0) to [out=90,in=90](4,0);
\draw[thick] (2,0) to [out=90,in=90](3,0);
\draw[thick] (5,0) to [out=90,in=90](6,0);
\draw[thick] (7,0) to [out=90,in=270](3,3);
\draw[thick] (8,0) to [out=90,in=270](8,3);
\draw[thick] (2,3) to [out=270,in=270](1,3);
\draw[thick] (4,3) to [out=270,in=270](7,3);
\draw[thick] (5,3) to [out=270,in=270](6,3);
\draw[thick] (9,0) to [out=90,in=270](9,3);

\draw[fill] (5,1.5) circle [radius=0.15];
\draw[fill] (2.5,0.9) circle [radius=0.15];
\draw[] (1.7,0.73) circle [radius=0.2];
\draw[] (5.5,0.28) circle [radius=0.2];
\draw[] (5.5,2.73) circle [radius=0.2];
\draw[] (6,1.2) circle [radius=0.2];
\draw[] (4,1.8) circle [radius=0.2];
\draw[] (3.3,2.2) circle [radius=0.2];
\draw[] (9,1.8) circle [radius=0.2];
\draw[] (9,1.2) circle [radius=0.2];
\draw[] (8,1.5) circle [radius=0.2];

\draw[thick] (1,-1) to (1,0);
\draw[thick] (2,-1) to (2,0);
\draw[thick] (3,-1) to (3,0);
\draw[thick] (4,-1) to (4,0);
\draw[thick] (5,-1) to (5,0);
\draw[thick] (6,-1) to (6,0);
\draw[thick] (7,-1) to (7,0);
\draw[thick] (8,-1) to (8,0);
\draw[thick] (9,-1) to (9,0);

\draw[thick] (1,3) to (1,4);
\draw[thick] (2,3) to (2,4);
\draw[thick] (3,3) to (3,4);
\draw[thick] (4,3) to (4,4);
\draw[thick] (5,3) to (5,4);
\draw[thick] (6,3) to (6,4);
\draw[thick] (7,3) to (7,4);
\draw[thick] (8,3) to (8,4);
\draw[thick] (9,3) to (9,4);

\end{tikzpicture}\quad \begin{tikzpicture}[xscale=0.5,yscale=0.5]
\node[] at (-1,1.5) {$=$};
\draw[thick] (1,0) to [out=90,in=90](4,0);
\draw[thick] (2,0) to [out=90,in=90](3,0);
\draw[thick] (5,0) to [out=90,in=90](6,0);
\draw[thick] (7,0) to [out=90,in=270](3,3);
\draw[thick] (8,0) to [out=90,in=270](8,3);
\draw[thick] (2,3) to [out=270,in=270](1,3);
\draw[thick] (4,3) to [out=270,in=270](7,3);
\draw[thick] (5,3) to [out=270,in=270](6,3);
\draw[thick] (9,0) to [out=90,in=270](9,3);
\draw[thick] (1,-1) to (1,0);
\draw[thick] (2,-1) to (2,0);
\draw[thick] (3,-1) to (3,0);
\draw[thick] (4,-1) to (4,0);
\draw[thick] (5,-1) to (5,0);
\draw[thick] (6,-1) to (6,0);
\draw[thick] (7,-1) to (7,0);
\draw[thick] (8,-1) to (8,0);
\draw[thick] (9,-1) to (9,0);

\draw[thick] (1,3) to (1,4);
\draw[thick] (2,3) to (2,4);
\draw[thick] (3,3) to (3,4);
\draw[thick] (4,3) to (4,4);
\draw[thick] (5,3) to (5,4);
\draw[thick] (6,3) to (6,4);
\draw[thick] (7,3) to (7,4);
\draw[thick] (8,3) to (8,4);
\draw[thick] (9,3) to (9,4);

\draw[dashed] (0.8,3) to (9.2,3);
\draw[dashed] (0.8,0) to (9.2,0);

\draw[fill] (5,1.5) circle [radius=0.15];
\draw[fill] (2.5,0.9) circle [radius=0.15];
\draw[] (1,-0.3) circle [radius=0.2];
\draw[] (5,-0.3) circle [radius=0.2];
\draw[] (5,3.2) circle [radius=0.2];
\draw[] (7,-0.7) circle [radius=0.2];
\draw[] (7,-0.2) circle [radius=0.2];
\draw[] (3,3.8) circle [radius=0.2];
\draw[] (3,3.3) circle [radius=0.2];
\draw[] (9,3.2) circle [radius=0.2];
\draw[] (9,3.7) circle [radius=0.2];
\draw[] (8,3.7) circle [radius=0.2];

\end{tikzpicture} 
\end{equation*}
\caption{Example for the digramm of Fig. \ref{ex-dframed-blob-diagram}.}\label{eq-tikz-quasi-normalized-pearls+blob-diag}
\end{figure}
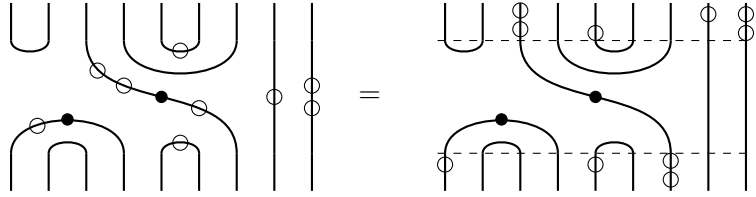
For example if $D^{\circ},$ is the element described in Fig. \ref{ex-dframed-blob-diagram}, then  $D^{\circ}=\mathfrak{z}_3^2\,\mathfrak{z}_5\,\mathfrak{z}_8\,\mathfrak{z}_9^2\,D\,\mathfrak{z}_1\,\mathfrak{z}_5\,\mathfrak{z}_7^2$ as we can see in Fig 10. The last assertion follows from the fact that the base diagram $D$ belongs to the submodule $\mathfrak{Bl}_{n}$ of $\mathfrak{Bl}_{d,n}$
\end{proof}
We now proceed to construct a normal form for the elements of $\mathfrak{Bl}_{d,n}$, which for $d=1$  is that of
Corollary \ref{NFBlob}. We begin by recalling that there exists a one-to-one correspondence $A\mapsto \BBU(A)$ between the set of all indexing matrices and the monoid $\mathfrak{Bl}_{n}.$ We will use this correspondence and the following five lemmas to find a normal form for the elements of $\mathfrak{Bl}_{d,n} $.

\begin{lemma}\label{lemma-normal-form-FrBlob-1}
\begin{enumerate} For  $D^{\circ}\in \mathfrak{Bl}_{d,n} $, we have:
    \item[(i)] $D^{\circ} =\zetan^{\alphan}=\zetan^{\alphan}\BBU\Big[\!\!\begin{array}{c}\infty\\\infty\end{array}\!\!\Big]$ or  $D^{\circ}=\zetan^{\alphan}\BBU(A)\zetan^{\bbeta}$, where  $\alphan,\bbeta$ are exponent vectors and  $A$ is a non-degenerate indexing matrix.
   \item[(ii)] If $\zetan^{\alphan}\BBU(A)\zetan^{\bbeta}=\zetan^{\bgamma}\BBU(A')\zetan^{\kappan}$, with  $\alphan,\bbeta,\bgamma,\kappan$  are exponent vectors,  and  $A,A'$ are indexing matrices, then  $A=A'.$
    \end{enumerate}
\end{lemma}

\begin{proof}
The first statement is a direct consequence of Lemma \ref{lemma-diag-normal-form-FrBlob-1}.
The second statement follows from the fact that if two $d$-abacus blob diagrams are equal, then their corresponding basis diagrams are equal.
\end{proof}
The proof of the following two lemmas follow directly by drawing the products implied in the claims of the lemmas and again the fact that two $d$-abacus blob diagrams are equal, if and only if their corresponding basis diagrams are equal and the number of beads modulo $d$ is the same on each component. For example, claim (iv) of Lemma \ref{zA0U} follows directly from the following diagram.
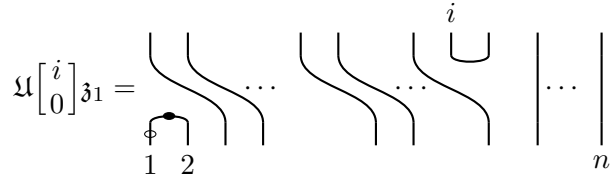
\begin{equation*}
    \begin{tikzpicture}[xscale=0.5,yscale=0.3]
\node[] at (-1,1.5) {$\BBU\Big[\!\!\begin{array}{c}i \\0\end{array}\!\!\Big]\mathfrak{z}_1=$};
\draw[thick] (1,0) to [out=90,in=90](2,0);
\draw[fill] (1.5,0.3) circle [radius=0.15];
\draw[thick] (3,0) to [out=90,in=270](1,3);
\draw[thick] (4,0) to [out=90,in=270](2,3);
\node[] at (4,1.5) {$\cdots$};
\draw[thick] (5,3) to [out=270,in=90](7,0);
\draw[thick] (6,3) to [out=270,in=90](8,0);
\draw[thick] (8,3) to [out=270,in=90](10,0);
\draw[thick] (11.3,3) to [out=270,in=90](11.3,0);
\node[] at (8,1.5) {$\cdots$};
\draw[thick] (10,3) to [out=270,in=270](9,3);
\node[] at (12,1.5) {$\cdots$};
\draw[thick] (13,3) to [out=270,in=90](13,0);
\node[below] at (2,-1) {$2$};
\node[below] at (1,-1) {$1$};
\node[above] at (9,4) {$i$};
\node[below] at (13,-1) {$n$};


\draw[thick](1,3) to (1,4);
\draw[thick](2,3) to (2,4);
\draw[thick](5,3) to (5,4);
\draw[thick](6,3) to (6,4);
\draw[thick](8,3) to (8,4);
\draw[thick](9,3) to (9,4);
\draw[thick](10,3) to (10,4);
\draw[thick](11.3,3) to (11.3,4);
\draw[thick](13,3) to (13,4);
\draw[thick](1,0) to (1,-1);
\draw[thick](2,0) to (2,-1);
\draw[thick](4,0) to (4,-1);
\draw[thick](3,0) to (3,-1);
\draw[thick](7,0) to (7,-1);
\draw[thick](8,0) to (8,-1);
\draw[thick](10,0) to (10,-1);
\draw[thick](11.3,0) to (11.3,-1);
\draw[thick](13,0) to (13,-1);
\draw[] (1,-0.5) circle [radius=0.15];
\end{tikzpicture} 
\end{equation*}

\begin{lemma}\label{zU} Set $A =\Big[\!\!\begin{array}{c}i\\j\end{array}\!\!\Big]$, where $i,j\in [1,n-1]$ and $j<i$.  Then:
\begin{enumerate}
\item[(i)]  $\mathfrak{z}_k\BBU(A) =  \BBU(A)\mathfrak{z}_k$ if and only if $k<j$ or  $k>i+1.$
\item[(ii)] $\mathfrak{z}_k\BBU(A) =  \BBU(A)\mathfrak{z}_{k+2}$ if and only if $k\in [j,i-1]$.
\item[(iii)] $\mathfrak{z}_k\BBU(A)=\mathfrak{z}_p\BBU(A)$ with $k\neq p$ if and only if
$k,p\in\{i,i+1\}$.
\item[(iv)] $\BBU(A)\mathfrak{z}_k=\BBU(A)\mathfrak{z}_p$ with $k\neq p$ if and only if
$k,p\in\{j,j+1\}$.
\end{enumerate}
\end{lemma}

\begin{lemma}\label{zA0U}
Set $A =\Big[\!\!\begin{array}{c}i\\0\end{array}\!\!\Big]$, where $i\in [1,n-1]$.  Then following relations hold in $\mathfrak{Bl}_{d,n} $.
\begin{enumerate}
\item[(i)] $\mathfrak{z}_k\BBU(A)=  \BBU(A)\mathfrak{z}_k $ if and only if $k>i+1$.
\item[(ii)] $\mathfrak{z}_k\BBU(A)=  \BBU(A)\mathfrak{z}_{k+2} $ if and only if $k\in[1,i-1]$.
\item[(iii)] $\mathfrak{z}_k\BBU(A)=\mathfrak{z}_p\BBU(A),$ with $k\neq p$ if and only if $k,p\in\{i,i+1\}.$
\item [(iv)] $\BBU(A)\mathfrak{z}_k=\BBU(A)\mathfrak{z}_p$ implies $k=p,$ if and only if, $k=1$ or $k=2.$
\end{enumerate}
\end{lemma}
\begin{proposition}\label{prop-interaction-z-BBU}
If $j\in [1,n]$ and $A=[B|C]$ is a non-degenerate indexing matrix, then in the product $\mathfrak{z}_j\BBU(A)$ only one of the following cases occurs.
    \begin{enumerate}
          \item[(i)] \label{sub-critical-case1-coro-interaction-z-BBU} For $j=1$ and $b_1=0,$ we have $\mathfrak{z}_j\BBU(A)\neq \BBU(A)\mathfrak{z}_i$ for all $i$. Furthermore,  $\mathfrak{z}_j\BBU(A)=\mathfrak{z}_i\BBU(A)$ if and only if $i=1$.
         \item[(ii)] \label{critical-case1-coro-interaction-z-BBU}  There is  $s \in [1,r]$ and  $k\in \{1,2\}$ such that:
        \begin{equation*}
        \mathfrak{z}_j\BBU(A)=\BBU\left[\begin{matrix}
                b_1\\
                0
            \end{matrix}\right]\cdots \BBU\left[\begin{matrix}
                b_s\\
                0
    \end{matrix}\right]\mathfrak{z}_k\BBU\left[\begin{matrix}
         b_{s+1}\\
                0
            \end{matrix}\right]\cdots \BBU\left[\begin{matrix}
                b_r\\
                0
            \end{matrix}\right]\BBU(C).
        \end{equation*}
          \item[(iii)] \label{Not-critical-case1-coro-interaction-z-BBU} There is $i\in[1,n] $ s.t.  $i\neq j$ and $\mathfrak{z}_j\BBU(A)=\mathfrak{z}_i\BBU(A).$
         \item[(iv)] \label{Not2-critical-case1-coro-interaction-z-BBU} There is $k\in[1, n]$ s.t.  $\mathfrak{z}_j\BBU(A)=\BBU(A)\mathfrak{z}_k.$
    \end{enumerate}
Furthermore, in  the product $\BBU(A)\mathfrak{z}_j$ only one of the following situations occurs:
    \begin{enumerate}
        \item[(v)] \label{critical-case2-coro-interaction-z-BBU}  There is $s\in[1,r]$ and $k\in \{1,2\}$ such that:
        \begin{equation}\label{FormCase(v)}
        \BBU(A)\mathfrak{z}_j=\BBU\left[\begin{matrix}
                b_{1}\\
                0
    \end{matrix}\right]\cdots \BBU\left[\begin{matrix}
                b_s\\
                0
    \end{matrix}\right]\mathfrak{z}_k\BBU\left[\begin{matrix}
                b_{s+1}\\
                0
    \end{matrix}\right]\cdots \BBU\left[\begin{matrix}
                b_{r}\\
                0
    \end{matrix}\right]\BBU(C).
        \end{equation}
          \item[(vi)] \label{Not-critical-case2-coro-interaction-z-BBU} There is $i\in[1,n]$ such that $i\neq j$ and $\BBU(A)\mathfrak{z}_j=\BBU(A)\mathfrak{z}_i.$
         \item[(vii)] \label{Not2-critical-case2-coro-interaction-z-BBU} There is  $k\in[1,n]$ such that $\BBU(A)\mathfrak{z}_j=\mathfrak{z}_k\BBU(A).$
    \end{enumerate}
\end{proposition}

\begin{proof}Suppose we are not in  case (i).
The proof is by induction on the length  $l$ of $A=[a_{i,j}]$. When  $l=1$, the proof follows from Lemma \ref{zU} and Lemma \ref{zA0U}.  Suppose the statement is true for $l'\in [1, l-1]$. We will show how to proceed only with the product
$\mathfrak{z}_j\BBU(A)$ since the other product is proven in a completely analogous way. We can write $\mathfrak{z}_j\BBU(A) = \mathfrak{z}_j\BBU(A')\BBU(A'')$, where $A'$ is obtained by deleting the last column of $A$ and $A''$ is the last column of $A$. Note that cases (i) to (iii) resulting from applying the inductive hypothesis to $\mathfrak{z}_j\BBU(A')$ are inherited towards $\mathfrak{z}_j\BBU(A)$. Hence, it remains to study only when applying the induction hypothesis on $\mathfrak{z}_j\BBU(A')$ we obtain case (iv), which says that there only  one $k$ such that $\mathfrak{z}_j\BBU(A')=\BBU(A')\mathfrak{z}_k.$ Then we have:
        \begin{equation*}
           \mathfrak{z}_j\BBU(A)=\BBU(A')\mathfrak{z}_k\BBU\left[\begin{matrix}
                a_{1l}\\
                a_{2l}
    \end{matrix}\right].
        \end{equation*}
        For the product $\mathfrak{z}_k\BBU\left[\begin{matrix}
                a_{1l}\\
                a_{2l}
    \end{matrix}\right]$ we distinguish between $k\in \{a_{1l},a_{1l}+1\} $ or not.
     If $k\notin \{a_{1l},a_{1l}+1\},$ then by  Lemma \ref{zU}, there is $i$ such that $\mathfrak{z}_k\BBU\left[\begin{matrix}
                a_{1l}\\
                a_{2l}
    \end{matrix}\right]=\BBU\left[\begin{matrix}
                a_{1l}\\
                a_{2l}
    \end{matrix}\right]\mathfrak{z}_i$ and therefore $ \mathfrak{z}_j\BBU(A)=\BBU(A) \mathfrak{z}_i$; we thus arrive at case (iv) of the proposition. If  $k\in \{a_{1l},a_{1l}+1\},$ we can assume without lost of generality that $k=a_{1l}+1.$ If $i=a_{1l},$ then we have $\mathfrak{z}_k\BBU\left[\begin{matrix}
                a_{1l}\\
                a_{2l}
    \end{matrix}\right]=\mathfrak{z}_i\BBU\left[\begin{matrix}
                a_{1l}\\
                a_{2l}
    \end{matrix}\right]$, with $i\neq k,$ so
            \begin{equation*}
            \mathfrak{z}_j\BBU(A)=\BBU(A')\mathfrak{z}_i\BBU\left[\begin{matrix}
                a_{1l}\\
                a_{2l}
    \end{matrix}\right].
            \end{equation*}
 Using the inductive hypothesis on  $\BBU(A')\mathfrak{z}_i$, we derive the analysis to  the exclusive case  (v) to case (vii) of the proposition.

{\it Case (v).} In this case, there is  $s\in [1,r]$ and $k\in\{1,2\},$ such that
$\BBU(A')\mathfrak{z}_i$ is in the form (\ref{FormCase(v)}), so we deduce:
\begin{equation*}
    \mathfrak{z}_j\BBU(A)=\BBU(A')\mathfrak{z}_i\BBU\left[\begin{matrix}
                a_{1l}\\
                a_{2l}
    \end{matrix}\right]=\BBU\left[\begin{matrix}
                b_1\\
                0
    \end{matrix}\right]\cdots\BBU\left[\begin{matrix}
                b_s\\
               0
    \end{matrix}\right]\mathfrak{z}_k\cdots\BBU\left[\begin{matrix}
                b_r\\
                0
    \end{matrix}\right]\BBU(C),
\end{equation*} that is the case (ii) for $\mathfrak{z}_j\BBU(A)$.

  {\it Case (vi).} In this case there is  $p\neq i=a_{il}$ such that $\BBU(A')\mathfrak{z}_i=\BBU(A')\mathfrak{z}_p$. Thus the inductive hypothesis impies that $p\neq k=a_{1l}+1$ since $\mathfrak{z}_j\BBU(A')=\BBU(A')\mathfrak{z}_k,$. Hence we have that there is a $q$ such that $\mathfrak{z}_p\BBU\left[\begin{matrix}
                a_{1l}\\
                a_{2l}
    \end{matrix}\right]=\BBU\left[\begin{matrix}
                a_{1l}\\
                a_{2l}
    \end{matrix}\right]\mathfrak{z}_q.$ Therefore
            \begin{equation*}
            \mathfrak{z}_j\BBU(A)=\BBU(A')\left(\mathfrak{z}_p\BBU\left[\begin{matrix}
                a_{1l}\\
                a_{2l}
    \end{matrix}\right]\right)=\BBU(A')\left(\BBU\left[\begin{matrix}
                a_{1l}\\
                a_{2l}
    \end{matrix}\right]\mathfrak{z}_q\right)=\BBU(A)\mathfrak{z}_q,
            \end{equation*}
            that is, we are in case (iv) for $\mathfrak{z}_j\BBU(A).$

{\it Case (vii).}       In this case there is  $p\in[1,n]$ such that $\BBU(A')\mathfrak{z}_i=\mathfrak{z}_p\BBU(A').$ Then
                \begin{equation*}
            \mathfrak{z}_j\BBU(A)=\BBU(A')\mathfrak{z}_i\BBU\left[\begin{matrix}
                a_{1l}\\
                a_{2l}
    \end{matrix}\right]=\mathfrak{z}_p\BBU(A).
                \end{equation*} Note that $p\neq j$ since $\mathfrak{z}_j\BBU(A')=\BBU(A')\mathfrak{z}_k$ and by induction applied on $A',$ we have for any $q\notin\{j,k\}$ that $\mathfrak{z}_q\BBU(A')\neq \BBU(A')\mathfrak{z}_k$ and $\mathfrak{z}_j\BBU(A')\neq \BBU(A')\mathfrak{z}_q.$ Therefore we are in case (iii) for $\mathfrak{z}_j\BBU(A).$

\end{proof}
Proposition \ref{prop-interaction-z-BBU}  allows the definition of certain matrices $L(A)$ and $R(A)$, which are defined below. As usual, from now on $M^t$  denotes the transpose of the matrix $M$.
\begin{definition}\label{def-left-right-Z-matrices}
Let $A=[B|C]$ be a degenerate indexing matrix, we define the  matrices  $L(A)$ and $R(A)$ of size   $3\times n$, as follows.
\begin{enumerate}
\item[(i)]Depending on the product $\mathfrak{z}_j\BBU(A),$ the $j$-th column of $L(A)$ is equal to:
\begin{enumerate}
\item  $\big[j\,\,\, j \,\,\, 0 \big]^t$  if we are in case (i) or (ii) of Proposition \ref{prop-interaction-z-BBU}.
\item $\big[ i\,\,\, j\,\,\, 0\big]^t$ if we are in the case (iii) of Proposition \ref{prop-interaction-z-BBU}.
\item $\big[ j\,\,\,j \,\,\, k \big]^t$ if we are in the case (iv) of Proposition \ref{prop-interaction-z-BBU}.
\end{enumerate}
\item [(ii)]  Depending on the product $\BBU(A)\mathfrak{z}_j,$ the $j$-th column of $R(A)$ is equal to:
\begin{enumerate}
\item  $\left[0\,\, j \,\, j\right]^t$ if we are in the case (v) of Proposition \ref{prop-interaction-z-BBU}.
\item $\left[0\,\, j \,\, i\right]^t$ if we are in the case (vi) of Proposition \ref{prop-interaction-z-BBU}.
\item $\left[k\,\, j \,\, 0\right]^t$ if we are in the case (vii) of Proposition \ref{prop-interaction-z-BBU}.
\end{enumerate}
\end{enumerate}
\end{definition}

\begin{example}\label{ex-left-right-Z-matrices-V0}
Let $n=7$ and consider the indexing matrix
    \begin{equation*}
        A=\left[\begin{matrix}
            1 & 2 & 3 & 6 \\
            0 & 1 & 2 & 4
        \end{matrix}\right].
    \end{equation*}
Lemmas \ref{zU} and \ref{zA0U} say:
 \item \begin{equation*}
            \mathfrak{z}_3\BBU(A)=\BBU\left[\begin{matrix}
                1\\
                0
            \end{matrix}\right]\left(\mathfrak{z}_3\BBU\left[\begin{matrix}
                2\\
                1
            \end{matrix}\right]\right)\BBU\left[\begin{matrix}
                3&6\\
                2&4
            \end{matrix}\right]=\left(\BBU\left[\begin{matrix}
                1\\
                0
            \end{matrix}\right]\mathfrak{z}_2\right)\BBU\left[\begin{matrix}
                2\\
                1
            \end{matrix}\right]\BBU\left[\begin{matrix}
                3&6\\
                2&4
            \end{matrix}\right].
        \end{equation*}
        Which in diagrammatic terms look like:
        
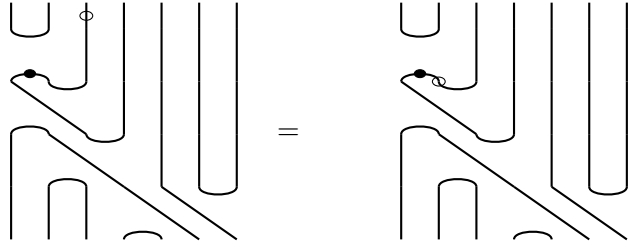
\begin{equation*}
            \begin{tikzpicture}[xscale=0.5,yscale=0.35]
\draw[thick] (1,4) to [out=270,in=270](2,4);
\draw[thick] (1,2) to [out=90,in=90](2,2);
\draw[fill] (1.5,2.3) circle [radius=0.15];
\draw[thick] (3,4) to (3,2);
\draw[thick] (4,4) to (4,2);
\draw[thick] (5,4) to (5,2);
\draw[thick] (6,4) to (6,2);
\draw[thick] (7,4) to (7,2);
\draw[] (3,4.5) circle [radius=0.17];
\draw[thick] (1,4) to (1,5);
\draw[thick] (2,4) to (2,5);
\draw[thick] (3,4) to (3,5);
\draw[thick] (4,4) to (4,5);
\draw[thick] (5,4) to (5,5);
\draw[thick] (6,4) to (6,5);
\draw[thick] (7,4) to (7,5);
\draw[thick] (2,2) to [out=270,in=270](3,2);
\draw[thick] (1,0) to [out=90,in=90](2,0);
\draw[thick] (1,2) to (3,0);
\draw[thick] (4,2) to (4,0);
\draw[thick] (5,2) to (5,0);
\draw[thick] (6,2) to (6,0);
\draw[thick] (7,2) to (7,0);
\draw[thick] (3,0) to [out=270,in=270](4,0);
\draw[thick] (2,-2) to [out=90,in=90](3,-2);
\draw[thick] (1,-2) to (1,0);
\draw[thick] (2,0) to (4,-2);
\draw[thick] (5,-2) to (5,0);
\draw[thick] (6,-2) to (6,0);
\draw[thick] (7,-2) to (7,0);

\draw[thick] (6,-2) to [out=270,in=270](7,-2);
\draw[thick] (4,-4) to [out=90,in=90](5,-4);
\draw[thick] (1,-2) to (1,-4);
\draw[thick] (2,-2) to (2,-4);
\draw[thick] (3,-2) to (3,-4);
\draw[thick] (4,-2) to (6,-4);
\draw[thick] (5,-2) to (7,-4);
\end{tikzpicture} \quad\begin{tikzpicture}[xscale=0.5,yscale=0.35]
\node[] at (-2,0) {$=$};
\draw[thick] (1,4) to [out=270,in=270](2,4);
\draw[thick] (1,2) to [out=90,in=90](2,2);
\draw[fill] (1.5,2.3) circle [radius=0.15];
\draw[thick] (3,4) to (3,2);
\draw[thick] (4,4) to (4,2);
\draw[thick] (5,4) to (5,2);
\draw[thick] (6,4) to (6,2);
\draw[thick] (7,4) to (7,2);
\draw[] (2,2) circle [radius=0.17];
\draw[thick] (1,4) to (1,5);
\draw[thick] (2,4) to (2,5);
\draw[thick] (3,4) to (3,5);
\draw[thick] (4,4) to (4,5);
\draw[thick] (5,4) to (5,5);
\draw[thick] (6,4) to (6,5);
\draw[thick] (7,4) to (7,5);
\draw[thick] (2,2) to [out=270,in=270](3,2);
\draw[thick] (1,0) to [out=90,in=90](2,0);
\draw[thick] (1,2) to (3,0);
\draw[thick] (4,2) to (4,0);
\draw[thick] (5,2) to (5,0);
\draw[thick] (6,2) to (6,0);
\draw[thick] (7,2) to (7,0);
\draw[thick] (3,0) to [out=270,in=270](4,0);
\draw[thick] (2,-2) to [out=90,in=90](3,-2);
\draw[thick] (1,-2) to (1,0);
\draw[thick] (2,0) to (4,-2);
\draw[thick] (5,-2) to (5,0);
\draw[thick] (6,-2) to (6,0);
\draw[thick] (7,-2) to (7,0);

\draw[thick] (6,-2) to [out=270,in=270](7,-2);
\draw[thick] (4,-4) to [out=90,in=90](5,-4);
\draw[thick] (1,-2) to (1,-4);
\draw[thick] (2,-2) to (2,-4);
\draw[thick] (3,-2) to (3,-4);
\draw[thick] (4,-2) to (6,-4);
\draw[thick] (5,-2) to (7,-4);
\end{tikzpicture} 
        \end{equation*}

        that is, we are in case (ii) of Proposition \ref{prop-interaction-z-BBU}.  Similarly, we have:

     \begin{eqnarray*}
        \mathfrak{z}_4\BBU(A)  &=&\BBU\left[\begin{matrix}
                1\\
                0
            \end{matrix}\right]\left(\BBU\left[\begin{matrix}
                2\\
                1
            \end{matrix}\right]\mathfrak{z}_3\right)\BBU\left[\begin{matrix}
                3\\
                2
            \end{matrix}\right]\BBU\left[\begin{matrix}
                6\\
                4
            \end{matrix}\right]\ref{prop-interaction-z-BBU}. ght]
         =\BBU\left[\begin{matrix}
                1\\
                0
                \end{matrix}\right]\mathfrak{z}_1\BBU\left[\begin{matrix}
                2&3&6\\
                1&2&4
            \end{matrix}\right]. \\
        \end{eqnarray*}
        again, this is case (ii) of Proposition  \ref{prop-interaction-z-BBU}.
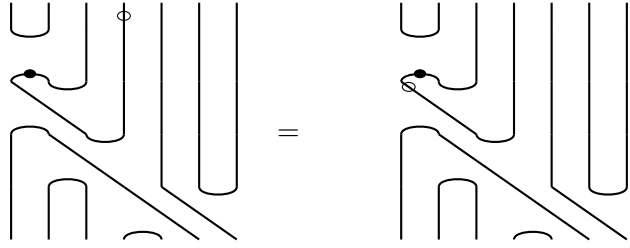
         \begin{equation*}
\begin{tikzpicture}[xscale=0.5,yscale=0.35]
\draw[thick] (1,4) to [out=270,in=270](2,4);
\draw[thick] (1,2) to [out=90,in=90](2,2);
\draw[fill] (1.5,2.3) circle [radius=0.15];
\draw[thick] (3,4) to (3,2);
\draw[thick] (4,4) to (4,2);
\draw[thick] (5,4) to (5,2);
\draw[thick] (6,4) to (6,2);
\draw[thick] (7,4) to (7,2);
\draw[] (4,4.5) circle [radius=0.17];
\draw[thick] (1,4) to (1,5);
\draw[thick] (2,4) to (2,5);
\draw[thick] (3,4) to (3,5);
\draw[thick] (4,4) to (4,5);
\draw[thick] (5,4) to (5,5);
\draw[thick] (6,4) to (6,5);
\draw[thick] (7,4) to (7,5);
\draw[thick] (2,2) to [out=270,in=270](3,2);
\draw[thick] (1,0) to [out=90,in=90](2,0);
\draw[thick] (1,2) to (3,0);
\draw[thick] (4,2) to (4,0);
\draw[thick] (5,2) to (5,0);
\draw[thick] (6,2) to (6,0);
\draw[thick] (7,2) to (7,0);
\draw[thick] (3,0) to [out=270,in=270](4,0);
\draw[thick] (2,-2) to [out=90,in=90](3,-2);
\draw[thick] (1,-2) to (1,0);
\draw[thick] (2,0) to (4,-2);
\draw[thick] (5,-2) to (5,0);
\draw[thick] (6,-2) to (6,0);
\draw[thick] (7,-2) to (7,0);

\draw[thick] (6,-2) to [out=270,in=270](7,-2);
\draw[thick] (4,-4) to [out=90,in=90](5,-4);
\draw[thick] (1,-2) to (1,-4);
\draw[thick] (2,-2) to (2,-4);
\draw[thick] (3,-2) to (3,-4);
\draw[thick] (4,-2) to (6,-4);
\draw[thick] (5,-2) to (7,-4);
\end{tikzpicture} \quad\begin{tikzpicture}[xscale=0.5,yscale=0.35]
\node[] at (-2,0) {$=$};
\draw[thick] (1,4) to [out=270,in=270](2,4);
\draw[thick] (1,2) to [out=90,in=90](2,2);
\draw[fill] (1.5,2.3) circle [radius=0.15];
\draw[thick] (3,4) to (3,2);
\draw[thick] (4,4) to (4,2);
\draw[thick] (5,4) to (5,2);
\draw[thick] (6,4) to (6,2);
\draw[thick] (7,4) to (7,2);
\draw[] (1.2,1.8) circle [radius=0.17];
\draw[thick] (1,4) to (1,5);
\draw[thick] (2,4) to (2,5);
\draw[thick] (3,4) to (3,5);
\draw[thick] (4,4) to (4,5);
\draw[thick] (5,4) to (5,5);
\draw[thick] (6,4) to (6,5);
\draw[thick] (7,4) to (7,5);
\draw[thick] (2,2) to [out=270,in=270](3,2);
\draw[thick] (1,0) to [out=90,in=90](2,0);
\draw[thick] (1,2) to (3,0);
\draw[thick] (4,2) to (4,0);
\draw[thick] (5,2) to (5,0);
\draw[thick] (6,2) to (6,0);
\draw[thick] (7,2) to (7,0);
\draw[thick] (3,0) to [out=270,in=270](4,0);
\draw[thick] (2,-2) to [out=90,in=90](3,-2);
\draw[thick] (1,-2) to (1,0);
\draw[thick] (2,0) to (4,-2);
\draw[thick] (5,-2) to (5,0);
\draw[thick] (6,-2) to (6,0);
\draw[thick] (7,-2) to (7,0);

\draw[thick] (6,-2) to [out=270,in=270](7,-2);
\draw[thick] (4,-4) to [out=90,in=90](5,-4);
\draw[thick] (1,-2) to (1,-4);
\draw[thick] (2,-2) to (2,-4);
\draw[thick] (3,-2) to (3,-4);
\draw[thick] (4,-2) to (6,-4);
\draw[thick] (5,-2) to (7,-4);
\end{tikzpicture} 
\end{equation*}
In a similar way, we obtain:  $$\mathfrak{z}_1\BBU(A)=\mathfrak{z}_2\BBU(A), \quad          \mathfrak{z}_5\BBU(A)=\BBU(A)\mathfrak{z}_7, \quad \text{and}\quad         \mathfrak{z}_6\BBU(A)=\mathfrak{z}_7\BBU(A).$$

 On the other hand we have:
    $$\BBU(A)\mathfrak{z}_1=\BBU(A)\mathfrak{z}_6, \quad
            \BBU(A)\mathfrak{z}_2=\BBU(A)\mathfrak{z}_3, \quad
            \BBU(A)\mathfrak{z}_4=\BBU(A)\mathfrak{z}_5 \quad \text{and}\quad
            \BBU(A)\mathfrak{z}_7=\mathfrak{z}_5\BBU(A).$$


    Thus we obtain:
    \begin{equation*}
        L(A)=\left[\begin{matrix}
            2 & 1 & 3 & 4 & 5 & 7 & 6 \\
            1 & 2 & 3 & 4 & 5 & 6 & 7 \\
            0 & 0 & 0 & 0 & 7 & 0 & 0
        \end{matrix}\right]\quad\text{and}\quad  R(A)=\left[\begin{matrix}
            0 & 0 & 0 & 0 & 0 & 0 & 5 \\
            1 & 2 & 3 & 4 & 5 & 6 & 7 \\
            6 & 3 & 2 & 5 & 4 & 1 & 0
        \end{matrix}\right].
    \end{equation*}
\end{example}

\begin{definition}\label{def-langle-ranlge-notation}
  Let $A=[B|C]$ be a non degenerate indexing matrix $A=[B|C]$. Put  $L(A)=[l_{ij}]$ and $R(A)=[r_{ij}]$, we  set:
$$
L_0(A):=\{\min\{l_{1j},l_{2j}\}\,;\, j\in[1,n]\}, \quad R_0(A):=\{
\min\{r_{2j},r_{3j}\}\,;\, j\in[1,n]\}\setminus\{0\},
$$
and we define the subsets $\mathrm{Exp}_L(A)$ and $\mathrm{Exp}_R(A)$ of $(\mathbb{Z}/d\mathbb{Z})^n$ by:
\begin{eqnarray*}
\mathrm{Exp}_L(A) &=&\{(\alpha_1, \ldots, \alpha_n)\,;\,\alpha_i =0\quad  \text{if}\quad i\not\in L_0(A) \},\\
\mathrm{Exp}_R(A) &=&\{(\alpha_1, \ldots, \alpha_n)\,;\,\alpha_i =0\quad  \text{if}\quad i\not\in R_0(A) \}.
\end{eqnarray*}
Finally, for $\alphan :=(\alpha_1,\ldots, \alpha_n)\in \mathrm{Exp}_L(A)$ (resp. $\mathrm{Exp}_R(A)$) we set
$$
\langle \alphan ,A \rangle =\prod_{j\in L_0(A)}\mathfrak{z}_j^{\alpha_j} \quad \text{(resp. $\langle A, \alphan \rangle =\prod_{j\in R_0(A)}\mathfrak{z}_j^{\alpha_j}$).}
$$
\end{definition}

\begin{notation}
    In  the case $A$ is the degenerate matrix, we set $L_0(A)= [1,n]$, $\mathrm{Exp}_L(A) = (\mathbb{Z}/d\mathbb{Z})^n$ and $R_0(A)=\emptyset$, $\mathrm{Exp}_R(A) =\{(0,\ldots , 0)\}.$
\end{notation}
\rm 

\begin{example}\label{ex-left-right-Z-matrices-2}
Set $d=3$ and keep the notation from Example \ref{ex-left-right-Z-matrices-V0}. Then we have

    \begin{equation*}
        L_0(A)=\{1,3,4,5,6\},\quad R_0(A)=\{1,2,4\}.
    \end{equation*}
For $\alphan=(2,0,1,1,2,1,0)\in\mathrm{Exp}_L(A) $ and $\bbeta=(1,2,0,1,0,0,0)\in\mathrm{Exp}_R(A)$, we have:
    \begin{equation*}
        \langle \alphan, A \rangle =\mathfrak{z}_1^{2}\mathfrak{z}_3\mathfrak{z}_4\mathfrak{z}_5^{2}\mathfrak{z}_6,\quad \langle  A,\bbeta  \rangle=\mathfrak{z}_1\mathfrak{z}_2^{2}\mathfrak{z}_4.
    \end{equation*}
    Note that
    \begin{equation*}
        |\mathrm{Exp}_L(A)|=3^5,\quad\text{and}\quad |\mathrm{Exp}_R(A)|=3^3.
    \end{equation*}
\end{example}

 \begin{corollary}\label{coro-normal-form-framed-blob-mon}
 Every element of $ \mathfrak{Bl}_{d,n} $ can be written in the  normal form
    $\langle \alphan,A \rangle\BBU(A)\langle A, \bbeta \rangle$,
    where $A$ is an indexing matrix,  $\alphan\in\mathrm{Exp}_L(A), \bbeta\in \mathrm{Exp}_R(A)$.
\end{corollary}
\begin{proof}
From Lemmas \ref{lemma-diag-normal-form-FrBlob-1}, \ref{lemma-normal-form-FrBlob-1} and Proposition ,\ref{prop-interaction-z-BBU}, it follows that every element can be written in the form $\langle \alphan,A \rangle\BBU(A)\langle A, \bbeta \rangle$.
The uniqueness of the writing is a consequence of Lemma \ref{lemma-normal-form-FrBlob-1}, Proposition \ref{prop-interaction-z-BBU} and Definition \ref{def-langle-ranlge-notation} (Note that implicitly we are using the fact that two $d$-abacus blob diagrams are equal, if and only if their corresponding basis diagrams are equal and the number of beads modulo $d$ is the same on each component).
\end{proof}

\begin{example}\label{ex-left-right-Z-matrices-3}
Keeping the   notations from Example \ref{ex-left-right-Z-matrices-2}, recalling that
 that:
 \begin{equation*}
        \mathfrak{z}_2\BBU(A)=\mathfrak{z}_1\BBU(A),\quad \BBU(A)\mathfrak{z}_5=\BBU(A)\mathfrak{z}_4,\quad \BBU(A)z_7=z_5\BBU(A).
\end{equation*}
 Thus, for  $X=\mathfrak{z}_1\mathfrak{z}_2\mathfrak{z}_3\BBU(A)\mathfrak{z}_5\mathfrak{z}_7,$ we have:
    \begin{equation*}
X=\mathfrak{z}_3(((z_1z_2\BBU(A))\mathfrak{z}_7)\mathfrak{z}_5)=\mathfrak{z}_3\mathfrak{z}_1^2((\BBU(A)\mathfrak{z}_7)\mathfrak{z}_5)=\mathfrak{z}_3\mathfrak{z}_1^2\mathfrak{z}_5(\BBU(A)\mathfrak{z}_5)=\mathfrak{z}_1^2\mathfrak{z}_3\mathfrak{z}_5\BBU(A)\mathfrak{z}_4.
    \end{equation*}
 This last expression is the normal form of $X$. In fact we have:
    \begin{equation*}
        X=\langle  \alphan,A \rangle\BBU(A)\langle A,\bbeta\rangle,
    \end{equation*}
    where $\alphan=(2,0,1,0,1,0,0)\in \mathrm{Exp}_{L}(A)$ and $\bbeta=(0,0,0,1,0,0,0)\in \mathrm{Exp}_{R}(A)$.
\end{example}

\begin{example}
    Let $D^{\circ}$ the diagram of Fig.\ref{ex-dframed-blob-diagram}. The rightmost diagram in Fig.\ref{eq-tikz-quasi-normalized-pearls+blob-diag} corresponds to the normal form of $D^{\circ}$. From  Example \ref{ex-matrix-form-tikz-blobbed-diagram} follows that the normal form of  $D^{\circ}$ is
    \begin{equation*}
\mathfrak{z}_3^2\mathfrak{z}_5\mathfrak{z}_8\mathfrak{z}_9^2\left(\BBU\left[\begin{matrix}
            1&2&5&6\\
            0&0&2&5
        \end{matrix}\right]\right)\mathfrak{z}_1\mathfrak{z}_5\mathfrak{z}_7^2.
    \end{equation*}
\end{example}

Let $A=[B|C]$ be  a non degenerated indexing matrix with blob-rank with $r$. Define
\begin{equation*}
\calL(A):=\{(k,1):k\in L_0(A)\},\quad \calR(A):=\{(k,0):k\in R_0(A)\}.
\end{equation*}

Note that  $\calL(A) \cap {\calR(A)}=\emptyset$ and $|\calL (A)\cup\calR(A)|=|L_0 (A)|+|R_0(A)|.$
 Note also that we can write:
\begin{equation}\label{DecomA}
\BBU(A)=\BBU\left[\begin{matrix}
                b_1\\
                0
            \end{matrix}\right]\BBU(A'),\quad \text{where}
            \quad
                A'=\left[\begin{matrix}
                    a_{12}&\cdots& a_{1l}\\
                    a_{22}&\cdots &a_{2l}
                \end{matrix}\right]=\left[\begin{matrix}
                    b_2&\cdots &b_r&c_{11}&\cdots &c_{1m}\\
                    0&\cdots &0&c_{21}&\cdots &c_{2m}
                \end{matrix}\right].
            \end{equation}

\begin{lemma}\label{lemma-injection-psi}
There is an injective function $\psi:\calL(A')\cup \calR(A')\rightarrow \calL(A)\cup \calR(A)$.
\end{lemma}
\begin{proof}
It is enough check the case $b_1>0$, since the case $b_1=0$ follows analogously.
Because  $R_0(A')\subset R_0(A)$, we can define $\psi(k,0)=(k,0)$. We will  define $\psi(k,1)$ according to the following four cases:  (i) $2<k\leq b_1+1$, (ii) $k>b_1+1$, (iii) $k=1$ or (iv) $k=2$, which are the exclusive cases of Lemma \ref{zA0U}. Thus, in virtue of this lemma, we can define $\psi(k,1)=(k-2,1)$ and $\psi(k,1)=(k,1)$ in the cases (i) and (ii), respectively. In the case (iii) we have two exclusive possibilities: The first possibility is that  there is $p$ such that $\mathfrak{z}_1\BBU(A')=\BBU(A')\mathfrak{z}_p.$
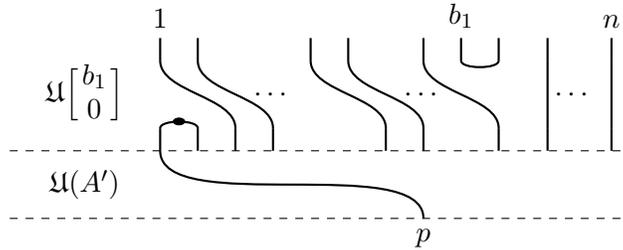
\begin{equation*}
    \begin{tikzpicture}[xscale=0.5,yscale=0.3]
\node[] at (-1,1.5) {$\BBU\Big[\!\!\begin{array}{c}b_1 \\0\end{array}\!\!\Big]$};
\draw[thick] (1,0) to [out=90,in=90](2,0);
\draw[fill] (1.5,0.3) circle [radius=0.15];
\draw[thick] (3,0) to [out=90,in=270](1,3);
\draw[thick] (4,0) to [out=90,in=270](2,3);
\node[] at (4,1.5) {$\cdots$};
\draw[thick] (5,3) to [out=270,in=90](7,0);
\draw[thick] (6,3) to [out=270,in=90](8,0);
\draw[thick] (8,3) to [out=270,in=90](10,0);
\draw[thick] (11.3,3) to [out=270,in=90](11.3,0);
\node[] at (8,1.5) {$\cdots$};
\draw[thick] (10,3) to [out=270,in=270](9,3);
\node[] at (12,1.5) {$\cdots$};
\draw[thick] (13,3) to [out=270,in=90](13,0);
\draw[dashed] (-3,-1) to (13.5,-1);
\node[above] at (1,4) {$1$};
\node[above] at (9,4) {$b_1$};
\node[above] at (13,4) {$n$};


\draw[thick](1,3) to (1,4);
\draw[thick](2,3) to (2,4);
\draw[thick](5,3) to (5,4);
\draw[thick](6,3) to (6,4);
\draw[thick](8,3) to (8,4);
\draw[thick](9,3) to (9,4);
\draw[thick](10,3) to (10,4);
\draw[thick](11.3,3) to (11.3,4);
\draw[thick](13,3) to (13,4);
\draw[thick](1,0) to (1,-1);
\draw[thick](2,0) to (2,-1);
\draw[thick](4,0) to (4,-1);
\draw[thick](3,0) to (3,-1);
\draw[thick](7,0) to (7,-1);
\draw[thick](8,0) to (8,-1);
\draw[thick](10,0) to (10,-1);
\draw[thick](11.3,0) to (11.3,-1);
\draw[thick](13,0) to (13,-1);
\node[] at (-1,-2.5) {$\BBU(A')$};
\draw[thick] (1,-1) to [out=270,in=90](8,-4);
\draw[dashed] (-3,-4) to (13.5,-4);
\node[below] at (8,-4) {$p$};
\end{tikzpicture} 
\end{equation*}
So we define $\psi(1,1)=(p,0)$; note that in this situation we have $p\notin R_0(A')$ but  $p\in R_0(A),$ so $\psi(1,1)=(p,0)$ is well defined. The other possibility is that
there exist $i> 1,$ such that $\mathfrak{z}_1\BBU(A')=\mathfrak{z}_i\BBU(A'),$ hence $i\notin L_0(A')$. By definition  $a_{12}>a_{11}=b_1,$ which implies that $i>3.$ Therefore by Lemma \ref{zA0U}, $\BBU\left[\begin{matrix}
                b_1\\
                0
            \end{matrix}\right]\mathfrak{z}_i=\mathfrak{z}_p\BBU\left[\begin{matrix}
                b_1\\
                0
            \end{matrix}\right]$ for a unique  $p\in L_0\left[\begin{matrix}
                b_1\\
                0
            \end{matrix}\right]$ such that $p\neq b_1,b_1+1.$
        
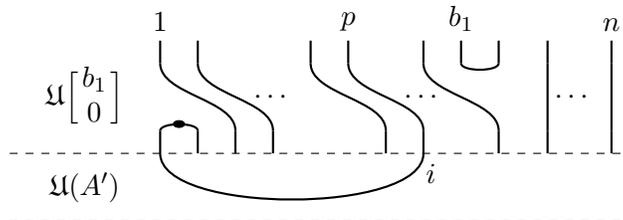
\begin{equation*}
\begin{tikzpicture}[xscale=0.5,yscale=0.3]
\node[] at (-1,1.5) {$\BBU\Big[\!\!\begin{array}{c}b_1 \\0\end{array}\!\!\Big]$};
\draw[thick] (1,0) to [out=90,in=90](2,0);
\draw[fill] (1.5,0.3) circle [radius=0.15];
\draw[thick] (3,0) to [out=90,in=270](1,3);
\draw[thick] (4,0) to [out=90,in=270](2,3);
\node[] at (4,1.5) {$\cdots$};
\draw[thick] (5,3) to [out=270,in=90](7,0);
\draw[thick] (6,3) to [out=270,in=90](8,0);
\draw[thick] (8,3) to [out=270,in=90](10,0);
\draw[thick] (11.3,3) to [out=270,in=90](11.3,0);
\node[] at (8,1.5) {$\cdots$};
\draw[thick] (10,3) to [out=270,in=270](9,3);
\node[] at (12,1.5) {$\cdots$};
\draw[thick] (13,3) to [out=270,in=90](13,0);
\draw[dashed] (-3,-1) to (13.5,-1);
\node[above] at (1,4) {$1$};
\node[above] at (9,4) {$b_1$};
\node[above] at (13,4) {$n$};


\draw[thick](1,3) to (1,4);
\draw[thick](2,3) to (2,4);
\draw[thick](5,3) to (5,4);
\draw[thick](6,3) to (6,4);
\draw[thick](8,3) to (8,4);
\draw[thick](9,3) to (9,4);
\draw[thick](10,3) to (10,4);
\draw[thick](11.3,3) to (11.3,4);
\draw[thick](13,3) to (13,4);
\draw[thick](1,0) to (1,-1);
\draw[thick](2,0) to (2,-1);
\draw[thick](4,0) to (4,-1);
\draw[thick](3,0) to (3,-1);
\draw[thick](7,0) to (7,-1);
\draw[thick](8,0) to (8,-1);
\draw[thick](10,0) to (10,-1);
\draw[thick](11.3,0) to (11.3,-1);
\draw[thick](13,0) to (13,-1);
\node[] at (-1,-2.5) {$\BBU(A')$};
\draw[thick] (1,-1) to [out=270,in=270](8,-1);
\draw[dashed] (-3,-4) to (13.5,-4);
\node[below] at (8.2,-1) {$i$};
\node[above] at (6,4) {$p$};
\end{tikzpicture} 
\end{equation*}

            Thus,  we define $\psi(1,1)=(p,1).$ The defintion of $\psi(2,1)$ follows analogously. Finally,  by the exclusiveness of the cases described in Proposition \ref{prop-interaction-z-BBU} it follows that $\psi$ is injective.
            \end{proof}

\begin{lemma}\label{lemma-almost-surj-psi}
 The following equation holds:
 \begin{equation*}
|\calL(A)\cup\calR(A)|-|\calL(A')\cup\calR(A')|=1.
 \end{equation*}
\end{lemma}

\begin{proof}
It is enough to prove that
\begin{equation*}
|\calL(A)\cup\calR(A)|-|\psi\left(\calL(A')\cup\calR(A')\right)|=1,
 \end{equation*}
 where $\psi$  is the (injective) function of the proof of Lemma \ref{lemma-injection-psi}.
Let  $(p,x)\in\calL(A)\cup \calR(A)$ such that $(p,x)\neq (b_1,1).$ Then by the construction of $\psi$, we have the following situations:
\begin{itemize}
\item If $(p,x)=(p,0)\in \calR(A')\cap \calR(A),$ then $(p,0)\in \psi(\calL(A')\cup \calR(A'))$ since in that case $\psi(p,0)=(p,0).$
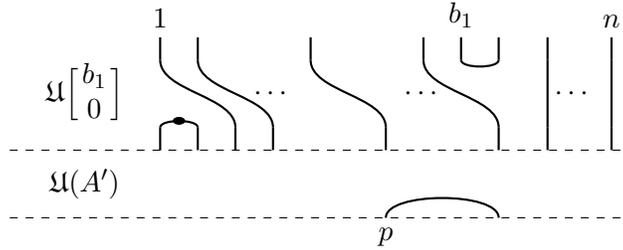
\begin{equation*}
    \begin{tikzpicture}[xscale=0.5,yscale=0.3]
\node[] at (-1,1.5) {$\BBU\Big[\!\!\begin{array}{c}b_1 \\0\end{array}\!\!\Big]$};
\draw[thick] (1,0) to [out=90,in=90](2,0);
\draw[fill] (1.5,0.3) circle [radius=0.15];
\draw[thick] (3,0) to [out=90,in=270](1,3);
\draw[thick] (4,0) to [out=90,in=270](2,3);
\node[] at (4,1.5) {$\cdots$};
\draw[thick] (5,3) to [out=270,in=90](7,0);
\draw[thick] (8,3) to [out=270,in=90](10,0);
\draw[thick] (11.3,3) to [out=270,in=90](11.3,0);
\node[] at (8,1.5) {$\cdots$};
\draw[thick] (10,3) to [out=270,in=270](9,3);
\node[] at (12,1.5) {$\cdots$};
\draw[thick] (13,3) to [out=270,in=90](13,0);
\draw[dashed] (-3,-1) to (13.5,-1);
\node[above] at (1,4) {$1$};
\node[above] at (9,4) {$b_1$};
\node[above] at (13,4) {$n$};


\draw[thick](1,3) to (1,4);
\draw[thick](2,3) to (2,4);
\draw[thick](5,3) to (5,4);
\draw[thick](8,3) to (8,4);
\draw[thick](9,3) to (9,4);
\draw[thick](10,3) to (10,4);
\draw[thick](11.3,3) to (11.3,4);
\draw[thick](13,3) to (13,4);
\draw[thick](1,0) to (1,-1);
\draw[thick](2,0) to (2,-1);
\draw[thick](4,0) to (4,-1);
\draw[thick](3,0) to (3,-1);
\draw[thick](7,0) to (7,-1);
\draw[thick](10,0) to (10,-1);
\draw[thick](11.3,0) to (11.3,-1);
\draw[thick](13,0) to (13,-1);
\node[] at (-1,-2.5) {$\BBU(A')$};
\draw[thick] (7,-4) to [out=90,in=90](10,-4);
\draw[dashed] (-3,-4) to (13.5,-4);
\node[below] at (7,-4) {$p$};
\end{tikzpicture} 
\end{equation*}
\item If $(p,x)=(p,0)\in \calR(A)$ but $(p,0)\notin\calR(A'),$ then, $\BBU(A')\mathfrak{z}_{p}=\mathfrak{z}_j\BBU(A')$ for $j\in\{1,2\}$. Note that in particular $(j,1)\in\calL(A')$.
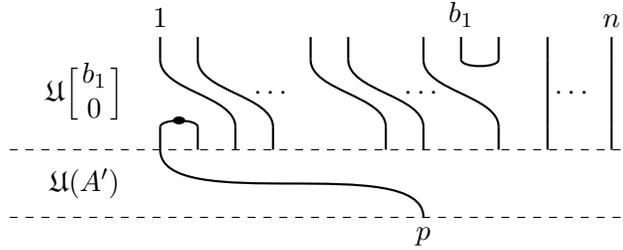
\begin{equation*}
    \begin{tikzpicture}[xscale=0.5,yscale=0.3]
\node[] at (-1,1.5) {$\BBU\Big[\!\!\begin{array}{c}b_1 \\0\end{array}\!\!\Big]$};
\draw[thick] (1,0) to [out=90,in=90](2,0);
\draw[fill] (1.5,0.3) circle [radius=0.15];
\draw[thick] (3,0) to [out=90,in=270](1,3);
\draw[thick] (4,0) to [out=90,in=270](2,3);
\node[] at (4,1.5) {$\cdots$};
\draw[thick] (5,3) to [out=270,in=90](7,0);
\draw[thick] (6,3) to [out=270,in=90](8,0);
\draw[thick] (8,3) to [out=270,in=90](10,0);
\draw[thick] (11.3,3) to [out=270,in=90](11.3,0);
\node[] at (8,1.5) {$\cdots$};
\draw[thick] (10,3) to [out=270,in=270](9,3);
\node[] at (12,1.5) {$\cdots$};
\draw[thick] (13,3) to [out=270,in=90](13,0);
\draw[dashed] (-3,-1) to (13.5,-1);
\node[above] at (1,4) {$1$};
\node[above] at (9,4) {$b_1$};
\node[above] at (13,4) {$n$};


\draw[thick](1,3) to (1,4);
\draw[thick](2,3) to (2,4);
\draw[thick](5,3) to (5,4);
\draw[thick](6,3) to (6,4);
\draw[thick](8,3) to (8,4);
\draw[thick](9,3) to (9,4);
\draw[thick](10,3) to (10,4);
\draw[thick](11.3,3) to (11.3,4);
\draw[thick](13,3) to (13,4);
\draw[thick](1,0) to (1,-1);
\draw[thick](2,0) to (2,-1);
\draw[thick](4,0) to (4,-1);
\draw[thick](3,0) to (3,-1);
\draw[thick](7,0) to (7,-1);
\draw[thick](8,0) to (8,-1);
\draw[thick](10,0) to (10,-1);
\draw[thick](11.3,0) to (11.3,-1);
\draw[thick](13,0) to (13,-1);
\node[] at (-1,-2.5) {$\BBU(A')$};
\draw[thick] (1,-1) to [out=270,in=90](8,-4);
\draw[dashed] (-3,-4) to (13.5,-4);
\node[below] at (8,-4) {$p$};
\end{tikzpicture} 
\end{equation*}
In that case we have $\psi(j,1)=(p,0),$ and then $(p,0)\in \psi(\calL(A')\cup \calR(A')).$
\item If $(p,x)=(p,1)\in \calL(A)$ with $1\leq p<b_1$ or $p>b_1+1$ then there is a $2<i\leq n$ such that $\mathfrak{z}_p\BBU\left[\begin{matrix}
b_1\\
0
\end{matrix}\right]=\BBU\left[\begin{matrix}
b_1\\
0
\end{matrix}\right]\mathfrak{z}_i,$ and we have two exclusive cases:
\begin{itemize}
\item If $(i,1)\in\calL(A'),$ then $\psi(i,1)=(p,1),$ so
$(p,1)\in \psi(\calL(A')\cup \calR(A')).$
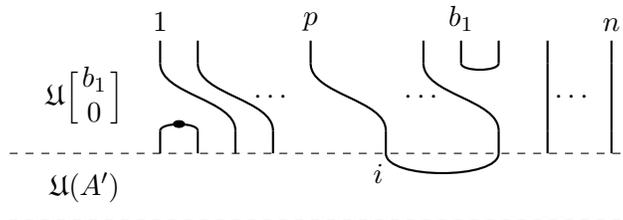
\begin{equation*}
    \begin{tikzpicture}[xscale=0.5,yscale=0.3]
\node[] at (-1,1.5) {$\BBU\Big[\!\!\begin{array}{c}b_1 \\0\end{array}\!\!\Big]$};
\draw[thick] (1,0) to [out=90,in=90](2,0);
\draw[fill] (1.5,0.3) circle [radius=0.15];
\draw[thick] (3,0) to [out=90,in=270](1,3);
\draw[thick] (4,0) to [out=90,in=270](2,3);
\node[] at (4,1.5) {$\cdots$};
\draw[thick] (5,3) to [out=270,in=90](7,0);
\draw[thick] (8,3) to [out=270,in=90](10,0);
\draw[thick] (11.3,3) to [out=270,in=90](11.3,0);
\node[] at (8,1.5) {$\cdots$};
\draw[thick] (10,3) to [out=270,in=270](9,3);
\node[] at (12,1.5) {$\cdots$};
\draw[thick] (13,3) to [out=270,in=90](13,0);
\draw[dashed] (-3,-1) to (13.5,-1);
\node[above] at (1,4) {$1$};
\node[above] at (9,4) {$b_1$};
\node[above] at (13,4) {$n$};


\draw[thick](1,3) to (1,4);
\draw[thick](2,3) to (2,4);
\draw[thick](5,3) to (5,4);
\draw[thick](8,3) to (8,4);
\draw[thick](9,3) to (9,4);
\draw[thick](10,3) to (10,4);
\draw[thick](11.3,3) to (11.3,4);
\draw[thick](13,3) to (13,4);
\draw[thick](1,0) to (1,-1);
\draw[thick](2,0) to (2,-1);
\draw[thick](4,0) to (4,-1);
\draw[thick](3,0) to (3,-1);
\draw[thick](7,0) to (7,-1);
\draw[thick](10,0) to (10,-1);
\draw[thick](11.3,0) to (11.3,-1);
\draw[thick](13,0) to (13,-1);
\node[] at (-1,-2.5) {$\BBU(A')$};
\draw[thick] (7,-1) to [out=270,in=270](10,-1);
\draw[dashed] (-3,-4) to (13.5,-4);
\node[below] at (6.8,-1) {$i$};
\node[above] at (5,4) {$p$};
\end{tikzpicture} 
\end{equation*}
\item If $(i,1)\notin \calL(A'),$ then  $\mathfrak{z}_i\BBU(A')=\mathfrak{z}_j\BBU(A')$ for $j\in\{1,2\},$ then $(j,1)\in \calL(A')$ and by definition $\psi(j,1)=(p,1),$ so
$(p,1)\in \psi(\calL(A')\cup \calR(A')).$
For example:
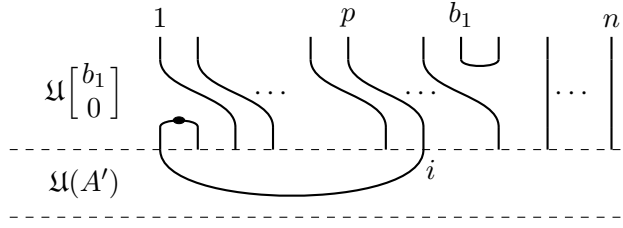
\begin{equation*}
   \begin{tikzpicture}[xscale=0.5,yscale=0.3]
\node[] at (-1,1.5) {$\BBU\Big[\!\!\begin{array}{c}b_1 \\0\end{array}\!\!\Big]$};
\draw[thick] (1,0) to [out=90,in=90](2,0);
\draw[fill] (1.5,0.3) circle [radius=0.15];
\draw[thick] (3,0) to [out=90,in=270](1,3);
\draw[thick] (4,0) to [out=90,in=270](2,3);
\node[] at (4,1.5) {$\cdots$};
\draw[thick] (5,3) to [out=270,in=90](7,0);
\draw[thick] (6,3) to [out=270,in=90](8,0);
\draw[thick] (8,3) to [out=270,in=90](10,0);
\draw[thick] (11.3,3) to [out=270,in=90](11.3,0);
\node[] at (8,1.5) {$\cdots$};
\draw[thick] (10,3) to [out=270,in=270](9,3);
\node[] at (12,1.5) {$\cdots$};
\draw[thick] (13,3) to [out=270,in=90](13,0);
\draw[dashed] (-3,-1) to (13.5,-1);
\node[above] at (1,4) {$1$};
\node[above] at (9,4) {$b_1$};
\node[above] at (13,4) {$n$};


\draw[thick](1,3) to (1,4);
\draw[thick](2,3) to (2,4);
\draw[thick](5,3) to (5,4);
\draw[thick](6,3) to (6,4);
\draw[thick](8,3) to (8,4);
\draw[thick](9,3) to (9,4);
\draw[thick](10,3) to (10,4);
\draw[thick](11.3,3) to (11.3,4);
\draw[thick](13,3) to (13,4);
\draw[thick](1,0) to (1,-1);
\draw[thick](2,0) to (2,-1);
\draw[thick](4,0) to (4,-1);
\draw[thick](3,0) to (3,-1);
\draw[thick](7,0) to (7,-1);
\draw[thick](8,0) to (8,-1);
\draw[thick](10,0) to (10,-1);
\draw[thick](11.3,0) to (11.3,-1);
\draw[thick](13,0) to (13,-1);
\node[] at (-1,-2.5) {$\BBU(A')$};
\draw[thick] (1,-1) to [out=270,in=270](8,-1);
\draw[dashed] (-3,-4) to (13.5,-4);
\node[below] at (8.2,-1) {$i$};
\node[above] at (6,4) {$p$};
\end{tikzpicture} 
\end{equation*}
\end{itemize}
\end{itemize}
With the above analysis we conclude that
\begin{equation*}
\psi(\calL(A')\cup \calR(A'))=  \left(\calL(A)\cup \calR(A)\right)-\{(b_1,1)\}.
\end{equation*}
Therefore,
$
|\calL(A)\cup\calR(A)|-|\psi\left(\calL(A')\cup\calR(A')\right)|=1
,$
as desired.
\end{proof}

\begin{proposition}\label{lemma-equation-n+r}
If $A=[B|C]$ is a non degenerated indexing matrix with blob-rank $r\geq 0,$ then
    \begin{equation*}
        |L_0(A)|+|R_0(A)|= n+r\quad \text{and}\quad
        |\mathrm{Exp}_L(A)
        |\cdot|\mathrm{Exp}_R(A)|= d^{n+r}.
    \end{equation*}
\end{proposition}
\begin{proof}
Note that the second  equality follows from the first and the facts that $|\textrm{Exp}_L(A)|=d^{|L_0(A)|}$ and $|\textrm{Exp}_R(A)|=d^{|R_0(A)|}.$ The first  equality will be proven by induction on the rank $r$ of $A$.
If $r=0$, the  Proposition \ref{prop-interaction-z-BBU} says that we have the following two exclusive cases for $\mathfrak{z}_j\BBU(A)$, with $j\in[1,n]$:
\begin{itemize}
\item There is $k\in[1,n]$ such that $\mathfrak{z}_j\BBU(A)=\BBU(A)\mathfrak{z}_k.$ In this case $j\in L_0(A)$ and $k\notin R_0(A).$

\item There is $i\neq j$ such that $\mathfrak{z}_j\BBU(A)=\mathfrak{z}_i\BBU(A).$ In this case we have $|\{i,j\}\cap L_0(A)|=1.$
\end{itemize}
Analogously, for each $k\in[1,n]$ we have the following two exclusive possibilities for $\BBU(A)\mathfrak{z}_k$:
 \begin{itemize}
\item There is  $j$ such that $\BBU(A)\mathfrak{z}_k$ such that $\BBU(A)\mathfrak{z}_k=\mathfrak{z}_j\BBU(A).$ In this case $j\in L_0(A)$ and $k\notin R_0(A).$
\item There is $i\neq k$ such that $\BBU(A)\mathfrak{z}_k=\BBU(A)\mathfrak{z}_i.$ In this case we have $|\{i,k\}\cap R_0(A)|=1.$
\end{itemize}
  Thus we deduce that  $|L_0(A)|+|R_0(A)|=n$, that is, the claim is true for $r=0$.

  Suppose now that the claim  is true for $r-1,$ with $r>0$.  We can write: $\BBU(A)=\BBU\left[\begin{matrix}
                b_1\\
                0
            \end{matrix}\right]\BBU(A')$, where $A'$ is as in (\ref{DecomA}).
 By inductive hypothesis we have $|L_0(A')|+|R_0(A')|=n+r-1,$ since
 the blob-rank of $A'$ is $r-1.$ Now,  by Lemma \ref{lemma-almost-surj-psi}, and the fact that $|L_0(A')|+|R_0(A')|=|\calL(A')\cup\calR(A')|,$  we have:
 \begin{equation*}
     |L_0(A)|+|R_0(A)|=|\calL(A)\cup\calR(A)|=|\calL(A')\cup\calR(A')|+1=n+r.
 \end{equation*} As desired.
\end{proof}

\begin{remark}\rm
Note that Proposition \ref{lemma-equation-n+r} is also holds for the degenerate case. In fact, if $A$ is a degenerate index matrix, then its blob-rank is $0$, so
$|\calL(A)|=n, |\calR(A)|=0,  |\textrm{Exp}_L(A)|=d^n,\quad |\textrm{Exp}_R(A)|=1.$ Thus the proposition remains valid for the degenerate case as well.
\end{remark}

Next, we will count the number of all indexing arrays with a given blob-rank.

\begin{definition}\label{def-Omega-numbers}
    Given $r\in [0,n]$, we define the number $\Omega_{r}^{(n)}$ as the number of all indexing matrices with blob-rank equal to $r.$ We set $\Omega_{0}^{(0)}=1.$
\end{definition}

\begin{theorem}\label{theo-card-Fblob-mon-first}
    The cardinality of the abacus blob monoid  is given by:
    \begin{equation}\label{eq-dim-Frblob-mon-first}
    |\mathfrak{Bl}_{d,n} |= d^{n}\sum_{k=0}^{n}\Omega_k^{(n)}d^k.
    \end{equation}

\end{theorem}
\begin{proof}
 Denote $O_k$ the set of all the indexing matrices with blob-rank equal to $k,$ then Corollary \ref{coro-normal-form-framed-blob-mon} implies that $\mathfrak{Bl}_{d,n}$ can be decomposed, as set,  in the following disjoint union:
    \begin{equation*}
        \mathfrak{Bl}_{d,n}=\bigcup_{k=0}^{n}\left\{\langle \zetan, A\rangle^{\alphan}\BBU(A)\langle A,\zetan\rangle^{\bbeta}: A\in O_k, \alphan\in \textrm{Exp}_L(A),\bbeta\in\textrm{Exp}_R(A)
        \right\}.
    \end{equation*} From Proposition \ref{lemma-equation-n+r} we obtain
    \begin{equation*}
|\mathfrak{Bl}_{d,n}|=\sum_{k=0}^{n}\Omega_{k}^{(n)}|\textrm{Exp}_L(A)||\textrm{Exp}_R(A)|=\sum_{k=0}^{n}\Omega_{k}^{(n)}d^{n+k}.
    \end{equation*}
So the proof is finished.
\end{proof}

\section{Framed blob monoid}
The goal of this section is to show a presentation of the  abacus blob monoid. To do so, we introduce the monoid {\it framed blob monoid} $Bl_{d,n}$, which is defined by generators and relations, and prove that $Bl_{d,n}$ and $\mathfrak{Bl}_{d,n}$ are isomorphic, see Theorem \ref{theo-iso-AlgFrblob-DiagFrblob}.

\begin{definition}\label{def-framed-blob-mon}
 The $d$-framization (or simply framed blob monoid)
 $Bl_{d,n}$ of  $Bl_n$, is the monoid generated by $u_0,u_1, \ldots,  u_{n-1}$, $z_1, \ldots , z_{n}$ satisfying the defining relations of $Bl_n$ together with the following relations:
\begin{align}
  z_i^d =1, &\qquad  z_iz_j =z_jz_i,\label{Bld1}\\
  z_iu_i = z_{i+1}u_i, &\qquad u_i z_i = u_iz_{i+1},\label{Bld2}\\
u_0 z_i =z_i u_0 \quad \text{if $i\not=1$,}&\qquad u_i z_j =z_j u_i \quad \text{if $j\not=i,i+1$,}\label{Bld3}\\
u_i z_i^k u_i=u_i, & \qquad u_1 z_1^k u_0 u_1=u_1, \qquad u_1 u_0 z_1^k u_1=u_1.\label{Bld4}
\end{align}
\end{definition}
\begin{remark}\rm Note that the abacus monoid $C_{d}^n$ and the blob monoid $Bl_{n}$ are submonoids of $Bl_{d,n}$.
\end{remark}

Using the defining relations of $Bl_{d,n}$ it follows that any element of $Bl_{d,n}$ can be decomposed in the form $ztz'$, where $z,z'\in C_d^n$ and $t\in Bl_{n}$. To improve this decomposition we need to introduce the following notation.
\begin{notation}(Cf. Notation \ref{Not1})
 For the  exponent vector  $\alphan:=(\alpha_1,\dots,\alpha_n)\in (\ZZ/d\ZZ)^n$,  we set  $Z^{\alphan}:=z_1^{\alpha_1}\cdots z_{n}^{\alpha_n}$.
\end{notation}
\begin{lemma}\label{lemma-normal-form-FrBlob-2}
  Every  $X\in Bl_{d,n}$ can be written in the form
   $$
   X=Z^{\alphan}=Z^{\alphan}U\Big[\!\!\begin{array}{c}\infty\\\infty\end{array}\!\!\Big]
\quad \text{or}\quad  X=Z^{\alphan}U(A)Z^{\bbeta},$$ where $\alphan,\bbeta$ are exponent vectors and $A$ is an  indexing matrix.
\end{lemma}

\begin{proof}
   Let $X\in Bl_{d,n}$ and  assume that $X\neq Z^{\alphan}$. Then $$X=Z^{\alphan_0}u_{i_1}Z^{\alphan_1}u_{i_2}\cdots Z^{\alphan_{m-1}}u_{i_m}Z^{\alphan_{m}}.$$ for certain indices $i_j$ and vectors of exponents $\alphan_0,\dots ,\alphan_m.$ For every exponent vector $\alphan$ we have \begin{equation}\label{eq-drel-plus-framed-blob-mon-1} u_iZ^{\alphan}u_j=Z^{\bbeta}u_iu_jZ^{ \bgamma} \end{equation} since (\ref{Bld2}) and (\ref{Bld4}).This last equation and a recursive argument imply that
   \begin{equation}
        X=Z^{\alphan}u_{i_1}u_{i_2}\cdots u_{i_m}Z^{\bbeta}
    \end{equation}
    for appropriate vector of exponents $\alphan,\bbeta.$
    Now the product  $Y=u_{i_1}\cdots u_{i_m}\in Bl_n$ can be written in its normal, that is,  $Y= U(A)$ for an appropriate indexing matrix $A$.
    This concludes the proof.
\end{proof}
Our next goal is to give a upper bound  for the cardinality of $Bl_n$. For this we need several technical results and definitions. We  start with the following two lemmas which give the commutation relations between framing elements and  $U(A)\in Bl_n$.

\begin{lemma}\label{zUAbs} Set $A =\Big[\!\!\begin{array}{c}i\\j\end{array}\!\!\Big]$, where $i,j\in [1,n-1]$ with $j<i$.  The following relations hold in $Bl_{d,n}$.
\begin{enumerate}
\item[(i)]
\begin{align}\label{eq1-lemma-interaction-z-BBU-case1}
    z_kU(A)=\left\{\begin{array}{ll}
                                U(A)z_k & \textrm{if}\quad k<j\quad \textrm{or}\quad k>i+1 \\
                                 \quad & \quad \\
                                 U(A)z_{k+2}&\textrm{if}\quad j\leq k\leq i-1\\
                                 \quad & \quad \\
                                 z_{i+1}U(A)&\textrm{if}\quad k=i\\
                                 \quad & \quad \\
                                 z_{i}U(A)&\textrm{if}\quad k=i+1.
                               \end{array}\right.
\end{align}
\item[(ii)]
 \begin{align}\label{eq2-lemma-interaction-z-BBU-case1}
    U(A)z_k=\left\{\begin{array}{ll}
                                 z_kU(A) & \textrm{if}\quad k<j\quad \textrm{or}\quad k>i+1 \\
                                 \quad & \quad \\
                                 z_{k-2}U(A)&\textrm{if}\quad j+2\leq k\leq i+1\\
                                 \quad & \quad \\
                                U(A)z_{j+1}&\textrm{if}\quad k=j\\
                                 \quad & \quad \\
                                 U(A)z_{j}&\textrm{if}\quad k=j+1.
                               \end{array}\right.
  \end{align}
\end{enumerate}
\end{lemma}
\begin{proof}
All relations are a direct consequence of the defining relations of $Bl_{d,n}$. As an example we prove in details $ z_kU(A)= U(A)z_{k+2}$. Notice that for $k\in [j+1,i]$, we  have:
$z_ku_{k+1}u_{k}  \stackrel {(\ref{Bld3})}{=} u_{k+1} z_k u_{k} \stackrel {(\ref{Bld2})}{=} u_{k+1} z_{k+1} u_{k}\stackrel {(\ref{Bld2})}{=} u_{k+1} z_{k+2}u_k
\stackrel {(\ref{Bld3})}{=} u_{k+1} u_kz_{k+2}$, hence,
\begin{equation}\label{zkRight}
z_ku_{k+1}u_{k}=u_{k+1} u_kz_{k+2}.
\end{equation}
Thus, the proof is obtained by moving $z_k$ in $z_kU(A)$ to the right until it meets $u_{k+1}u_{k}$ and then applying (\ref{zkRight}) and (\ref{Bld3}).
\end{proof}

\begin{lemma}\label{zA0Abs}
Set $A =\Big[\!\!\begin{array}{c}i\\0\end{array}\!\!\Big]$, where $i\in [1,n-1]$.  The following relations hold in $Bl_{d,n}$.
\begin{enumerate}
\item[(i)]
 \begin{align}\label{eq3-lemma-interaction-z-BBU-case1}
    z_kU(A)=\left\{\begin{array}{ll}
                                 U(A)z_k & \textrm{if}\quad k>i+1 \\
                                 \quad & \quad \\
                                 U(A)z_{k+2}&\textrm{if}\quad 1\leq k\leq i-1\\
                                 \quad & \quad \\
                                 z_{i+1}U(A)&\textrm{if}\quad k=i\\
                                 \quad & \quad \\
                                 z_{i}U(A)&\textrm{if}\quad k=i+1.
                                 \end{array}\right.
  \end{align}
\item[(ii)]
\begin{align}\label{eq4-lemma-interaction-z-BBU-case1}
    U(A)z_k=\left\{\begin{array}{ll}
                                 z_kU(A) & \textrm{if}\quad k>i+1 \\
                                 \quad & \quad \\
                                 z_{k-2}U(A)&\textrm{if}\quad 3\leq k\leq i+1.
                               \end{array}\right.
  \end{align}
\end{enumerate}
\end{lemma}
\begin{proof}
Analogous to the proof of  Lemma \ref{zUAbs}.
\end{proof}

\begin{proposition}\label{prop-abstract-interaction-z-BBU}
Given a non degenerated indexing matrix $A=[B|C],$
   then:
\begin{enumerate}
   \item[(i)] For each $1\leq j\leq n$ we have that  one of the following sentences is true with respect to the product $z_jU(A):$
    \begin{enumerate}
         \item \label{abstract-sub-critical-case1-coro-interaction-z-BBU} $j=1$ and $b_1=0,$ then
         \begin{equation*}
        z_jU(A)=\ulz_1U \Big[\!\!\begin{array}{c}0\\0\end{array}\!\!\Big] U\left[\begin{matrix}
             b_2&\cdots &b_r&c_{11}&\cdots&c_{1r}\\
            0 &\cdots &0&c_{21}&\cdots &c_{2r}
        \end{matrix}\right].
        \end{equation*}
        \item \label{abstract-critical-case1-coro-interaction-z-BBU}  There is an $1\leq s \leq r$ and a $k\in \{1,2\}$ such that:
        \begin{equation*}
        z_jU(A)=U  \Big[\!\!\begin{array}{c}b_1\\0\end{array}\!\!\Big]\cdots U  \Big[\!\!\begin{array}{c}b_s\\0\end{array}\!\!\Big]z_kU  \Big[\!\!\begin{array}{c}b_{s+1}\\0\end{array}\!\!\Big]\cdots U
  \Big[\!\!\begin{array}{c}b_r\\0\end{array}\!\!\Big] U(C).
        \end{equation*}
         \item \label{abstract-Not-critical-case1-coro-interaction-z-BBU} There is a $1\leq i\leq n$ such that $i\neq j$ and $z_jU(A)=z_iU(A).$
        \item \label{abstract-Not2-critical-case1-coro-interaction-z-BBU} There is a $1\leq k\leq n$ such that $z_jU(A)=U(A)z_k.$
    \end{enumerate}
     \item[(ii)] For each $1\leq j\leq n$ we have that  one of the following sentences is true with respect to the product $U(A)\ulz_j:$
    \begin{enumerate}
        \item \label{abstract-critical-case2-coro-interaction-z-BBU}  There is an $1\leq s \leq r$ and a $k\in \{1,2\}$ such that:
        \begin{equation*}
        U(A)z_j=U  \Big[\!\!\begin{array}{c}b_1\\0\end{array}\!\!\Big]\cdots U  \Big[\!\!\begin{array}{c}b_s\\0\end{array}\!\!\Big]z_kU  \Big[\!\!\begin{array}{c}b_{s+1}\\0\end{array}\!\!\Big]\cdots U
  \Big[\!\!\begin{array}{c}b_r\\0\end{array}\!\!\Big] U(C).
        \end{equation*}
         \item \label{abstract-Not-critical-case2-coro-interaction-z-BBU} There is a $1\leq i\leq n$ such that $i\neq j$ and $U(A)z_j=U(A)z_i.$
        \item \label{abstract-Not2-critical-case2-coro-interaction-z-BBU} There is a $1\leq k\leq n$ such that $U(A)z_j=z_kU(A).$
    \end{enumerate}
\end{enumerate}
\end{proposition}
\begin{proof}
 Analogous to the proof of Proposition \ref{prop-interaction-z-BBU}, but using Lemma \ref{zUAbs} and Lemma \ref{zA0Abs} instead of Lemma \ref{zU} and Lemma \ref{zA0U}, respectively.
\end{proof}

\begin{corollary}\label{coro-abstract-normal-form-framed-blob-mon}
   Every $X\in Bl_{d,n}$ can be written in the form
   $$X=\langle \alphan,A\rangle U(A)\langle A, \bbeta\rangle,$$ where $A$ is an indexing matrix  and  $\alphan\in\mathrm{Exp}_L(A)$ and $\bbeta\in \mathrm{Exp}_R(B).$
\end{corollary}
\begin{proof}
Completely analogous to the proof of Corollary  \ref{coro-normal-form-framed-blob-mon}.
To be precise, we only need to use Lemma \ref{lemma-normal-form-FrBlob-2} instead  Lemma \ref{lemma-diag-normal-form-FrBlob-1}  and Proposition \ref{prop-abstract-interaction-z-BBU} instead Proposition \ref{prop-interaction-z-BBU}, during the proof of Corollary \ref{coro-normal-form-framed-blob-mon}.
\end{proof}

\begin{theorem}\label{theo-iso-AlgFrblob-DiagFrblob}
    The map $u_i\mapsto  \mathfrak{u}_i, z_i\mapsto \mathfrak{z}_i$ defines an isomorphism  $\Phi:Bl_{d,n}\rightarrow \mathfrak{Bl}_{d,n}.$
\end{theorem}
\begin{proof}The Lemma \ref{lemma-pre-presentation-frblob-digrams} and the Lemma \ref{lemma-diag-normal-form-FrBlob-1} say that $\Phi$ is an epimorphism. Therefore, the  proof follows by proving that $|Bl_{d,n}|\leq |\mathfrak{Bl}_{d,n}|.$ Now, this inequality is a consequence of Corollary \ref{coro-abstract-normal-form-framed-blob-mon}, the Proposition \ref{lemma-equation-n+r} and Theorem \ref{theo-card-Fblob-mon-first}.
\end{proof}

\section{Some combinatorics}
 This section consists of two subsections showing some combinatorial formulas for the numbers appearing in the article. The first subsection provides a formula for computing the numbers $\Omega_r^{(n)}$ and we also show a Pascal triangle for computing them, see Remark \ref{Pascal1}.
The second subsection is devoted to proving the Theorem \ref{theo-alternative-formula-diag-frBlob}, which provides a formula for computing the cardinality of $\mathfrak{Bl}_{d,n}$ in terms of certain numbers $\chi_n^{(n)}$. These numbers can be computed using a Pascal triangle as shown in Remark   \ref{Pascal2}; also we show an explicit formula for compute them.

\subsection{Some formulas for $\Omega_k^{(n)}$}
To obtain a recursive formula for $\Omega_n$, we start by observing that if $A=[B|C]$ with
\begin{equation*}
       B=\left[\begin{matrix}
        b_{1}&\cdots& b_{r}\\
        0 &\cdots & 0
    \end{matrix}\right],\quad
    C=\left[\begin{matrix}
        c_{11} & \cdots & c_{1m} \\
        c_{21}& \cdots & c_{2m}
    \end{matrix}\right]
\end{equation*}
then  $c_{11}>b_{r}.$ This  motivates the following definition.

\begin{definition}\label{def-theta-numbers}
    For $k\in[1, n-1]$ we define $\theta_{k}^{(n)}$ as the number of positive indexing matrices $C=[c_{ij}]$ such that $c_{11}=k.$ We set  $\theta_1^{(1)}:=1.$
\end{definition}

\begin{lemma}\label{lemma-recur-theta-numbers}
For $n\geq 2$ we have:
\begin{enumerate}
\item[(i)] $\theta_1^{(n)} = \mathrm{C}_{n-1}$,
\item[(ii)] $\theta_{n-1}^{(n)}=n-1$,
\item[(iii)] $\theta_{k}^{(n)} =\theta_{k-1}^{(n-1)}+\theta_{k+1}^{(n)}$
for $k\in[2,n-2]$ with  $n>2$.
\end{enumerate}
\end{lemma}

\begin{proof}
  Let $\calA_n$ be the set of all indexing matrices and let $\calA_{n}^{+}$ the set of all the positive indexing matrices. We have  $|\calA_{n}^{+}|=\mathrm{C}_n$, see Remark \ref{nthCatalan}.  Denote by  $\Theta_{k}^{(n)}$  the set of all the positive indexing matrices $C=[c_{ij}]$ such that $c_{11}=k$. Let's prove the first claim:
If $k=1,$ and $[c_{ij}]\in\Theta_1^{(n)},$ then $c_{21}=1,$ that is, the matrices of $\Theta_1^{(n)}$ are:
      \begin{equation*}
          \left[\begin{matrix}
              1 \\1
          \end{matrix}\right]
          \quad\textrm{or in the form }\quad C=\left[\begin{matrix}
              1& c_{12}&\cdots c_{1m}\\
               1& c_{22}&\cdots c_{2m}
          \end{matrix}\right].
      \end{equation*}
Now, counting the number of matrices $C$ is equivalent to counting the number of the following positive non degenerate positive indexed matrices:
      \begin{equation*}
          D= \left[\begin{matrix}
              d_{11}\ldots  d_{1,m-1}\\
                 d_{12}\ldots  d_{2,m-1}
          \end{matrix}\right],\quad \text{with  $d_{ij}:=c_{i,j+1}-1$.}
      \end{equation*}
   Thus, one deduces  that $\theta_{1}^{(n)}=|\Theta_{1}^{(n)}|=\mathrm{C}_{n-1}$.

   For the second claim, note that $\Theta_{n-1}^{(n)}$ consists of matrices that have only one column, so they have the form $C=
\Big[\!\!\begin{array}{c}n-1\\ x\end{array}\!\!\Big]$, where $ x\in[1,n-1],$ therefore $\theta_{n-1}^{(n)}=|\Theta^{(n)}_{n-1}|=n-1.$

Finally, to prove the third claim, we start by observing that the non degenerate positive indexing matrices  are parametrized by the partitions of $N=\{1,\dots,n-1\}\cup\{1',\dots,(n-1)'\}$ whose blocks have cardinality 1 or 2. To be precise, the  matrix  $C=[c_{ij}]$ is parametrized by the partition $P_C$, which  is  formed by blocks $P_j=\{c_{1j},c_{2j}'\}$ and the singleton blocks $\{a\}$, where  $a$ is  not an entry of $C$.
 Note that for blocks $\{a,b'\}$ we have $a\geq b$.
As usual, we represent the partition using diagrams. Below is an example of $C$ and its respective $P_C$.
\begin{equation*}
\begin{tikzpicture}[xscale=0.5,yscale=0.4]
\node[] at (-6,1.5) {$C=\left(\begin{matrix}
2&3&\cdots & n-1\\
1&3 &\cdots & n-1
\end{matrix}\right)\text{\quad and\quad  }
P_C=$};
\draw[thick] (1,0) to [out=90,in=270](2,3);
\draw[thick](1,3)--(1,2.5);
\draw[thick](2,0)--(2,0.5);
\draw[thick] (3,0) to [out=90,in=270](3,3);
\draw[thick] (6,0) to [out=90,in=270](6,3);
\node[below] at (1,0) {$1'$};
\node[below] at (2,0) {$2'$};
\node[below] at (3,0) {$3'$};
\node[below] at (6,0) {$(n-1)'$};
\node[above] at (1,3) {$1$};
\node[above] at (2,3) {$2$};
\node[above] at (3,3) {$3$};
\node[above] at (6,3) {$n-1$};
\node[] at (4.5,1.5) {$\cdots$};
\end{tikzpicture} 
\end{equation*}
Using this diagrammatical representation, we obtain
      \begin{equation}\label{eq-recur-theta-numbers-1}
        \theta_k^{(n)}=1+ \sum_{j=k-1}^{n-2}\theta_{j}^{(n-1)}\quad \text{ for $k\in[2,n-1]$,}
      \end{equation}
where the summand 1  counts the matrix $\Big[\!\!\begin{array}{c}k\\1\end{array}\!\!\Big]$. We will now study the summation above. Note that any other matrix $D\in \Theta_k^{(n)}$ will start either with the same column $\Big[\!\!\begin{array}{c}k\\1\end{array}\!\!\Big]$ or with a column of the form $\Big[\!\!\begin{array}{c}k\\x\end{array}\!\!\Big]$, with $x\in[2,k]$.
In diagrammatic terms we have:
      \begin{equation*}
          \begin{tikzpicture}[xscale=0.5,yscale=0.4]
\node[] at (-1,1.5) {$P_D=$};
\draw[thick] (1,0) to [out=90,in=270](5,3);
\draw[thick](1,3)--(1,2.5);
\draw[thick](2,3)--(2,2.5);
\draw[thick](3,3)--(3,2.5);
\draw[fill] (10,3) circle [radius=0.05];
\draw[fill] (2,0) circle [radius=0.05];
\draw[fill] (3,0) circle [radius=0.05];
\draw[fill] (5,0) circle [radius=0.05];
\draw[fill] (7,0) circle [radius=0.05];
\draw[fill] (7,3) circle [radius=0.05];
\draw[fill] (10,0) circle [radius=0.05];
\node[below] at (1,0) {$1'$};
\node[below] at (2,0) {$2'$};
\node[below] at (3,0) {$3'$};
\node[below] at (5,0) {$k'$};
\node[below] at (7,0) {$(k+1)'$};
\node[above] at (1,3) {$1$};
\node[above] at (2,3) {$2$};
\node[above] at (3,3) {$3$};
\node[above] at (5,3) {$k$};
\node[above] at (7,3) {$k+1$};
\node[above] at (10,3) {$n-1$};
\node[below] at (10,0) {$(n-1)'$};
\node[] at (4,3) {$\cdots$};
\node[] at (4,0) {$\cdots$};
\node[] at (8.5,0) {$\cdots$};
\node[] at (8.5,3) {$\cdots$};
\node[] at (12,1) {or};
\end{tikzpicture} \quad \begin{tikzpicture}[xscale=0.5,yscale=0.4]
\node[] at (-1,1.5) {$P_D=$};
\draw[thick] (3,0) to [out=90,in=270](5,3);
\draw[thick](1,3)--(1,2.5);
\draw[thick](2,3)--(2,2.5);
\draw[thick](3,3)--(3,2.5);
\draw[fill] (10,3) circle [radius=0.05];
\draw[thick](1,0)--(1,0.5);
\draw[fill] (5,0) circle [radius=0.05];
\draw[fill] (7,0) circle [radius=0.05];
\draw[fill] (7,3) circle [radius=0.05];
\draw[fill] (10,0) circle [radius=0.05];
\node[below] at (1,0) {$1'$};
\node[below] at (3,0) {$x'$};
\node[below] at (5,0) {$k'$};
\node[below] at (7,0) {$(k+1)'$};
\node[above] at (1,3) {$1$};
\node[above] at (2,3) {$2$};
\node[above] at (3,3) {$3$};
\node[above] at (5,3) {$k$};
\node[above] at (7,3) {$k+1$};
\node[above] at (10,3) {$n-1$};
\node[below] at (10,0) {$(n-1)'$};
\node[] at (4,3) {$\cdots$};
\node[] at (4,0) {$\cdots$};
\node[] at (8.5,0) {$\cdots$};
\node[] at (8.5,3) {$\cdots$};
\end{tikzpicture} 
      \end{equation*}
where in both cases the first column of the matrix $D$ is represented by the arc connecting the points $k$ with $1'$ or $x'$, respectively. The other columns of the matrix will be represented by connection between points in the top with points in the bottom that satisfy the defining  conditions of the indexing matrix, that is, non-intersecting arcs going down vertically or from right to left (see subsection \ref{def-indexing-matrix}).

Now, to relate all such possible indexing matrices $D$ to a smaller case, we can re-enumerate our diagrams, subtracting one at each point, as follows:
\begin{equation*}
      \begin{tikzpicture}[xscale=0.5,yscale=0.4]
\node[] at (-1,1.5) {$P_D=$};
\draw[thick] (1,0) to [out=90,in=270](5,3);
\draw[thick](1,3)--(1,2.5);
\draw[thick](2,3)--(2,2.5);
\draw[thick](3,3)--(3,2.5);
\draw[fill] (10,3) circle [radius=0.05];
\draw[fill] (2,0) circle [radius=0.05];
\draw[fill] (3,0) circle [radius=0.05];
\draw[fill] (5,0) circle [radius=0.05];
\draw[fill] (7,0) circle [radius=0.05];
\draw[fill] (7,3) circle [radius=0.05];
\draw[fill] (10,0) circle [radius=0.05];
\node[below] at (1,0) {$0'$};
\node[below] at (2,0) {$1'$};
\node[below] at (3,0) {$2'$};
\node[below] at (5,0) {$(k-1)'$};
\node[below] at (7,0) {$k'$};
\node[above] at (1,3) {$0$};
\node[above] at (2,3) {$1$};
\node[above] at (3,3) {$2$};
\node[above] at (5,3) {$k-1$};
\node[above] at (7,3) {$k$};
\node[above] at (10,3) {$n-2$};
\node[below] at (10,0) {$(n-2)'$};
\node[] at (4,3) {$\cdots$};
\node[] at (4,0) {$\cdots$};
\node[] at (8.5,0) {$\cdots$};
\node[] at (8.5,3) {$\cdots$};
\node[] at (12,1) {or};
\end{tikzpicture} \quad \begin{tikzpicture}[xscale=0.5,yscale=0.4]
\node[] at (-1,1.5) {$P_D=$};
\draw[thick] (3,0) to [out=90,in=270](5,3);
\draw[thick](1,3)--(1,2.5);
\draw[thick](2,3)--(2,2.5);
\draw[thick](3,3)--(3,2.5);
\draw[fill] (10,3) circle [radius=0.05];
\draw[thick](1,0)--(1,0.5);
\draw[fill] (5,0) circle [radius=0.05];
\draw[fill] (7,0) circle [radius=0.05];
\draw[fill] (7,3) circle [radius=0.05];
\draw[fill] (10,0) circle [radius=0.05];
\node[below] at (1,0) {$0'$};
\node[below] at (3,0) {$y'$};
\node[below] at (5,0) {$(k-1)'$};
\node[below] at (7,0) {$k'$};
\node[above] at (1,3) {$0$};
\node[above] at (2,3) {$1$};
\node[above] at (3,3) {$2$};
\node[above] at (5,3) {$k-1$};
\node[above] at (7,3) {$k$};
\node[above] at (10,3) {$n-2$};
\node[below] at (10,0) {$(n-2)'$};
\node[] at (4,3) {$\cdots$};
\node[] at (4,0) {$\cdots$};
\node[] at (8.5,0) {$\cdots$};
\node[] at (8.5,3) {$\cdots$};
\end{tikzpicture} 
  \end{equation*}
  In the first case, we do not consider the arc from $k-1$ to $0',$ then we can see that, to complete our diagram we have exactly $\sum_{j=k}^{n-2}\theta_j^{(n-1)}$ possibilities. In the second case we consider  the arc from $k-1$ to $y',$ for $1\leq y\leq k-1,$ then we are considering all the indexing matrices in the set $\Theta_{k-1}^{(n-1)},$ therefore we obtain other $\theta_{k-1}^{(n-1)}$ possibilities to form one of the desired matrices $D.$ Putting all together we obtain (\ref{eq-recur-theta-numbers-1}).
  If we take $k\in [2,n-2]$ then (\ref{eq-recur-theta-numbers-1}) is valid for $k$ and $k+1.$ In the former case the equation looks like:
  \begin{equation}\label{eq-recur-theta-numbers-2}
        \theta_{k+1}^{(n)}=1+ \sum_{j=k}^{n-2}\theta_{j}^{(n-1)}.
      \end{equation}
Therefore, if we subtract (\ref{eq-recur-theta-numbers-2}) to (\ref{eq-recur-theta-numbers-1}), we obtain $\theta_{k}^{(n)}-\theta_{k+1}^{(n)}=\theta_{k-1}^{(n-1)}$, as desired.

\end{proof}
A consequence of the proof of Lemma \ref{lemma-recur-theta-numbers} is the following corollary.
\begin{corollary}\label{coro-recur-theta-numbers}
    For each $n>3$ and each $k\in [2,n-2]$ we have:
    \begin{equation*}
            \theta_k^{(n)}=1+\sum_{j=k-1}^{n-2}\theta_j^{(n-1)}.
        \end{equation*}
\end{corollary}
\begin{remark}\rm
The Lemma \ref{lemma-recur-theta-numbers} yields the following Pascal triangle, which allows computing $\theta_k^{(n)}$'s recursively as follows:
\begin{equation}\label{eq-tikz-theta-pascal}
    \begin{tikzpicture}[xscale=0.3,yscale=0.4]

\node[] at (0,2) {$\theta_k^{(n)}$};
\node[] at (0,0) {$1$};
\node[] at (-2,-2) {$2$};
\node[] at (2,-2) {$2$};
\node[] at (-4,-4) {$5$};
\node[] at (0,-4) {$5$};
\node[] at (4,-4) {$3$};
\node[] at (-6,-6) {$14$};
\node[] at (-2,-6) {$14$};
\node[] at (2,-6) {$9$};
\node[] at (6,-6) {$4$};
\node[] at (-8,-8) {$42$};
\node[] at (-4,-8) {$42$};
\node[] at (0,-8) {$28$};
\node[] at (4,-8) {$14$};
\node[] at (8,-8) {$5$};
\node[] at (-10,-10) {$132$};
\node[] at (-6,-10) {$132$};
\node[] at (-2,-10) {$90$};
\node[] at (2,-10) {$48$};
\node[] at (6,-10) {$20$};
\node[] at (10,-10) {$6$};
\node[] at (-20,0) {$n=2$};
\node[] at (-20,-2) {$n=3$};
\node[] at (-20,-4) {$n=4$};
\node[] at (-20,-6) {$n=5$};
\node[] at (-20,-8) {$n=6$};
\node[] at (-20,-10) {$n=7$};
\draw[->] (-1.5,-2.5)--(-0.5,-3.5);
\draw[->] (0.5,-4.5)--(1.5,-5.5);
\draw[->] (-3.5,-4.5)--(-2.5,-5.5);
\draw[->] (2.5,-6.5)--(3.5,-7.5);
\draw[->] (-1.5,-6.5)--(-0.5,-7.5);
\draw[->] (-5.5,-6.5)--(-4.5,-7.5);
\draw[->] (4.5,-8.5)--(5.5,-9.5);
\draw[->] (0.5,-8.5)--(1.5,-9.5);
\draw[->] (-3.5,-8.5)--(-2.5,-9.5);
\draw[->] (-7.5,-8.5)--(-6.5,-9.5);
\draw[->] (3-2,-2)--(1-2,-2);
\draw[->] (3,-4)--(1,-4);
\draw[->] (3-4,-4)--(1-4,-4);
\draw[->] (5,-6)--(3,-6);
\draw[->] (1,-6)--(-1,-6);
\draw[->] (-3,-6)--(-5,-6);
\draw[->] (7,-8)--(5,-8);
\draw[->] (3,-8)--(1,-8);
\draw[->] (-1,-8)--(-3,-8);
\draw[->] (-5,-8)--(-7,-8);
\draw[->] (9,-10)--(7,-10);
\draw[->] (5,-10)--(3,-10);
\draw[->] (1,-10)--(-1,-10);
\draw[->] (-3,-10)--(-5,-10);
\draw[->] (-7,-10)--(-9,-10);
\end{tikzpicture} 
\end{equation}
\end{remark}

\begin{lemma}\label{lemma-rel1-omega-vs-theta-numbers}
    For all $n\geq 1$ we have $\Omega_0^{(n)}= \mathrm{C}_n$,  $\Omega_n^{(n)}=1$ and
\begin{equation}\label{eq-lemma-rel1-omega-vs-theta-numbers}
\Omega_r^{(n)}=\sum_{i=r-1}^{n-1}\binom{i}{r-1}\left(1+\sum_{j=i+1}^{n-1}\theta_j^{(n)}\right), \qquad r \in [1,n-1].
\end{equation}
\end{lemma}
\begin{proof}
Since $\Omega_0^{(n)}$ is the number of positive indexing matrices, it follows  that $\Omega_0^{(n)}=\mathrm{C}_n$, see Remark \ref{nthCatalan}. Now, if $r=n$ then the only indexing matrix with blob-rank  $n,$ is
       \begin{equation*}
           A=\left[\begin{matrix}
               0 & 1& \cdots & n-1 \\
               0 & 0 & \cdots & 0
           \end{matrix}\right].
       \end{equation*}
Hence $\Omega_n^{(n)}=1.$

Finally, suppose that $r \in [1,n-1]$. Under this assumption, any indexing matrix with blob-rank $r$ is either an initial matrix with blob rank $r,$ or one of the form $A=[B|C],$ where $B$ is an initial indexing matrix of blob-rank  $r$ and $C$ is a non degenerated positive indexing matrix, that is, we have the following cases:
       \begin{equation*}
           A=\left[\begin{matrix}
               a_1 &\cdots & a_r \\
               0 & \cdots & 0
           \end{matrix}\right]\quad \textrm{or}\quad
           A=\left[\begin{matrix}
               b_1 & \cdots & b_r& c_{11}&\cdots &c_{1m}\\
               0 & \cdots &0 & c_{21}&\cdots & c_{2m}
           \end{matrix}\right].
       \end{equation*}
For the first case note that $0\leq a_1<a_2<\cdots<a_r\leq n-1$. In particular, $r-1\leq a_r\leq n-1.$ Now, for each choice of $r-1\leq a_r\leq n-1,$ we have exactly $\binom{a_r}{r-1}$ possible choices for $0\leq a_1<a_2<\cdots<a_{r-1}.$
We then deduce that there are $\sum_{i=r-1}^{n-1}\binom{i}{r-1}$ possible initial indexing matrices that have blob rank $r$.
For the second case,  $A=[B|C],$ we have the restrictions $0\leq b_1<b_2<\cdots<b_r<c_{11}<\cdots c_{1m}.$ In particular, $r-1\leq b_r<n-1$. Then we have  $\sum_{i=r-1}^{n-2}\binom{i}{r-1}$ possible choices for the initial matrix $B.$ Now, using Definition \ref{def-theta-numbers}, for each choice of $B,$  there are $\sum_{j=b_r+1}^{n-1}\theta_{j}^{(n)}$  positive indexing matrices $C,$ such that $A=[B|C]$ is an indexing matrix. Therefore we obtain
       \begin{equation*}
           \Omega_r^{(n)}=\sum_{i=r-1}^{n-1}\binom{i}{r-1}\left(1+\sum_{j=i+1}^{n-1}\theta_j^{(n)}\right)
       \end{equation*}
       as desired. Note that for $r\in [1, n-1]$ and $ i=n-1$, we have $\sum_{j=i+1}^{n-1}\theta_{j}^{(n)}=0$.
\end{proof}
 A combination of Lemma \ref{lemma-rel1-omega-vs-theta-numbers} with Corollary \ref{coro-recur-theta-numbers} leaves the following corollary.
\begin{corollary}\label{eq-rel2-omega-vs-theta-numbers} For $n>1$, we have
\begin{equation*}
   \Omega_r^{(n)}=\sum_{i=r-1}^{n-1}\binom{i}{r-1}\theta_{i+2}^{(n+1)},\quad
    r\in [1, n-1].
\end{equation*}
\end{corollary}
\begin{remark}\label{Pascal1}\rm
Note that, combining the Lemma \ref{lemma-recur-theta-numbers}, Corollary \ref{coro-recur-theta-numbers} and (\ref{eq-tikz-theta-pascal}), we obtain the following Pascal triangle to recursively calculate the $\Omega_k^{(n)}$'s.

\begin{equation}\label{eq-tikz-Omega-pascal}
    \begin{tikzpicture}[xscale=0.3,yscale=0.4]
\node[] at (0,2) {$\Omega_k^{(n)}:$};
\node[] at (0,0) {$1$};
\node[] at (-2,-2) {$1$};
\node[] at (2,-2) {$1$};
\node[] at (-4,-4) {$2$};
\node[] at (0,-4) {$3$};
\node[] at (4,-4) {$1$};
\node[] at (-6,-6) {$5$};
\node[] at (-2,-6) {$9$};
\node[] at (2,-6) {$5$};
\node[] at (6,-6) {$1$};
\node[] at (-8,-8) {$14$};
\node[] at (-4,-8) {$28$};
\node[] at (0,-8) {$20$};
\node[] at (4,-8) {$7$};
\node[] at (8,-8) {$1$};
\node[] at (-10,-10) {$42$};
\node[] at (-6,-10) {$90$};
\node[] at (-2,-10) {$75$};
\node[] at (2,-10) {$35$};
\node[] at (6,-10) {$9$};
\node[] at (10,-10) {$1$};
\node[] at (-20,0) {$n=0$};
\node[] at (-20,-2) {$n=1$};
\node[] at (-20,-4) {$n=2$};
\node[] at (-20,-6) {$n=3$};
\node[] at (-20,-8) {$n=4$};
\node[] at (-20,-10) {$n=5$};
\draw[->] (-0.5,-0.5)--(-1.5,-1.5);
\draw[->] (1.5,-2.5)--(0.5,-3.5);
\draw[->] (-2.5,-2.5)--(-3.5,-3.5);
\draw[->] (3.5,-4.5)--(2.5,-5.5);
\draw[->] (-0.5,-4.5)--(-1.5,-5.5);
\draw[->] (-4.5,-4.5)--(-5.5,-5.5);
\draw[->] (5.5,-6.5)--(4.5,-7.5);
\draw[->] (1.5,-6.5)--(0.5,-7.5);
\draw[->] (-2.5,-6.5)--(-3.5,-7.5);
\draw[->] (-6.5,-6.5)--(-7.5,-7.5);
\draw[->] (-8.5,-8.5)--(-9.5,-9.5);
\draw[->] (-4.5,-8.5)--(-5.5,-9.5);
\draw[->] (-0.5,-8.5)--(-1.5,-9.5);
\draw[->] (3.5,-8.5)--(2.5,-9.5);
\draw[->] (7.5,-8.5)--(6.5,-9.5);
\end{tikzpicture} 
\end{equation}
Observe that the  arrows in the above Pascal triangle express the following relation for all $n>1$
\begin{equation}\label{eq-Omega-vs-theta}
\Omega_{k}^{(n)}=\theta_{2k+1}^{(n+k+1)},\quad \textrm{for a fixed}\quad k\in[0, n-1]
\end{equation} See the odd diagonals of the Pascal triangle of (\ref{eq-tikz-theta-pascal}).
\end{remark}
\subsection{Alternative formula}
To state the main result of this section, we need to introduce the numbers $\chi_{k}^{(n)}$, which are the number of TL-diagrams with a given number of arcs exposed on the left side. To be precise, we make the following definition.

\begin{definition}\label{def-chi-numbers}
Given $k\in[1,n]$, we define the number $\chi_{k}^{(n)}$ as the number of TL-diagrams on $n$ points with exactly $k$ arcs exposed to the left side of the diagram.
\end{definition}

For example, in the diagram $D$ of equation (\ref{eq-tikz-TL-diagram-1}) there are exactly four exposed arcs on the left side.

\begin{theorem}\label{theo-alternative-formula-diag-frBlob}
The cardilnality of $\mathfrak{Bl}_{d,n},$ is given by:
\begin{equation}\label{EQ-lemma-lower-bound-for-diag-frBlob}
|\mathfrak{Bl}_{d,n}|=  d^{n}\sum_{k^=1}^{n}\chi_{k}^{(n)}(1+d)^{k}.
\end{equation}
\end{theorem}

\begin{proof}
Let $k$ be in $[1,n]$. The Definition \ref{def-chi-numbers} says that the number $\chi_{k}^{(n)}$ is the number of TL-diagrams on $n$ points, with exactly $k$ arcs exposed to the left. Therefore, for each $j\in [0,k]$ there are exactly $\chi_{k}^{(n)}\binom{k}{j}$ different TL-diagrams with $k$ arcs exposed to the left, and with $j$ of them decorated by a blob.
Now, each blobbed arc splits in  two components and each not blobbed arc counts as a component itself. Also, note that each component can contain a number of beads ranging from $0$ to $d-1$. Therefore we have  $\chi_{k}^{(n)}\sum_{j=0}^{k}\binom{k}{j}d^{n-j}d^{2j}=d^{n-k}\chi_{k}^{(n)}(d+d^2)^k$ different $d-$framed blob diagrams with $k$ arcs exposed to the left side of the diagram. Thus, we have:
    \begin{equation*}
    |\mathfrak{Bl}_{d,n}|=d^{n-k}\sum_{k=1}^{n}\chi_{k}^{(n)}(d+d^2)^k=d^n\sum_{k=1}^{n}\chi_{k}^{(n)}(1+d)^k.
    \end{equation*}
\end{proof}
From Theorem \ref{theo-card-Fblob-mon-first}, Theorem \ref{theo-alternative-formula-diag-frBlob} and Theorem \ref{theo-iso-AlgFrblob-DiagFrblob}, we get the following corollary.
\begin{corollary}
     The cardinality of ${Bl}_{d,n}$ is given by:
    \begin{equation}
        |{Bl}_{d,n}|= d^{n}\sum_{k=0}^{n}\Omega_k^{(n)}d^k=  d^{n}\sum_{k^=1}^{n}\chi_{k}^{(n)}(1+d)^{k}.
    \end{equation}
\end{corollary}

\begin{corollary}
The following equation holds.
\begin{equation}
  \sum_{k=0}^{n}\Omega_k^{(n)}d^k=\sum_{k=1}^{n}\chi_{k}^{(n)}(1+d)^k.
    \end{equation}
\end{corollary}
\begin{proof}
    It follows directly from Theorem \ref{theo-card-Fblob-mon-first} and Theorem \ref{theo-alternative-formula-diag-frBlob}.
\end{proof}
\begin{remark}\rm
Note that in a TL-diagram, each arc exposed to the left side can \emph{cover} a certain number of arcs (the not exposed to the left side ones. For example, in the following TL-diagram, we have colored the arcs exposed to the left side as: the red arc cover one arc, the blue arc cover four arcs, while the green and the purple arcs do not cover any other arcs.

\begin{equation}\label{eq-tikz-TL-diagram-w-color}
\begin{tikzpicture}[xscale=0.3,yscale=0.3]
\node[] at (-1,1.5) {$D=$};
\draw[thick,red] (1,0) to [out=90,in=90](4,0);
\draw[thick,black] (2,0) to [out=90,in=90](3,0);
\draw[thick,green] (5,0) to [out=90,in=90](6,0);
\draw[thick,blue] (7,0) to [out=90,in=270](3,3);
\draw[thick,black] (8,0) to [out=90,in=270](8,3);
\draw[thick,purple] (2,3) to [out=270,in=270](1,3);
\draw[thick,black] (4,3) to [out=270,in=270](7,3);
\draw[thick,black] (5,3) to [out=270,in=270](6,3);
\draw[thick,black] (9,0) to [out=90,in=270](9,3);


\end{tikzpicture} 
\end{equation}
\end{remark}

In order to have  explicit formulae for the $\chi_k^{(n)}$'s, we will use the so called {\it
planar $2-$partition.} Put $K=\{1,2,\dots,k,k+1,\dots ,2k\},$ a planar $2-$partition of $K$ is a partition $P=\{P_1,P_2,\dots,P_k\}$ of $K$ such that each block $P_j$ contains exactly two elements of $K$ and for each three elements of $K,$ $a<b<c$ we have that if $P_j=\{a,c\}$ is one of the blocks of $P,$ then there is a $b'\neq b$ such that $a<b'<c$ and $P_i=\{b,b'\}$ is also a block of $P.$ It is well known that the number of  possible planar $2-$partitions of $K$ is equal to the Catalan number $\mathrm{C}_k$. Notice that a planar $2-$partition can be represented by a diagram consisting of $k$ non-intersecting arcs, with endpoints located on two parallel lines and numbered clockwise. Note that the number of points on each line does not need to be the same; it could even be that such a diagram has all points on the same line.
For example, if $k=3$, then the following diagrams are the $\mathrm{C}_3=5$ possible planar $2-$partitions of $K=\{1,2,\dots,6\}$.
\begin{equation}\label{eq-ex-diag-catalan}
    \begin{tikzpicture}[xscale=0.5,yscale=0.3]
\draw[thick] (1,0) to [out=90,in=270](1,3);
\draw[thick] (2,0) to [out=90,in=270](2,3);
\draw[thick] (3,0) to [out=90,in=90](4,0);
\node[above] at (1,3) {$1$};
\node[above] at (2,3) {$2$};
\node[below] at (1,0) {$6$};
\node[below] at (2,0) {$5$};
\node[below] at (3,0) {$4$};
\node[below] at (4,0) {$3$};
\node[] at (5,0.5) {$;$};
\draw[thick] (6,0) to [out=90,in=90](7,0);
\draw[thick] (6,3) to [out=270,in=270](7,3);
\draw[thick] (8,0) to [out=90,in=90](9,0);
\node[above] at (6,3) {$1$};
\node[above] at (7,3) {$2$};
\node[below] at (6,0) {$6$};
\node[below] at (7,0) {$5$};
\node[below] at (8,0) {$4$};
\node[below] at (9,0) {$3$};
\node[] at (10,0.5) {$;$};
\draw[thick] (11,0) to [out=90,in=90](14,0);
\draw[thick] (12,0) to [out=90,in=90](13,0);
\draw[thick] (11,3) to [out=270,in=270](12,3);
\node[above] at (11,3) {$1$};
\node[above] at (12,3) {$2$};
\node[below] at (11,0) {$6$};
\node[below] at (12,0) {$5$};
\node[below] at (13,0) {$4$};
\node[below] at (14,0) {$3$};
\node[] at (15,0.5) {$;$};
\draw[thick] (16,0) to [out=90,in=90](17,0);
\draw[thick] (16,3) to [out=270,in=90](18,0);
\draw[thick] (17,3) to [out=270,in=90](19,0);
\node[above] at (16,3) {$1$};
\node[above] at (17,3) {$2$};
\node[below] at (16,0) {$6$};
\node[below] at (17,0) {$5$};
\node[below] at (18,0) {$4$};
\node[below] at (19,0) {$3$};
\node[] at (20,0.5) {$;$};
\draw[thick] (21,0) to [out=90,in=270](21,3);
\draw[thick] (22,0) to [out=90,in=90](23,0);
\draw[thick] (24,0) to [out=90,in=270](22,3);
\node[above] at (21,3) {$1$};
\node[above] at (22,3) {$2$};
\node[below] at (21,0) {$6$};
\node[below] at (22,0) {$5$};
\node[below] at (23,0) {$4$};
\node[below] at (24,0) {$3$};
\end{tikzpicture} 
\end{equation}
From now on we set $\mathrm{C}_0:=1.$ Consider the TL-diagram of
(\ref{eq-tikz-TL-diagram-w-color}), and we wonder, how many TL-diagrams can we form from it, leaving all the exposed arcs on the left fixed? The blue arc is seen to generate $\mathrm{C}_4$ different configurations for the $4$ arcs covered by it. Analogously, we can say that the red arc generate $\mathrm{C}_1$ configurations, and the purple and the green only $\mathrm{C}_0$ configurations,  respectively. Therefore, there are exactly $\mathrm{C}_0\mathrm{C}_4\mathrm{C}_0\mathrm{C}_1$ different TL-diagrams with the same four arcs exposed to the left side, as the one in (\ref{eq-tikz-TL-diagram-w-color}). Note that we expressed the number $\mathrm{C}_0\mathrm{C}_4\mathrm{C}_0\mathrm{C}_1$ this way in order to impose a clockwise order on the arcs exposed to the left. Now note that if we change the order of the factors, for example as $\mathrm{C}_4\mathrm{C}_0\mathrm{C}_1\mathrm{C}_0$ we can also construct a TL-diagrams with the first exposed arc, covering $4$ arcs, the second exposed arc covering $0$ arcs, the third one covering $1$ arc and the last one covering $0$ arcs. For example:

\begin{equation}\label{eq-tikz-TL-diagram-w-color2}
\begin{tikzpicture}[xscale=0.3,yscale=0.3]
\node[] at (-1,1.5) {$D'=$};
\draw[thick,red] (1,0) to [out=90,in=90](2,0);
\draw[thick,black] (2,3) to [out=270,in=270](3,3);
\draw[thick,green] (3,0) to [out=90,in=90](6,0);
\draw[thick,blue] (7,0) to [out=90,in=90](8,0);
\draw[thick,black] (9,3) to [out=270,in=270](8,3);
\draw[thick,purple] (9,0) to [out=90,in=270](1,3);
\draw[thick,black] (4,3) to [out=270,in=270](7,3);
\draw[thick,black] (5,3) to [out=270,in=270](6,3);
\draw[thick,black] (4,0) to [out=90,in=90](5,0);


\end{tikzpicture} 
\end{equation}
Now, if we left fixed the exposed arcs in the diagram of (\ref{eq-tikz-TL-diagram-w-color2}), we can check that there are exactly $\mathrm{C}_4\mathrm{C}_0\mathrm{C}_1\mathrm{C}_0$ possible configurations.
More generally, note that in the  diagrams of (\ref{eq-tikz-TL-diagram-w-color}) and (\ref{eq-tikz-TL-diagram-w-color2}), we have $9$ arcs, where exactly four of them are exposed to the left. Therefore, we have five covered arcs, this gives the possibility to build other TL-diagrams with four exposed and five covered arcs. For example, in $D''$ below, we have that the first exposed arc, covers $3$ arcs, the second exposed arc covers one arc, the third exposed arc covers $0$ arcs and the last exposed arc, covers $1$ arc.
\begin{equation}\label{eq-tikz-TL-diagram-w-color3}
\begin{tikzpicture}[xscale=0.3,yscale=0.4]
\node[] at (-1,1.5) {$D''=$};
\draw[thick,red] (1,0) to [out=90,in=90](4,0);
\draw[thick,black] (2,3) to [out=270,in=270](3,3);
\draw[thick,green] (5,0) to [out=90,in=90](6,0);
\draw[thick,blue] (7,0) to [out=90,in=270](9,3);
\draw[thick,purple] (1,3) to [out=270,in=270](8,3);
\draw[thick,black] (4,3) to [out=270,in=270](7,3);
\draw[thick,black] (5,3) to [out=270,in=270](6,3);
\draw[thick,black] (2,0) to [out=90,in=90](3,0);
\draw[thick,black] (8,0) to [out=90,in=90](9,0);


\end{tikzpicture} 
\end{equation}
Thus, we can generate exactly $\mathrm{C}_3\mathrm{C}_1\mathrm{C}_0\mathrm{C}_1$ different TL-diagrams with the same chosen exposed arcs as in the diagram of (\ref{eq-tikz-TL-diagram-w-color3}). The above analysis helps to deduce the following corollary:

\begin{lemma}\label{lemma-formula-chi-numbers}
For $k\in[1,n]$, we have:
 \begin{equation}\label{eq-formula-chi-numbers}
        \chi_{k}^{(n)}=\sum_{\bi\in\calQ(n,k)}\mathrm{C}_{i_1}\cdots \mathrm{C}_{i_k},
    \end{equation}
    where
    \begin{equation*}
         \calQ(n,k):=\{(i_1,\dots,i_k)\in [0, n-1]^k\, ;\, i_1+\dots+i_k=n-k\}.
    \end{equation*}
\end{lemma}

\begin{proof}
   Let $\bi=(i_1,\dots,i_k)\in \calQ(n,k).$ We will show that we can always construct $\mathrm{C}_{i_1}\cdots \mathrm{C}_{i_k}$ different TL diagrams with $k$ exposed arcs and such that the $j-$th arc (in the order explained above) covers exactly $i_j$ arcs. We start by defining an algorithm to build one of them. (Recall that the edge points are enumerated clockwise, as in (\ref{eq-ex-diag-catalan})).
   \begin{enumerate}
       \item[(i)] First, we start by drawing an arc connecting the points $1$ and $2+2i_1$. We then proceed to connect the points from $2$ to $1+2i_1$, taking consecutive pairs, from an even number to an odd number.
       \item[(ii)] Second, assuming that we already have drawn the first $j-1$ exposed arcs, (in clockwise order), covering, respectively, $i_1,\dots,i_{j-1}$ arcs. Thus, the greatest connected point at this moment is $\sum_{r=1}^{j-1}(2+2i_r).$ Then we draw the $j-$th exposed arc connecting the points $1+\sum_{r=1}^{j-1}(2+2i_r)$  and  $\sum_{r=1}^{j}(2+2i_r).$  We then connect points from $2+\sum_{r=1}^{j-1}(2+2i_r)$ to $\sum_{r=1}^{j}(2+2i_r)-1,$ by taking consecutive pairs, from an even number to an odd number.
   \end{enumerate}
 Using the above algorithm recursively, we get a $D$ satisfying our requirements.
 Note that for each $j\in[1,k]$ we have $\mathrm{C}_{i_j}$ possible configurations for the arcs covered by the $j-$th exposed arc, we just have chosen one of them, but we can access any other by modifying the way we choose the connected points for each covered arc. More generally, for $\bi=(i_1,\dots,i_k)$ there are exactly $\mathrm{C}_{i_1}\cdots \mathrm{C}_{i_k}$ possible TL-diagrams with the same exposed arcs as the diagram obtained by our algorithm.
   Since the above construction can be done for any arbitrary element $\bi\in\calQ(n,k),$ we deduce that
   \begin{equation*}       \chi_{k}^{(n)}\geq\sum_{\bi\in\calQ(n,k)}\mathrm{C}_{i_1}\cdots \mathrm{C}_{i_k}.
   \end{equation*}
Finally, the lemma follows from the fact that any TL-diagram can be obtained following the  algorithm above, or one of the possible modifications of it. In fact, If $D'$ is a TL-diagram, with $k$ arcs exposed to the left, we can associate an element $\bi=(i_1,\dots,i_k)\in\calQ(n,k)$ simply by defining $i_j$ as the number of arcs covered by the $j-$th exposed arc. Now $D'$ is one of the $C_{i_1}\cdots C_{i_k}$ possible configurations associated to $\bi.$ Therefore we have:
\begin{equation*}       \chi_{k}^{(n)}\leq\sum_{\bi\in\calQ(n,k)}\mathrm{C}_{i_1}\cdots \mathrm{C}_{i_k}.
   \end{equation*}
And we are done.
\end{proof}

\begin{example}
Here we show how the algorithm given in Lemma \ref{lemma-formula-chi-numbers} works, for
     $n=9,k=4$ and $\bi=(3,1,0,1)\in \calQ(9,4)$.
The first step is to connect the point $1$ with $8=2+2i_1$, and then to connect the points from $2$ to $7$ by taking consecutive pairs, i.e.,
        \begin{equation}\label{eq-tikz-TL-diagram-w-color4}
\begin{tikzpicture}[xscale=0.5,yscale=0.5]
\node[] at (-1,1.5) {$D=$};
\draw[thick,black] (2,3) to [out=270,in=270](3,3);
\draw[thick,purple] (1,3) to [out=270,in=270](8,3);
\draw[thick,black] (6,3) to [out=270,in=270](7,3);
\draw[thick,black] (4,3) to [out=270,in=270](5,3);

\draw[fill] (1,0) circle [radius=0.05];
\draw[fill] (2,0) circle [radius=0.05];
\draw[fill] (3,0) circle [radius=0.05];
\draw[fill] (4,0) circle [radius=0.05];
\draw[fill] (5,0) circle [radius=0.05];
\draw[fill] (6,0) circle [radius=0.05];
\draw[fill] (7,0) circle [radius=0.05];
\draw[fill] (8,0) circle [radius=0.05];
\draw[fill] (9,0) circle [radius=0.05];
\draw[fill] (9,3) circle [radius=0.05];
\node[above] at (1,3) {$1$};
\node[above] at (2,3) {$2$};
\node[above] at (3,3) {$3$};
\node[above] at (4,3) {$4$};
\node[above] at (5,3) {$5$};
\node[above] at (6,3) {$6$};
\node[above] at (7,3) {$7$};
\node[above] at (8,3) {$8$};
\node[above] at (9,3) {$9$};
\node[below] at (9,0) {$10$};
\node[below] at (8,0) {$11$};
\node[below] at (7,0) {$12$};
\node[below] at (6,0) {$13$};
\node[below] at (5,0) {$14$};
\node[below] at (4,0) {$15$};
\node[below] at (3,0) {$16$};
\node[below] at (2,0) {$17$};
\node[below] at (1,0) {$18$};

\end{tikzpicture} 
\end{equation}
The second step is to connect the point $9=1+2+2i_1$ with $12:=\sum_{j=1}^{2}(2+2i_j)$, and then to connect the points $10$ and $11$:
    \begin{equation}\label{eq-tikz-TL-diagram-w-color5}
\begin{tikzpicture}[xscale=0.5,yscale=0.5]
\node[] at (-1,1.5) {$D=$};
\draw[thick,black] (2,3) to [out=270,in=270](3,3);
\draw[thick,blue] (7,0) to [out=90,in=270](9,3);
\draw[thick,purple] (1,3) to [out=270,in=270](8,3);
\draw[thick,black] (6,3) to [out=270,in=270](7,3);
\draw[thick,black] (4,3) to [out=270,in=270](5,3);
\draw[thick,black] (8,0) to [out=90,in=90](9,0);

\draw[fill] (1,0) circle [radius=0.05];
\draw[fill] (2,0) circle [radius=0.05];
\draw[fill] (3,0) circle [radius=0.05];
\draw[fill] (4,0) circle [radius=0.05];
\draw[fill] (5,0) circle [radius=0.05];
\draw[fill] (6,0) circle [radius=0.05];
\draw[fill] (7,0) circle [radius=0.05];
\draw[fill] (8,0) circle [radius=0.05];
\draw[fill] (9,0) circle [radius=0.05];
\draw[fill] (9,3) circle [radius=0.05];
\node[above] at (1,3) {$1$};
\node[above] at (2,3) {$2$};
\node[above] at (3,3) {$3$};
\node[above] at (4,3) {$4$};
\node[above] at (5,3) {$5$};
\node[above] at (6,3) {$6$};
\node[above] at (7,3) {$7$};
\node[above] at (8,3) {$8$};
\node[above] at (9,3) {$9$};
\node[below] at (9,0) {$10$};
\node[below] at (8,0) {$11$};
\node[below] at (7,0) {$12$};
\node[below] at (6,0) {$13$};
\node[below] at (5,0) {$14$};
\node[below] at (4,0) {$15$};
\node[below] at (3,0) {$16$};
\node[below] at (2,0) {$17$};
\node[below] at (1,0) {$18$};

\end{tikzpicture} 
\end{equation}
 Following the algorithm until the end, we obtain:
   \begin{equation}\label{eq-tikz-TL-diagram-w-color6}
\begin{tikzpicture}[xscale=0.5,yscale=0.5]
\node[] at (-1,1.5) {$D=$};
\draw[thick,red] (1,0) to [out=90,in=90](4,0);
\draw[thick,black] (2,3) to [out=270,in=270](3,3);
\draw[thick,green] (5,0) to [out=90,in=90](6,0);
\draw[thick,blue] (7,0) to [out=90,in=270](9,3);
\draw[thick,purple] (1,3) to [out=270,in=270](8,3);
\draw[thick,black] (6,3) to [out=270,in=270](7,3);
\draw[thick,black] (4,3) to [out=270,in=270](5,3);
\draw[thick,black] (2,0) to [out=90,in=90](3,0);
\draw[thick,black] (8,0) to [out=90,in=90](9,0);

\draw[fill] (1,0) circle [radius=0.05];
\draw[fill] (2,0) circle [radius=0.05];
\draw[fill] (3,0) circle [radius=0.05];
\draw[fill] (4,0) circle [radius=0.05];
\draw[fill] (5,0) circle [radius=0.05];
\draw[fill] (6,0) circle [radius=0.05];
\draw[fill] (7,0) circle [radius=0.05];
\draw[fill] (8,0) circle [radius=0.05];
\draw[fill] (9,0) circle [radius=0.05];
\draw[fill] (9,3) circle [radius=0.05];
\node[above] at (1,3) {$1$};
\node[above] at (2,3) {$2$};
\node[above] at (3,3) {$3$};
\node[above] at (4,3) {$4$};
\node[above] at (5,3) {$5$};
\node[above] at (6,3) {$6$};
\node[above] at (7,3) {$7$};
\node[above] at (8,3) {$8$};
\node[above] at (9,3) {$9$};
\node[below] at (9,0) {$10$};
\node[below] at (8,0) {$11$};
\node[below] at (7,0) {$12$};
\node[below] at (6,0) {$13$};
\node[below] at (5,0) {$14$};
\node[below] at (4,0) {$15$};
\node[below] at (3,0) {$16$};
\node[below] at (2,0) {$17$};
\node[below] at (1,0) {$18$};

\end{tikzpicture} 
\end{equation}
 Below we show all the  $\mathrm{C}_3\mathrm{C}_1\mathrm{C}_0\mathrm{C}_1=5$ possible TL-diagrams with the same exposed arcs as diagram $D.$
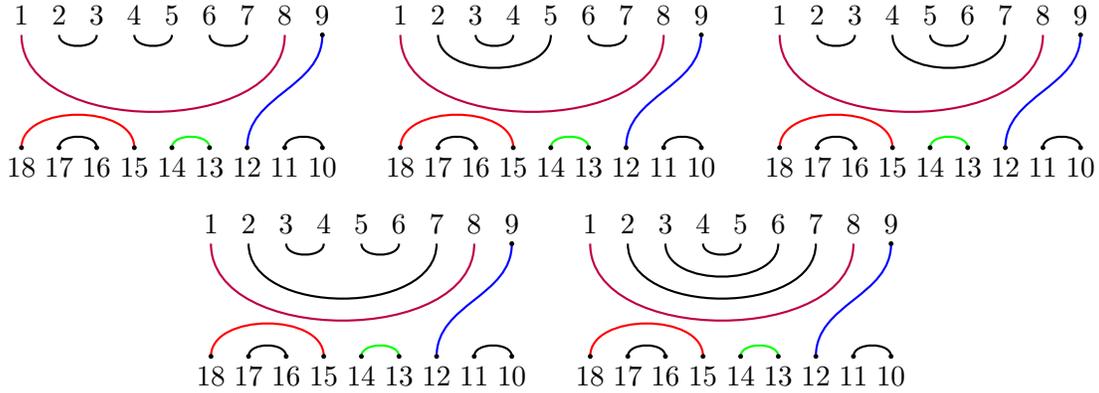
    \begin{equation*}
\begin{tikzpicture}[xscale=0.5,yscale=0.5]
\draw[thick,red] (1,0) to [out=90,in=90](4,0);
\draw[thick,black] (2,3) to [out=270,in=270](3,3);
\draw[thick,green] (5,0) to [out=90,in=90](6,0);
\draw[thick,blue] (7,0) to [out=90,in=270](9,3);
\draw[thick,purple] (1,3) to [out=270,in=270](8,3);
\draw[thick,black] (6,3) to [out=270,in=270](7,3);
\draw[thick,black] (4,3) to [out=270,in=270](5,3);
\draw[thick,black] (2,0) to [out=90,in=90](3,0);
\draw[thick,black] (8,0) to [out=90,in=90](9,0);

\draw[fill] (1,0) circle [radius=0.05];
\draw[fill] (2,0) circle [radius=0.05];
\draw[fill] (3,0) circle [radius=0.05];
\draw[fill] (4,0) circle [radius=0.05];
\draw[fill] (5,0) circle [radius=0.05];
\draw[fill] (6,0) circle [radius=0.05];
\draw[fill] (7,0) circle [radius=0.05];
\draw[fill] (8,0) circle [radius=0.05];
\draw[fill] (9,0) circle [radius=0.05];
\draw[fill] (9,3) circle [radius=0.05];
\node[above] at (1,3) {$1$};
\node[above] at (2,3) {$2$};
\node[above] at (3,3) {$3$};
\node[above] at (4,3) {$4$};
\node[above] at (5,3) {$5$};
\node[above] at (6,3) {$6$};
\node[above] at (7,3) {$7$};
\node[above] at (8,3) {$8$};
\node[above] at (9,3) {$9$};
\node[below] at (9,0) {$10$};
\node[below] at (8,0) {$11$};
\node[below] at (7,0) {$12$};
\node[below] at (6,0) {$13$};
\node[below] at (5,0) {$14$};
\node[below] at (4,0) {$15$};
\node[below] at (3,0) {$16$};
\node[below] at (2,0) {$17$};
\node[below] at (1,0) {$18$};

\end{tikzpicture} \quad \begin{tikzpicture}[xscale=0.5,yscale=0.5]
\draw[thick,red] (1,0) to [out=90,in=90](4,0);
\draw[thick,black] (2,3) to [out=270,in=270](5,3);
\draw[thick,green] (5,0) to [out=90,in=90](6,0);
\draw[thick,blue] (7,0) to [out=90,in=270](9,3);
\draw[thick,purple] (1,3) to [out=270,in=270](8,3);
\draw[thick,black] (6,3) to [out=270,in=270](7,3);
\draw[thick,black] (3,3) to [out=270,in=270](4,3);
\draw[thick,black] (2,0) to [out=90,in=90](3,0);
\draw[thick,black] (8,0) to [out=90,in=90](9,0);

\draw[fill] (1,0) circle [radius=0.05];
\draw[fill] (2,0) circle [radius=0.05];
\draw[fill] (3,0) circle [radius=0.05];
\draw[fill] (4,0) circle [radius=0.05];
\draw[fill] (5,0) circle [radius=0.05];
\draw[fill] (6,0) circle [radius=0.05];
\draw[fill] (7,0) circle [radius=0.05];
\draw[fill] (8,0) circle [radius=0.05];
\draw[fill] (9,0) circle [radius=0.05];
\draw[fill] (9,3) circle [radius=0.05];
\node[above] at (1,3) {$1$};
\node[above] at (2,3) {$2$};
\node[above] at (3,3) {$3$};
\node[above] at (4,3) {$4$};
\node[above] at (5,3) {$5$};
\node[above] at (6,3) {$6$};
\node[above] at (7,3) {$7$};
\node[above] at (8,3) {$8$};
\node[above] at (9,3) {$9$};
\node[below] at (9,0) {$10$};
\node[below] at (8,0) {$11$};
\node[below] at (7,0) {$12$};
\node[below] at (6,0) {$13$};
\node[below] at (5,0) {$14$};
\node[below] at (4,0) {$15$};
\node[below] at (3,0) {$16$};
\node[below] at (2,0) {$17$};
\node[below] at (1,0) {$18$};

\end{tikzpicture} \quad \begin{tikzpicture}[xscale=0.5,yscale=0.5]
\draw[thick,red] (1,0) to [out=90,in=90](4,0);
\draw[thick,black] (5,3) to [out=270,in=270](6,3);
\draw[thick,green] (5,0) to [out=90,in=90](6,0);
\draw[thick,blue] (7,0) to [out=90,in=270](9,3);
\draw[thick,purple] (1,3) to [out=270,in=270](8,3);
\draw[thick,black] (4,3) to [out=270,in=270](7,3);
\draw[thick,black] (2,3) to [out=270,in=270](3,3);
\draw[thick,black] (2,0) to [out=90,in=90](3,0);
\draw[thick,black] (8,0) to [out=90,in=90](9,0);

\draw[fill] (1,0) circle [radius=0.05];
\draw[fill] (2,0) circle [radius=0.05];
\draw[fill] (3,0) circle [radius=0.05];
\draw[fill] (4,0) circle [radius=0.05];
\draw[fill] (5,0) circle [radius=0.05];
\draw[fill] (6,0) circle [radius=0.05];
\draw[fill] (7,0) circle [radius=0.05];
\draw[fill] (8,0) circle [radius=0.05];
\draw[fill] (9,0) circle [radius=0.05];
\draw[fill] (9,3) circle [radius=0.05];
\node[above] at (1,3) {$1$};
\node[above] at (2,3) {$2$};
\node[above] at (3,3) {$3$};
\node[above] at (4,3) {$4$};
\node[above] at (5,3) {$5$};
\node[above] at (6,3) {$6$};
\node[above] at (7,3) {$7$};
\node[above] at (8,3) {$8$};
\node[above] at (9,3) {$9$};
\node[below] at (9,0) {$10$};
\node[below] at (8,0) {$11$};
\node[below] at (7,0) {$12$};
\node[below] at (6,0) {$13$};
\node[below] at (5,0) {$14$};
\node[below] at (4,0) {$15$};
\node[below] at (3,0) {$16$};
\node[below] at (2,0) {$17$};
\node[below] at (1,0) {$18$};

\end{tikzpicture} 
\end{equation*}
\begin{equation*}
\begin{tikzpicture}[xscale=0.5,yscale=0.5]
\draw[thick,red] (1,0) to [out=90,in=90](4,0);
\draw[thick,black] (2,3) to [out=270,in=270](7,3);
\draw[thick,green] (5,0) to [out=90,in=90](6,0);
\draw[thick,blue] (7,0) to [out=90,in=270](9,3);
\draw[thick,purple] (1,3) to [out=270,in=270](8,3);
\draw[thick,black] (5,3) to [out=270,in=270](6,3);
\draw[thick,black] (3,3) to [out=270,in=270](4,3);
\draw[thick,black] (2,0) to [out=90,in=90](3,0);
\draw[thick,black] (8,0) to [out=90,in=90](9,0);

\draw[fill] (1,0) circle [radius=0.05];
\draw[fill] (2,0) circle [radius=0.05];
\draw[fill] (3,0) circle [radius=0.05];
\draw[fill] (4,0) circle [radius=0.05];
\draw[fill] (5,0) circle [radius=0.05];
\draw[fill] (6,0) circle [radius=0.05];
\draw[fill] (7,0) circle [radius=0.05];
\draw[fill] (8,0) circle [radius=0.05];
\draw[fill] (9,0) circle [radius=0.05];
\draw[fill] (9,3) circle [radius=0.05];
\node[above] at (1,3) {$1$};
\node[above] at (2,3) {$2$};
\node[above] at (3,3) {$3$};
\node[above] at (4,3) {$4$};
\node[above] at (5,3) {$5$};
\node[above] at (6,3) {$6$};
\node[above] at (7,3) {$7$};
\node[above] at (8,3) {$8$};
\node[above] at (9,3) {$9$};
\node[below] at (9,0) {$10$};
\node[below] at (8,0) {$11$};
\node[below] at (7,0) {$12$};
\node[below] at (6,0) {$13$};
\node[below] at (5,0) {$14$};
\node[below] at (4,0) {$15$};
\node[below] at (3,0) {$16$};
\node[below] at (2,0) {$17$};
\node[below] at (1,0) {$18$};

\end{tikzpicture} \quad \begin{tikzpicture}[xscale=0.5,yscale=0.5]
\draw[thick,red] (1,0) to [out=90,in=90](4,0);
\draw[thick,black] (2,3) to [out=270,in=270](7,3);
\draw[thick,green] (5,0) to [out=90,in=90](6,0);
\draw[thick,blue] (7,0) to [out=90,in=270](9,3);
\draw[thick,purple] (1,3) to [out=270,in=270](8,3);
\draw[thick,black] (3,3) to [out=270,in=270](6,3);
\draw[thick,black] (5,3) to [out=270,in=270](4,3);
\draw[thick,black] (2,0) to [out=90,in=90](3,0);
\draw[thick,black] (8,0) to [out=90,in=90](9,0);

\draw[fill] (1,0) circle [radius=0.05];
\draw[fill] (2,0) circle [radius=0.05];
\draw[fill] (3,0) circle [radius=0.05];
\draw[fill] (4,0) circle [radius=0.05];
\draw[fill] (5,0) circle [radius=0.05];
\draw[fill] (6,0) circle [radius=0.05];
\draw[fill] (7,0) circle [radius=0.05];
\draw[fill] (8,0) circle [radius=0.05];
\draw[fill] (9,0) circle [radius=0.05];
\draw[fill] (9,3) circle [radius=0.05];
\node[above] at (1,3) {$1$};
\node[above] at (2,3) {$2$};
\node[above] at (3,3) {$3$};
\node[above] at (4,3) {$4$};
\node[above] at (5,3) {$5$};
\node[above] at (6,3) {$6$};
\node[above] at (7,3) {$7$};
\node[above] at (8,3) {$8$};
\node[above] at (9,3) {$9$};
\node[below] at (9,0) {$10$};
\node[below] at (8,0) {$11$};
\node[below] at (7,0) {$12$};
\node[below] at (6,0) {$13$};
\node[below] at (5,0) {$14$};
\node[below] at (4,0) {$15$};
\node[below] at (3,0) {$16$};
\node[below] at (2,0) {$17$};
\node[below] at (1,0) {$18$};

\end{tikzpicture} 
\end{equation*}
 Now, letting $\bi$ to run over $\calQ(9,4),$ we can count $\chi_{4}^{(9)}=512.$
\end{example}

From Lemma \ref{lemma-formula-chi-numbers} we get the following corollary.
\begin{corollary}
$\chi_{1}^{(n)}=\mathrm{C}_{n-1},\quad  \chi_{n-1}^{(n)}=n-1,\quad \chi_n^{(n)}=1$, and
$$
\chi_{k}^{(n)}=\mathrm{C}_0\chi_{k-1}^{(n-1)}+\mathrm{C}_1\chi_{k+1}^{(n)}=\chi_{k-1}^{(n-1)}+\chi_{k+1}^{(n)},\quad  \text{for $k\in [2,n-1]$.}
$$
\end{corollary}
\begin{remark}\label{Pascal2}\rm
The Corollary above  and Lemma
\ref{lemma-recur-theta-numbers} imply that $\chi_{k}^{(n)}=\theta_{k}^{(n)},$ for  $k\in[1,n-1]$ and $n\geq 3.$ That is, the number of all TL-diagrams with exactly $k$ arcs exposed to the left side is equal to the number of positive indexing matrices $C=[c_{ij}],$ such that $c_{11}=k.$ The bijection between those sets is not direct. The above recursive relations yields the following Pascal triangle.

\begin{equation}
    \begin{tikzpicture}[xscale=0.3,yscale=0.4]
\node[] at (1,4) {$\chi_k^{(n)}:$};
\node[] at (0,0) {$1$};
\node[] at (-2,-2) {$2$};
\node[] at (2,-2) {$2$};
\node[] at (-4,-4) {$5$};
\node[] at (0,-4) {$5$};
\node[] at (4,-4) {$3$};
\node[] at (-6,-6) {$14$};
\node[] at (-2,-6) {$14$};
\node[] at (2,-6) {$9$};
\node[] at (6,-6) {$4$};
\node[] at (-8,-8) {$42$};
\node[] at (-4,-8) {$42$};
\node[] at (0,-8) {$28$};
\node[] at (4,-8) {$14$};
\node[] at (8,-8) {$5$};
\node[] at (-10,-10) {$132$};
\node[] at (-6,-10) {$132$};
\node[] at (-2,-10) {$90$};
\node[] at (2,-10) {$48$};
\node[] at (6,-10) {$20$};
\node[] at (10,-10) {$6$};
\node[] at (-20,2) {$n=1$};
\node[] at (-20,0) {$n=2$};
\node[] at (-20,-2) {$n=3$};
\node[] at (-20,-4) {$n=4$};
\node[] at (-20,-6) {$n=5$};
\node[] at (-20,-8) {$n=6$};
\node[] at (-20,-10) {$n=7$};

\draw[->] (0.5,-0.5)--(1.5,-1.5);
\draw[->] (-1.5,-2.5)--(-0.5,-3.5);
\draw[->] (2.5,-2.5)--(3.5,-3.5);
\draw[->] (4.5,-4.5)--(5.5,-5.5);
\draw[->] (4.5+2,-4.5-2)--(5.5+2,-5.5-2);
\draw[->] (4.5+4,-4.5-4)--(5.5+4,-5.5-4);
\draw[->] (0.5,-4.5)--(1.5,-5.5);
\draw[->] (-3.5,-4.5)--(-2.5,-5.5);
\draw[->] (2.5,-6.5)--(3.5,-7.5);
\draw[->] (-1.5,-6.5)--(-0.5,-7.5);
\draw[->] (-5.5,-6.5)--(-4.5,-7.5);
\draw[->] (4.5,-8.5)--(5.5,-9.5);
\draw[->] (0.5,-8.5)--(1.5,-9.5);
\draw[->] (-3.5,-8.5)--(-2.5,-9.5);
\draw[->] (-7.5,-8.5)--(-6.5,-9.5);
\draw[->] (3-2,-2)--(1-2,-2);
\draw[->] (3,-4)--(1,-4);
\draw[->] (3-4,-4)--(1-4,-4);
\draw[->] (5,-6)--(3,-6);
\draw[->] (1,-6)--(-1,-6);
\draw[->] (-3,-6)--(-5,-6);
\draw[->] (7,-8)--(5,-8);
\draw[->] (3,-8)--(1,-8);
\draw[->] (-1,-8)--(-3,-8);
\draw[->] (-5,-8)--(-7,-8);
\draw[->] (9,-10)--(7,-10);
\draw[->] (5,-10)--(3,-10);
\draw[->] (1,-10)--(-1,-10);
\draw[->] (-3,-10)--(-5,-10);
\draw[->] (-7,-10)--(-9,-10);
\node[] at (1,2) {$1$};
\node[] at (3,0) {$1$};
\node[] at (5,-2) {$1$};
\node[] at (7,-4) {$1$};
\node[] at (9,-6) {$1$};
\node[] at (11,-8) {$1$};
\node[] at (13,-10) {$1$};
\draw[->] (4.5-2,0)--(3-2,0);
\draw[->] (4.5,-2)--(3,-2);
\draw[->] (6.5,-4)--(5,-4);
\draw[->] (8.5,-6)--(7,-6);
\draw[->] (10.5,-8)--(9,-8);
\draw[-> ] (12.5,-10)--(11,-10);
\end{tikzpicture} 
\end{equation}
Cf. Pascal triangle of Remark \ref{Pascal1}.
Combining with (\ref{eq-Omega-vs-theta}), we conclude that:
\begin{equation}\label{eq-Omega-vs-chi}
\Omega_{r}^{(n)}=\chi_{2r+1}^{(n+r+1)},\quad \textrm{where $r\in[1, n-1]$}.
\end{equation}
and this is another interesting equation relating sets of indexing matrices with certain TL-diagrams.
\end{remark}

\section{Connected framed blob monoids}

In this section we define another framization of the blob monoid, called the connected framed blob monoid $Bl_{d,n}^c$. We realize $Bl_{d,n}^c$ as the $d$-abacus hook monoid $\mathfrak{H}_{d,n},$ that is, the monoid generated by diagrams obtained by adding beads to the hook-diagrams defined in Subsection 2.5.3,  up to certain appropriate relations. Most of the results of this section and their proofs are analogous at the corresponding ones given in Section \ref{ssec-abacus-blob-mon}, for this reason we have decided do not write all the proofs.  At the end of this section we provide two formulas to compute the cardinality of this connected framed monoid from which we conclude that this framization of the blob monoid is not isomorphic to the one definied in Section \ref{ssec-abacus-blob-mon}.

We start with an abstract definition of $Bl_{d,n}^c.$

\begin{definition}\label{def-connect-frm-blob-mon}
   The connected framed blob monoid, denote $Bl_{d,n}^c$,   is the monoid   generated by $u_ 0,u_1,\dots ,u_{n-1}$, $ z_0,z_1,\dots,z_n$, subject to the defining relations of the blob monoid  between the $u_i$'s, the framing relations between the $z_i$'s together with the following relations:
\begin{align}\label{eq-all-rel-conn-fram-blob}
 z_iu_i = z_{i+1}u_i, &\qquad u_iz_i = u_iz_{i+1},\qquad u_i z_j =z_j u_i \quad \text{if $j\not=i,i+1$,}\\\label{fracon}
 & z_0 u_0 = u_0 z_0, \quad  u_i z_i^k u_i=u_i, \quad \text{if $i\not=0$,}
\end{align}
\end{definition}
\begin{remark}\rm  (Cf. \cite[Section 3.5]{AiJuPa2024})
Note that the abacus (resp. blob)  monoid  is  a submonoid  of $Bl_{d,n}^c$ since $z_0, z_1,\ldots , z_n$ (resp. $u_0, u_1, \ldots, u_{n}$) form a closed subsets satisfying the defining relations of $C_d^{n+1}$ (resp. $Bl_n$). Also note that relation (\ref{fracon})  says that in this case the blob generator is not an obstruction for the framing generators. This is an important difference between this monoid and the monoid $Bl_{d,n}$ of Definition \ref{def-framed-blob-mon}.
\end{remark}

An {\it abacus hook-diagram}, is a  hook-diagram, whose arcs have at most $d-1$ beads, sliding  freely on it. Given an abacus hook-diagram ${D}^{c,\circ}$, we denote by  $D^c$ its {\it base diagram},  which is obtained by erasing  the beads in ${D}^{c,\circ}.$ We  declare that two abacus hook-diagrams  are equal if their base diagrams are equal and the number of beads in each corresponding arc are the same.

\begin{definition}\label{def-abacus-hook-mon}
The abacus hook monoid $\mathfrak{H}_{d,n},$ is the monoid  formed  by
the abacus hook diagrams, on $n$ points, with the usual product by concatenation,  and  the  following rules:
\begin{enumerate}
\item[(i)] Loops, containing or not a decoration (blob and/or beads) will be erased from the diagram.
\item[(ii)] The following local relation is imposed.

\begin{equation*}
\begin{tikzpicture}[xscale=0.5,yscale=0.3]
\draw[thick] (0,0) to [out=90,in=90](1,0);
\draw[thick] (0,3) to [out=270,in=270](1,3);
\draw[thick] (0,0) to [out=90,in=270](0,3);

\draw[thick] (0,0+3) to [out=90,in=90](1,0+3);
\draw[thick] (0,3+3) to [out=270,in=270](1,3+3);
\draw[thick] (0,0+3) to [out=90,in=270](0,3+3);
\node[] at (2,3) {$=$};
\draw[thick] (3,0) to [out=90,in=90](4,0);
\draw[thick](3,0)--(3,6);
\draw[thick] (3,6) to [out=270,in=270](4,3+3);
\end{tikzpicture} {\hspace{3cm}}\begin{tikzpicture}[xscale=0.5,yscale=0.3]
\draw[thick] (0,0) to [out=90,in=90](1,0);
\draw[thick] (0,2) to [out=270,in=270](1,2);
\draw[thick] (0,0) to [out=90,in=270](0,3);
\draw[thick](1,2)--(1,4);

\draw[thick] (0,4) to [out=90,in=90](1,4);
\draw[thick] (0,3+3) to [out=270,in=270](1,3+3);
\draw[thick] (0,0+3) to [out=90,in=270](0,3+3);
\node[] at (2,3) {$=$};
\draw[thick] (3,0) to [out=90,in=90](4,0);
\draw[thick](3,0)--(3,6);
\draw[thick] (3,6) to [out=270,in=270](4,3+3);

\node[above] at (-0.5,3) {$k$};
\draw[] (0,3) circle [radius=0.25];

\node[above] at (3.5,3) {$k$};
\draw[] (3,3) circle [radius=0.25];
\end{tikzpicture} 
\end{equation*}
\item[(iii)] The number of beads on any arc is reduced modulo $d.$
\end{enumerate}
\end{definition}

For example, for $d=3$ and $n=9,$ in $\mathfrak{H}_{d,n}$ the following equation holds:
\begin{equation}\label{ex-framed-hook-diagram}
    \begin{tikzpicture}[xscale=0.5,yscale=0.6]

\draw[thick](0,0)--(0,3);
\draw[thick] (1,0) to [out=90,in=90](0,0);
\draw[thick] (4,0) to [out=90,in=90](1,0);
\draw[thick] (7,0) to [out=90,in=90](4,0);
\draw[thick] (3,3) to [out=270,in=270](0,3);

\draw[thick] (2,0) to [out=90,in=90](3,0);
\draw[thick] (5,0) to [out=90,in=90](6,0);
\draw[thick] (8,0) to [out=90,in=270](8,3);
\draw[thick] (2,3) to [out=270,in=270](1,3);
\draw[thick] (4,3) to [out=270,in=270](7,3);
\draw[thick] (5,3) to [out=270,in=270](6,3);
\draw[thick] (9,0) to [out=90,in=270](9,3);

\draw[] (0,1.5) circle [radius=0.15];
\draw[] (1.5,2.7) circle [radius=0.15];
\draw[] (2.5,0.9) circle [radius=0.15];
\draw[] (4.5,0.65) circle [radius=0.15];
\draw[] (6.5,0.65) circle [radius=0.15];
\draw[] (8,0.9) circle [radius=0.15];
\draw[] (8,1.9) circle [radius=0.15];
\node[] at (10,1.5) {$=$};
\node[below] at (0,0) {$0'$};
\node[below] at (1,0) {$1'$};
\node[below] at (2,0) {$2'$};
\node[below] at (3,0) {$3'$};
\node[below] at (4,0) {$4'$};
\node[below] at (5,0) {$5'$};
\node[below] at (6,0) {$6'$};
\node[below] at (7,0) {$7'$};
\node[below] at (8,0) {$8'$};
\node[below] at (9,0) {$9'$};
\node[above] at (0,3) {$0$};
\node[above] at (1,3) {$1$};
\node[above] at (2,3) {$2$};
\node[above] at (3,3) {$3$};
\node[above] at (4,3) {$4$};
\node[above] at (5,3) {$5$};
\node[above] at (6,3) {$6$};
\node[above] at (7,3) {$7$};
\node[above] at (8,3) {$8$};
\node[above] at (9,3) {$9$};

\end{tikzpicture} \begin{tikzpicture}[xscale=0.5,yscale=0.6]

\draw[thick](0,0)--(0,3);
\draw[thick] (1,0) to [out=90,in=90](0,0);
\draw[thick] (4,0) to [out=90,in=90](1,0);
\draw[thick] (7,0) to [out=90,in=90](4,0);
\draw[thick] (3,3) to [out=270,in=270](0,3);

\draw[thick] (2,0) to [out=90,in=90](3,0);
\draw[thick] (5,0) to [out=90,in=90](6,0);
\draw[thick] (8,0) to [out=90,in=270](8,3);
\draw[thick] (2,3) to [out=270,in=270](1,3);
\draw[thick] (4,3) to [out=270,in=270](7,3);
\draw[thick] (5,3) to [out=270,in=270](6,3);
\draw[thick] (9,0) to [out=90,in=270](9,3);

\draw[] (1.5,2.7) circle [radius=0.15];
\draw[] (1.5,2.1) circle [radius=0.15];
\draw[] (8,0.9) circle [radius=0.15];
\draw[] (8,1.9) circle [radius=0.15];
\node[below] at (0,0) {$0'$};
\node[below] at (1,0) {$1'$};
\node[below] at (2,0) {$2'$};
\node[below] at (3,0) {$3'$};
\node[below] at (4,0) {$4'$};
\node[below] at (5,0) {$5'$};
\node[below] at (6,0) {$6'$};
\node[below] at (7,0) {$7'$};
\node[below] at (8,0) {$8'$};
\node[below] at (9,0) {$9'$};
\node[above] at (0,3) {$0$};
\node[above] at (1,3) {$1$};
\node[above] at (2,3) {$2$};
\node[above] at (3,3) {$3$};
\node[above] at (4,3) {$4$};
\node[above] at (5,3) {$5$};
\node[above] at (6,3) {$6$};
\node[above] at (7,3) {$7$};
\node[above] at (8,3) {$8$};
\node[above] at (9,3) {$9$};

\end{tikzpicture} 
\end{equation}
\mbox{\,}
Define now the  elements $\mathfrak{z}_i$,    $\mathfrak{g}_0$ and $\mathfrak{u}_i$ in  $\mathfrak{H}_{d,n}$  as show  the  next figure.
\begin{figure}[h]
\begin{equation*}\label{eq-hook-generators-Frblob-mon}
    \begin{tikzpicture}[xscale=0.3,yscale=0.3]
\node[] at (-1,1.5) {$\mathfrak{z}_i=$};
\draw[thick] (1,0) to [out=90,in=270](1,3);
\draw[thick] (2,0) to [out=90,in=270](2,3);
\node[] at (3,1.5) {$\cdots$};
\draw[thick] (4,0) to [out=90,in=270](4,3);
\draw[thick] (5,0) to [out=90,in=270](5,3);
\draw[thick] (6,0) to [out=90,in=270](6,3);
\draw[thick] (7,0) to [out=90,in=270](7,3);
\node[] at (8,1.5) {$\cdots$};
\draw[thick] (9,0) to [out=90,in=270](9,3);
\node[below] at (5,0) {$i'$};
\node[below] at (1,0) {$0'$};
\node[below] at (9,0) {$n'$};
\draw[] (5,1.5) circle [radius=0.25];
\node[] at (9.5,1.4) {};
\end{tikzpicture}
\qquad
\begin{tikzpicture}[xscale=0.3,yscale=0.3]
\node[] at (-1,1.5) {$\mathfrak{g}_0=$};
\draw[thick] (1,0) to [out=90,in=270](1,3);
\draw[thick] (1,0) to [out=90,in=90](2,0);
\draw[thick] (1,3) to [out=270,in=270](2,3);
\draw[thick] (3,0) to [out=90,in=270](3,3);
\node[] at (4,1.5) {$\cdots$};
\draw[thick] (5,0) to [out=90,in=270](5,3);
\draw[thick] (6,0) to [out=90,in=270](6,3);
\draw[thick] (7,0) to [out=90,in=270](7,3);
\node[] at (8,1.5) {$\cdots$};
\draw[thick] (9,0) to [out=90,in=270](9,3);
\node[below] at (1,0) {$0'$};
\node[below] at (3.2,0) {$2'$};
\node[below] at (9,0) {$n'$};
\node[] at (9.5,1.4) {};

\end{tikzpicture}
\qquad
\begin{tikzpicture}[xscale=0.3,yscale=0.3]
\node[] at (-1,1.5) {$\mathfrak{u}_i=$};
\draw[thick] (1,0) to [out=90,in=270](1,3);
\draw[thick] (2,0) to [out=90,in=270](2,3);
\node[] at (3,1.5) {$\cdots$};
\draw[thick] (4,0) to [out=90,in=270](4,3);
\draw[thick] (5,0) to [out=90,in=90](6,0);
\draw[thick] (5,3) to [out=270,in=270](6,3);
\draw[thick] (7,0) to [out=90,in=270](7,3);
\node[] at (8,1.5) {$\cdots$};
\draw[thick] (9,0) to [out=90,in=270](9,3);
\node[below] at (5,0) {$i'$};
\node[below] at (1,0) {$0'$};
\node[below] at (9,0) {$n'$};
\node[] at (9.5,1.4) {};

\end{tikzpicture} 
\end{equation*}
\caption{The elements $\mathfrak{z}_i, \mathfrak{g}_0, \mathfrak{u}_i\in \mathfrak{H}_{d,n}$.}
\end{figure}

Analogously to Notation \ref{Not1}, we set    $\zetan^{\alphan}=\mathfrak{z}_0^{\alpha_0}\mathfrak{z}_1^{\alpha_1}\cdots \mathfrak{z}_{n}^{\alpha_n}$ for  $\alphan:=(\alpha_0,\alpha_1,\dots,\alpha_n)\in (\mathbb{Z}/d\mathbb{Z})^{n+1}$.
\begin{lemma}\label{lemma-diag-normal-form-FrHook-1}
   Each element $D^{c,\circ}\in \mathfrak{H}_{d,n}$ can be written in  the  form $\zetan^{\alphan}$ or $\zetan^{\alphan}D^{c}\zetan^{\bbeta},$ where  $D^{c}$ is  the base diagram  of $D^{c,\circ}$. In particular, $\mathfrak{H}_{d,n}$ is generated by $\mathfrak{z}_0,\mathfrak{z}_1,\dots,\mathfrak{z}_n,  \mathfrak{g}_0,\mathfrak{u}_1,\dots,\mathfrak{u}_{n-1}$.
\end{lemma}
\begin{proof}
Analogous to the  proof of the Lemma \ref{lemma-diag-normal-form-FrBlob-1}.
\end{proof}

\begin{lemma}\label{lemma-pre-presentation-frHook-digrams}
The map $z_i\mapsto \mathfrak{z}_i$, $u_0\mapsto \mathfrak{g}_0$, $u_i\mapsto \mathfrak{u}_i$, for $i\not=0$,  defines a monoid epimorphism  from $Bl_{d,n}^c$ onto $\mathfrak{H}_{d,n}.$
\end{lemma}

\begin{proof}
The proof is obtained from Lemma \ref{lemma-diag-normal-form-FrHook-1} and by  checking that the map respects
 the defining relations of $Bl_{d,n}^c$. For instance, from the fact that the beads move freely on their arcs, we obtain: $\mathfrak{g}_0\mathfrak{z}_0=\mathfrak{z}_0\mathfrak{g}_0$, that is,
 \begin{equation*}
 \begin{tikzpicture}[xscale=0.3,yscale=0.3]
\draw[thick](1,0)--(1,-1);
\draw[thick](2,0)--(2,-1);
\draw[thick](3,0)--(3,-1);
\draw[thick](5,0)--(5,-1);
\draw[thick](6,0)--(6,-1);
\draw[thick](7,0)--(7,-1);
\draw[thick](9,0)--(9,-1);
\draw[thick](1,3)--(1,4);
\draw[thick](2,3)--(2,4);
\draw[thick](3,3)--(3,4);
\draw[thick](5,3)--(5,4);
\draw[thick](6,3)--(6,4);
\draw[thick](7,3)--(7,4);
\draw[thick](9,3)--(9,4);
\draw[] (1,3.5) circle [radius=0.25];
\draw[thick] (1,0) to [out=90,in=270](1,3);
\draw[thick] (1,0) to [out=90,in=90](2,0);
\draw[thick] (1,3) to [out=270,in=270](2,3);
\draw[thick] (3,0) to [out=90,in=270](3,3);
\node[] at (4,1.5) {$\cdots$};
\draw[thick] (5,0) to [out=90,in=270](5,3);
\draw[thick] (6,0) to [out=90,in=270](6,3);
\draw[thick] (7,0) to [out=90,in=270](7,3);
\node[] at (8,1.5) {$\cdots$};
\draw[thick] (9,0) to [out=90,in=270](9,3);
\node[below] at (1,0-1) {$0'$};
\node[below] at (3.2,0-1) {$2'$};
\node[below] at (9,0-1) {$n'$};
\node[] at (10,1.4) {$=$};

\end{tikzpicture}
\begin{tikzpicture}[xscale=0.3,yscale=0.3]
\draw[thick](1,0)--(1,-1);
\draw[thick](2,0)--(2,-1);
\draw[thick](3,0)--(3,-1);
\draw[thick](5,0)--(5,-1);
\draw[thick](6,0)--(6,-1);
\draw[thick](7,0)--(7,-1);
\draw[thick](9,0)--(9,-1);
\draw[thick](1,3)--(1,4);
\draw[thick](2,3)--(2,4);
\draw[thick](3,3)--(3,4);
\draw[thick](5,3)--(5,4);
\draw[thick](6,3)--(6,4);
\draw[thick](7,3)--(7,4);
\draw[thick](9,3)--(9,4);
\draw[] (1,-0.5) circle [radius=0.25];
\draw[thick] (1,0) to [out=90,in=270](1,3);
\draw[thick] (1,0) to [out=90,in=90](2,0);
\draw[thick] (1,3) to [out=270,in=270](2,3);
\draw[thick] (3,0) to [out=90,in=270](3,3);
\node[] at (4,1.5) {$\cdots$};
\draw[thick] (5,0) to [out=90,in=270](5,3);
\draw[thick] (6,0) to [out=90,in=270](6,3);
\draw[thick] (7,0) to [out=90,in=270](7,3);
\node[] at (8,1.5) {$\cdots$};
\draw[thick] (9,0) to [out=90,in=270](9,3);
\node[below] at (1,0-1) {$0'$};
\node[below] at (3.2,0-1) {$2'$};
\node[below] at (9,0-1) {$n'$};

\end{tikzpicture} 
\end{equation*}
\end{proof}
If $D^{\circ,c}\in \mathfrak{H}_{d,n},$ then its base diagram $D^{c}$ can be seen as an element of the hook blob monoid $\mathfrak{H}_{n}$, and therefore it can be written in the form $\underline{U}(A),$ where $A$ is an indexing matrix (see Subsection \ref{ssec-Hook-blob}).

\begin{lemma}\label{lemma-normal-form-FrHook-1}
\begin{enumerate} For  $D^{\circ,c}\in \mathfrak{H}_{d,n} $, we have:
    \item[(i)] $D^{\circ,c} =\zetan^{\alphan}=\zetan^{\alphan}\underline{U}\Big[\!\!\begin{array}{c}\infty\\\infty\end{array}\!\!\Big]$ or  $D^{\circ,c}=\zetan^{\alphan}\underline{U}(A)\zetan^{\bbeta}$, where  $\alphan,\bbeta$ are exponent vectors,  $A$ is a non-degenerate indexing matrix.
    \item[(ii)] If $\zetan^{\alphan}\underline{U}(A)\zetan^{\bbeta}=\zetan^{\bgamma}\underline{U}(A')\zetan^{\kappan}$, with  $\alphan,\bbeta,\bgamma,\kappan$  are exponent vectors,  and  $A,A'$ are indexing matrices, then  $A=A'.$
    \end{enumerate}
\end{lemma}

\begin{proof}
Analogous to  Lemma \ref{lemma-normal-form-FrBlob-1}
\end{proof}

\begin{lemma}\label{conn-zU} Set $A =\Big[\!\!\begin{array}{c}i\\j\end{array}\!\!\Big]$, where $i,j\in [1,n-1]$ and $j<i.$  Then:
\begin{enumerate}
\item[(i)]  $\mathfrak{z}_k\ulU(A) =  \ulU(A)\mathfrak{z}_k$ if and only if $0\leq k<j$ or  $k>i+1.$
\item[(ii)] $\mathfrak{z}_k\ulU(A) =  \ulU(A)\mathfrak{z}_{k+2}$ if and only if $k\in [j,i-1]$.
\item[(iii)] $\mathfrak{z}_k\ulU(A)=\mathfrak{z}_p\ulU(A)$ with $k\neq p$ if and only if
$k,p\in\{i,i+1\}$.
\item[(iv)] $\ulU(A)\mathfrak{z}_k=\ulU(A)\mathfrak{z}_p$ with $k\neq p$ if and only if
$k,p\in\{j,j+1\}$.
\end{enumerate}
\end{lemma}

\begin{proof}
    Analogous to Lemma \ref{zU}.
\end{proof}

\begin{lemma}\label{conn-zA0U}
Set $A =\Big[\!\!\begin{array}{c}i\\0\end{array}\!\!\Big]$, where $i\in [1,n-1]$.  Then the following relations hold in $\mathfrak{H}_{d,n} $.
\begin{enumerate}
\item[(i)] $\mathfrak{z}_k\underline{U}(A)=  \underline{U}(A)\mathfrak{z}_k $ if and only if $k>i+1$ or $k=0.$
\item[(ii)] $\mathfrak{z}_k\underline{U}(A)=  \underline{U}(A)\mathfrak{z}_{k+2} $ if and only if $k\in[1,i-1]$.
\item[(iii)] $\mathfrak{z}_k\underline{U}(A)=\mathfrak{z}_p\underline{U}(A),$ with $k\neq p$ if and only if $k,p\in\{i,i+1\}.$
\item[(iv)] $\underline{U}(A)\mathfrak{z}_k=\underline{U}(A)\mathfrak{z}_p,$ with $k\neq p$ if and only if $k,p\in \{0,1,2\}$.
\end{enumerate}
\end{lemma}

\begin{proof}
    Analogous to Lemma \ref{zA0U}.
\end{proof}

\begin{proposition}\label{prop-interaction-z-BBU-conn}
If $j\in [1,n]$ and $A=[B|C]$ is a non-degenerate indexing matrix, then in the product $\mathfrak{z}_j\underline{U}(A)$ one, and only one, of the following cases occurs, depending on the product $\mathfrak{z}_j\underline{U}(A)$:
    \begin{enumerate}
          \item[(i)] \label{sub-critical-case1-coro-interaction-z-BBU-conn} $j=1$ and $b_1=0.$ In this case we have $\mathfrak{z}_j\underline{U}(A)=\mathfrak{z}_0\underline{U}(A)=\underline{U}(A)\mathfrak{z}_0.$
         \item[(ii)] \label{critical-case1-coro-interaction-z-BBU-conn}  There is  $s \in [1,r]$ and  $k\in \{1,2\}$ such that:
        \begin{equation*}
        \mathfrak{z}_j\underline{U}(A)=\underline{U}\left[\begin{matrix}
                b_1\\
                0
            \end{matrix}\right]\cdots \underline{U}\left[\begin{matrix}
                b_s\\
                0
    \end{matrix}\right]\mathfrak{z}_k\underline{U}\left[\begin{matrix}
         b_{s+1}\\
                0
            \end{matrix}\right]\cdots \underline{U}\left[\begin{matrix}
                b_r\\
                0
            \end{matrix}\right]\underline{U}(C).
        \end{equation*}
        In this case we have  $\mathfrak{z}_j\underline{U}(A)=\mathfrak{z}_0\underline{U}(A)=\underline{U}(A)\mathfrak{z}_0.$
          \item[(iii)]
          There is $i\in[1,n] $ s.t.  $i\neq j$ and $\mathfrak{z}_j\underline{U}(A)=\mathfrak{z}_i\underline{U}(A).$
         \item[(iv)] 
         There is $k\in[1, n]$ s.t.  $\mathfrak{z}_j\underline{U}(A)=\underline{U}(A)\mathfrak{z}_k.$
    \end{enumerate}
Furthermore, in  the product $\underline{U}(A)\mathfrak{z}_j$  one, and only one, of the following situations occurs, depending on the product $\underline{U}(A)\mathfrak{z}_j$ in the monoid $B_{d,n}^c.$
    \begin{enumerate}
        \item[(v)] \label{critical-case2-coro-interaction-z-BBU-conn}  There is $s\in[1,r]$ and $k\in \{1,2\}$ such that:
        \begin{equation}
        \underline{U}(A)\mathfrak{z}_j=\underline{U}\left[\begin{matrix}
                b_{1}\\
                0
    \end{matrix}\right]\cdots \underline{U}\left[\begin{matrix}
                b_s\\
                0
    \end{matrix}\right]\mathfrak{z}_k\underline{U}\left[\begin{matrix}
                b_{s+1}\\
                0
    \end{matrix}\right]\cdots\underline{U}\left[\begin{matrix}
                b_{r}\\
                0
    \end{matrix}\right]\underline{U}(C).
        \end{equation}
        In this case we have: $\underline{U}(A)\mathfrak{z}_j=\mathfrak{z}_0\underline{U}(A)=\underline{U}(A)\mathfrak{z}_0.$
          \item[(vi)] \label{Not-critical-case2-coro-interaction-z-BBU-conn} There is $i\in[1,n]$ such that $i\neq j$ and $\underline{U}(A)\mathfrak{z}_j=\underline{U}(A)\mathfrak{z}_i.$
         \item[(vii)] \label{Not2-critical-case2-coro-interaction-z-BBU-conn} There is  $k\in[1,n]$ such that $\underline{U}(A)\mathfrak{z}_j=\mathfrak{z}_k\underline{U}(A).$
    \end{enumerate}
\end{proposition}

\begin{proof}
    Analogous to Proposition \ref{prop-interaction-z-BBU}
\end{proof}

\begin{example}(Cf.  Example \ref{ex-left-right-Z-matrices-V0})\label{ex-left-right-Z-matrices-V0-conn}
Let $n=7$ and consider the indexing matrix
    \begin{equation*}
        A=\left[\begin{matrix}
            1 & 2 & 3 & 6 \\
            0 & 1 & 2 & 4
        \end{matrix}\right].
    \end{equation*}
Then we have
    
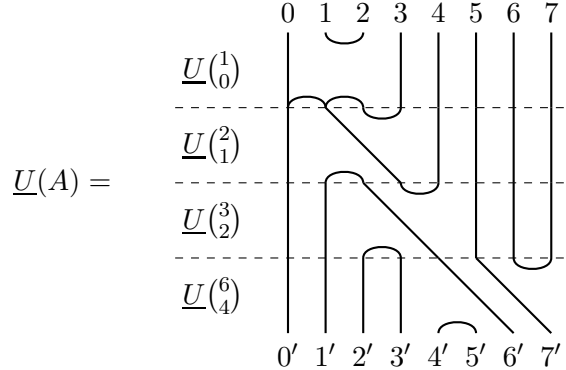
\begin{equation*}
       \begin{tikzpicture}[xscale=0.5,yscale=0.5]
\node[] at (-6,0) {$\underline{U}(A)=$};
\node[above] at (0,4) {$0$};
\node[above] at (1,4) {$1$};
\node[above] at (2,4) {$2$};
\node[above] at (3,4) {$3$};
\node[above] at (4,4) {$4$};
\node[above] at (5,4) {$5$};
\node[above] at (6,4) {$6$};
\node[above] at (7,4) {$7$};

\node[] at (-2,3) {$\underline{U}\binom{1}{0}$};
\draw[thick](0,4)--(0,2);
\draw[thick] (1,4) to [out=270,in=270](2,4);
\draw[thick] (1,2) to [out=90,in=90](2,2);
\draw[thick] (0,2) to [out=90,in=90](1,2);
\draw[thick] (3,4) to (3,2);
\draw[thick] (4,4) to (4,2);
\draw[thick] (5,4) to (5,2);
\draw[thick] (6,4) to (6,2);
\draw[thick] (7,4) to (7,2);
\draw[dashed] (-3,2) to (7.5,2);
\node[] at (-2,1) {$\underline{U}\binom{2}{1}$};
\draw[thick](0,2)--(0,0);
\draw[thick] (2,2) to [out=270,in=270](3,2);
\draw[thick] (1,0) to [out=90,in=90](2,0);
\draw[thick] (1,2) to (3,0);
\draw[thick] (4,2) to (4,0);
\draw[thick] (5,2) to (5,0);
\draw[thick] (6,2) to (6,0);
\draw[thick] (7,2) to (7,0);
\draw[dashed] (-3,0) to (7.5,0);
\node[] at (-2,-1) {$\underline{U}\binom{3}{2}$};
\draw[thick](0,0)--(0,-2);
\draw[thick] (3,0) to [out=270,in=270](4,0);
\draw[thick] (2,-2) to [out=90,in=90](3,-2);
\draw[thick] (1,-2) to (1,0);
\draw[thick] (2,0) to (4,-2);
\draw[thick] (5,-2) to (5,0);
\draw[thick] (6,-2) to (6,0);
\draw[thick] (7,-2) to (7,0);
\draw[dashed] (-3,-2) to (7.5,-2);

\node[] at (-2,-3) {$\underline{U}\binom{6}{4}$};
\draw[thick](0,-2)--(0,-4);
\draw[thick] (6,-2) to [out=270,in=270](7,-2);
\draw[thick] (4,-4) to [out=90,in=90](5,-4);
\draw[thick] (1,-2) to (1,-4);
\draw[thick] (2,-2) to (2,-4);
\draw[thick] (3,-2) to (3,-4);
\draw[thick] (4,-2) to (6,-4);
\draw[thick] (5,-2) to (7,-4);

\node[below] at (0,-4) {$0'$};
\node[below] at (1,-4) {$1'$};
\node[below] at (2,-4) {$2'$};
\node[below] at (3,-4) {$3'$};
\node[below] at (4,-4) {$4'$};
\node[below] at (5,-4) {$5'$};
\node[below] at (6,-4) {$6'$};
\node[below] at (7,-4) {$7'$};
\end{tikzpicture} 
    \end{equation*}
    We can easily check that:
    \begin{equation*}
        \begin{array}{ccc}
            \mathfrak{z}_1\underline{U}(A)=\mathfrak{z}_2\underline{U}(A), &  \mathfrak{z}_6\underline{U}(A)=\mathfrak{z}_7\underline{U}(A),& \mathfrak{z}_5\underline{U}(A)=\underline{U}(A)\mathfrak{z}_7, \\
            \quad &\quad& \quad \\
            \underline{U}(A)\mathfrak{z}_1=\underline{U}(A)\mathfrak{z}_6, & \underline{U}(A)\mathfrak{z}_2=\underline{U}(A)\mathfrak{z}_3, & \underline{U}(A)\mathfrak{z}_4=\underline{U}(A)\mathfrak{z}_5,\\
            \quad & \quad &\quad \\
            \mathfrak{z}_0\underline{U}(A)=\mathfrak{z}_3\underline{U}(A)=& \mathfrak{z}_4\underline{U}(A)=\underline{U}(A)\mathfrak{z}_0.
        \end{array}
    \end{equation*}

\end{example}

The following definition is analogous to Definition \ref{def-left-right-Z-matrices}, we just make some convenient changes according to the behavior of the products $\mathfrak{z}_j\underline{U}(A)$ and $\underline{U}(A)\mathfrak{z}_j$ in $\mathfrak{H}_{d,n}$ instead of $\mathfrak{z}_j\BBU(A)$ and $\BBU(A)\mathfrak{z}_j$, in $\mathfrak{Bl}_{d,n}.$

\begin{definition}
Let $A=[B|C]$ be a non degenerate indexing matrix, we define the  matrices  $L'(A)$ and $R'(A)$ of size $3\times n$, as follows.
\begin{enumerate}
\item[(i)]Depending on the product $\mathfrak{z}_j\underline{U}(A),$ the $j$-th column of $L'(A)$ is equal to:
\begin{enumerate}
\item $\big[-1\,\,\, j \,\,\, 0 \big]^t$ if we are in case (i) or (ii) of Proposition \ref{prop-interaction-z-BBU-conn}.
\item $\big[ i\,\,\, j\,\,\, 0\big]^t$ if we are in case (iii) of Proposition \ref{prop-interaction-z-BBU-conn}.
\item $\big[ j\,\,\,j \,\,\, k \big]^t$ if we are in case (iv) of Proposition \ref{prop-interaction-z-BBU-conn}.
\end{enumerate}
\item [(ii)]  Depending on the product $\underline{U}(A)\mathfrak{z}_j,$ the $j$-th column of $R'(A)$ is equal to:
\begin{enumerate}
\item  $\left[0\,\, j \,\, -1\right]^t$ if we are in case (v) of Proposition \ref{prop-interaction-z-BBU-conn}.
\item $\left[0\,\, j \,\, i\right]^t$ if we are in case (vi) of Proposition \ref{prop-interaction-z-BBU-conn}.
\item $\left[k\,\, j \,\, 0\right]^t$ if we are in case (vii) of Proposition \ref{prop-interaction-z-BBU-conn}.
\end{enumerate}
\end{enumerate}
\end{definition}
\begin{definition}\label{def-notation-corchetes}
  Let $A=[B|C]$ be a non degenerate indexing matrix $A=[B|C]$. Put  $L'(A)=[l_{ij}]$ and $R'(A)=[r_{ij}]$. We set:
$$
L'_0(A):=\{\min\{l_{1j},l_{2j}\}\,;\, j\in[1,n]\}\setminus \{-1\},
$$
$$R'_0(A):=\{0\}\cup(\{
\min\{r_{2j},r_{3j}\}\,;\, j\in[1,n]\}\setminus\{-1\}),
$$
and we define the subsets $\mathrm{Exp}_{L'}(A)$ and $\mathrm{Exp}_{R'}(A)$ of $(\mathbb{Z}/d\mathbb{Z})^{n+1}$ by:
\begin{eqnarray*}
\mathrm{Exp}_{L'}(A) &=&\{(\alpha_0,\alpha_1, \ldots, \alpha_n)\,;\,\alpha_i =0\quad  \text{if}\quad i\not\in L'_0(A) \},\\
\mathrm{Exp}_{R'}(A) &=&\{(\alpha_0,\alpha_1, \ldots, \alpha_n)\,;\,\alpha_i =0\quad  \text{if}\quad i\not\in R'_0(A) \}.
\end{eqnarray*}
Finally, for $\alphan :=(\alpha_0,\alpha_1,\ldots, \alpha_n)\in L'_0(A)$ (resp. $R'_0(A)$) we set
$$
[ \alphan ,A ] =\prod_{j\in L'_0(A)}\mathfrak{z}_j^{\alpha_j} \quad \text{(resp. $[ A, \alphan ] =\prod_{j\in R'_0(A)}\mathfrak{z}_j^{\alpha_j}$).}
$$
\end{definition}
\begin{example}(Cf. with Example \ref{ex-left-right-Z-matrices-2})\label{ex-left-right-Z-matrices-conn}
Let $n=7$ and consider the indexing matrix:
    \begin{equation*}
        A=\left[\begin{matrix}
            1 & 2 & 3 & 6 \\
            0 & 1 & 2 & 4
        \end{matrix}\right].
    \end{equation*}
    Then, according  to Example \ref{ex-left-right-Z-matrices-V0-conn}, we have:
    \begin{equation*}
        L'(A)=\left[\begin{matrix}
            2 & 1 & -1 & -1 & 5 & 7 & 6 \\
            1 & 2 & 3 & 4 & 5 & 6 & 7 \\
            0 & 0 & 0 & 0 & 7 & 0 & 0
        \end{matrix}\right],\quad\text{and}\quad  R'(A)=\left[\begin{matrix}
            0 & 0 & 0 & 0 & 0 & 0 & 5 \\
            1 & 2 & 3 & 4 & 5 & 6 & 7 \\
            6 & 3 & 2 & 5 & 4 & 1 & 0
        \end{matrix}\right].
    \end{equation*}
Then:
    \begin{equation*}
        L'_0(A)=\{1,5,6\},\quad R'_0(A)=\{0,1,2,4\}.
    \end{equation*}
For example if we take  $\alphan=(0,2,0,0,0,1,2,0)\in\mathrm{Exp}_{L'}(A) $ and $\bbeta=(2,1,2,0,1,0,0,0)\in\mathrm{Exp}_{R'}(A)$  and
    \begin{equation*}
        [ \alphan, A ] =\mathfrak{z}_1^{2}\mathfrak{z}_5{z}_6^{2},\quad [  A,\bbeta ]=\mathfrak{z}_0^2\mathfrak{z}_1\mathfrak{z}_2^{2}\mathfrak{z}_4.
    \end{equation*}
    Note that
    \begin{equation*}
        |\mathrm{Exp}_{L'}(A)|=3^3,\quad\text{and}\quad |\mathrm{Exp}_{R'}(A)|=3^4.    \end{equation*}
\end{example}

 \begin{corollary}(Cf.  Corollary \ref{coro-normal-form-framed-blob-mon})\label{coro-normal-form-framed-blob-mon-conn}
 Every element of $ \mathfrak{H}_{d,n} $ can be written in the  normal form
    $[ \alphan,A ]\underline{U}(A)[ A, \bbeta ]$,
    where $A$ is an indexing matrix,  $\alphan\in\mathrm{Exp}_{L'}(A), \bbeta\in \mathrm{Exp}_{R'}(A)$.
\end{corollary}

\begin{example}
    Let ${D}^{c,\circ}$ be the diagram of (\ref{ex-framed-hook-diagram}), then its normal form is given by
    \begin{equation*}
       {D}^{c,\circ}=\mathfrak{z}_1\mathfrak{z}_8^2\left(\underline{U}\left[\begin{matrix}
           1&2&5&6\\
           0&0&2&5
       \end{matrix}\right]\right)\mathfrak{z}_0.
    \end{equation*}
\end{example}

\begin{lemma}(Cf. Proposition \ref{lemma-equation-n+r})\label{lemma-equation-n+r-conn}
 Given a non degenerated indexing matrix $A=[B|C]$ with blob-rank $r\geq 0,$ then:
    \begin{align}
  \label{eq1-lemma-equation-n+r}
        |L'_0(A)|+|R'_0(A)|= n-r+1,
        \\ \label{eq2-lemma-equation-n+r}
        |\mathrm{Exp}_{L'}(A)
        |\cdot|\mathrm{Exp}_{R'}(A)|= d^{n-r+1}.
    \end{align}
\end{lemma}

The following two theorems are the analogous versions of Theorem \ref{theo-card-Fblob-mon-first} and Theorem \ref{theo-alternative-formula-diag-frBlob}, respectively.
\begin{theorem}\label{theo-card-Fblob-mon-first-conn}
    The cardinality of the abacus hook monoid  is given by:
    \begin{equation}
    |\mathfrak{H}_{d,n} |= d^{n+1}\sum_{k=0}^{n}\Omega_k^{(n)}d^{-k}.
    \end{equation}

\end{theorem}
\begin{proof}
 Denote $O_k$ the set of all the indexing matrices with blob-rank equal to $k,$ then Corollary \ref{coro-normal-form-framed-blob-mon-conn} implies that $\mathfrak{H}_{d,n}$ can be decomposed, as set,  in the following disjoint union:
    \begin{equation*}
        \mathfrak{H}_{d,n}=\bigcup_{k=0}^{n}\left\{[ \alphan, A]\underline{U}(A)[ A,{\bbeta}]: A\in O_k, \alphan\in \textrm{Exp}_{L'}(A),\bbeta\in\textrm{Exp}_{R'}(A)
        \right\}.
    \end{equation*} From Proposition \ref{lemma-equation-n+r-conn} we obtain
    \begin{equation*}
|\mathfrak{H}_{d,n}|=\sum_{k=0}^{n}\Omega_{k}^{(n)}|\textrm{Exp}_{L'}(A)||\textrm{Exp}_{R'}(A)|=\sum_{k=0}^{n}\Omega_{k}^{(n)}d^{n-k+1}.
    \end{equation*}
Hence, the proof is finished.
\end{proof}

\begin{theorem}\label{theo-alternative-formula-diag-frBlob-conn}
The cardilnality of $\mathfrak{H}_{d,n},$ is given by
    \begin{equation}\label{EQ-lemma-lower-bound-for-diag-frBlob-conn}
       |\mathfrak{H}_{d,n}|=  d^{n+1}\sum_{k^=1}^{n}\chi_{k}^{(n)}d^{-k}(1+d)^{k}.
    \end{equation}
\end{theorem}
\begin{proof}
    Let $ k\in [1,n] $. Recall that $\chi_{k}^{(n)}$ is the number of TL-diagrams at $n$ points, with exactly $k$ arcs exposed on the left, see Definition \ref{def-chi-numbers}. Then, for each $j\in [0,k]$ there are  $\chi_{k}^{(n)}\binom{k}{j}$ different TL-diagrams with $k$ arcs exposed to the left and with $j$ of them decorated by a blob.
    Now, by the  Theorem \ref{theo-iso-algblobmon-hookblobmon}, we know that  all blobbed arcs of a blobbed TL-diagram become a single arc, connected with the points $0$ and $0'$ in the corresponding hook-diagram.
    Further, each arc in a hook-diagram can contain at most $d-1$ beads.
Then, proceeding as in the proof of Theorem \ref{theo-alternative-formula-diag-frBlob}, the cardinality of $\mathfrak{H}_{d,n}$ can be  calculated as follows:
\begin{equation}
        \sum_{k=1}^{n}\chi_{k}^{(n)}\sum_{j=0}^{k}\binom{k}{j}d^{n-j+1}=\sum_{k=1}^{n}d^{n-k+1}\chi_{k}^{(n)}\sum_{j=0}^{k}\binom{k}{j}d^{k-j}=d^{n+1}\sum_{k=1}^{n}d^{-k}\chi_{k}^{(n)}(1+d)^k.
    \end{equation}
Thus the proof is concluded.
\end{proof}

\begin{corollary}\label{theo-presentation-for-fram-conn}
    The epimorphism $\psi$  of Lemma \ref{lemma-pre-presentation-frHook-digrams} is  an isomorphism. In particular, the defining presentation of  $B_{d,n}^c$ is a presentation of $\mathfrak{H}_{d,n}$.
\end{corollary}
\begin{proof}The proof follows the same strategy used to prove the Theorem \ref{theo-iso-AlgFrblob-DiagFrblob}. Thus we  define, as in Definition \ref{def-basis-matrix-blob}, the corresponding elements ${U}(A)\in B_{d,n}^c$ for each indexing matrix $A$, and then
imitating the  Definition \ref{def-notation-corchetes} we arrive finding a normal forms $[\alphan,A]{U}(A)[A,\bbeta]$ for the elements of $B_{d,n}^c$, where now  $\alphan\in \textrm{Exp}'_{L}(A)$ and  $\bbeta\in \textrm{Exp}'_{R}(A).$ We therefore conclude that $|B_{d,n}^c|\leq |\mathfrak{H}_{d,n}|.$ From which the proof follows.
\end{proof}

We close the paper with the following remark:

\begin{remark}
    Note that for $d\geq 2,$ Theorems \ref{theo-card-Fblob-mon-first} and \ref{theo-card-Fblob-mon-first-conn} show that the cardinalities of the monoids $\mathfrak{Bl}_{d,n}$ and $\mathfrak{H}_{d,n}$ are not the same. Therefore we conclude that $\mathfrak{Bl}_{d,n}$ and $\mathfrak{H}_{d,n}$ are two non isomorphic framizations of the blob monoid.
\end{remark}

\end{document}